\documentclass[a4paper, 11pt]{amsart}
\usepackage{amsmath, amssymb, amsfonts, amsthm, enumerate, mathtools, tikz-cd, hyperref, stmaryrd}
\usepackage[margin=1.2in]{geometry}

\newtheorem{thm}{Theorem}[section]
\newtheorem{lem}[thm]{Lemma}
\newtheorem{defi}[thm]{Definition}
\newtheorem{rem}[thm]{Remark}

\newtheorem{prop}[thm]{Proposition}
\newtheorem{cor}[thm]{Corollary}
\newtheorem{conjecture}[thm]{Conjecture}
\newtheorem{thmA}{Theorem}

\newcommand{\Mod}[1]{\ (\mathrm{mod}\ #1)}
\newtheorem{setup}[thm]{Setup}

\def\QQ{\mathbb{Q}}
\def\ZZ{\mathbb{Z}}
\def\FF{\mathbb{F}}
\def\TT{\mathbb{T}}
\def\frob{\text{Frob}}
\def\chibar{\bar\chi}

\def\ad{\text{Ad}}
\newcommand{\GL}{\mathrm{GL}}

\def\Ga1{\Gamma_1}
\def\rhob{\bar\rho}

\def\tr{\operatorname{tr}}
\def\red{\text{red}}
\def\univ{\text{univ}}

\def\ps{\text{pd}}
\def\spec{\text{Spec}}

\DeclareFontFamily{U}{wncy}{}
\DeclareFontShape{U}{wncy}{m}{n}{<->wncyr10}{}
 \DeclareSymbolFont{mcy}{U}{wncy}{m}{n}
 \DeclareMathSymbol{\Sh}{\mathord}{mcy}{"58} 

\begin{document}
\baselineskip 18pt

\title{On density of modular points in pseudo-deformation rings}
\author{Shaunak V. Deo} 
\email{shaunakdeo@iisc.ac.in}
\address{Department of Mathematics, Indian Institute of Science, Bangalore 560012, India}
\date{}

\subjclass[2010]{11F80(primary); 11F25, 11F33(secondary)}
\keywords{pseudo-representations; deformation rings of Galois representations; Hecke algebras; level raising}

\begin{abstract}
Given a continuous, odd, reducible and semi-simple $2$-dimensional representation $\bar\rho_0$ of $G_{\QQ,Np}$ over a finite field of odd characteristic $p$, we study the relation between the universal deformation ring of the pseudo-representation corresponding to $\bar\rho_0$ (pseudo-deformation ring) and the big $p$-adic Hecke algebra to prove that the maximal reduced quotient of the pseudo-deformation ring is isomorphic to the local component of the big $p$-adic Hecke algebra corresponding to $\bar\rho_0$ if a certain global Galois cohomology group has dimension $1$. This partially extends the results of B\"{o}ckle to the case of residually reducible representations.
%On our way to proving the main theorem, we prove that the universal deformation ring of a certain class of non-split reducible representations are reduced local complete intersection rings of Krull dimension $4$.
We give an application of our main theorem to the structure of Hecke algebras modulo $p$.
As another application of our methods and results, we prove a result about non-optimal levels of newforms lifting $\bar\rho_0$ in the spirit of Diamond--Taylor.
This also gives a partial answer to a conjecture of Billerey--Menares.
\end{abstract}

\maketitle

\section{Introduction}

Given an odd, continuous representation $\rhob_0 : \text{Gal}(\overline{\QQ}/\QQ) \to \GL_2(\FF)$, where $\FF$ is a finite field, it is natural to ask whether we can determine all its lifts to characteristic $0$ (if any) and whether all these lifts have some arithmetic significance. More concretely, if $p$ is the characteristic of $\FF$ and $\mathcal{O}$ is the ring of integers of a finite extension of $\QQ_p$ with residue field $\FF$ and uniformizer $\varpi$, then one would like to describe all representations $\rho : \text{Gal}(\overline{\QQ}/\QQ) \to \GL_2(\mathcal{O})$ such that $\rho \Mod{\varpi} = \rhob_0$ and determine whether all such $\rho$ arise from arithmetic objects (such as modular forms). These questions and their variants have been studied extensively for past three decades or so and they have beome one of the central topics in number theory.

In his seminal paper, Mazur (\cite{M}) developed deformation theory of Galois representations to study this question in more generality. In particular, 
%he studied lifts of $\rhob_0$ to all complete Noetherian rings with residue field $\FF$ and 
%(i.e. quotients of power series rings over the Witt vectors of $\FF$) 
 after restricting ramification, he proved the existence of the universal deformation ring of $\rhob_0$ which interpolates all lifts of $\rhob_0$ (up to equivalence) to complete local Noetherian rings with residue field $\FF$ when $\rhob_0$ is absolutely irreducible. Let us denote this ring by $R_{\rhob_0}$.
In the case when $\rhob_0$ arises form an ordinary modular form, Mazur also proved that (under certain hypotheses) the ordinary quotient of $R_{\rhob_0}$ is isomorphic to the Hecke algebra acting on a space of ordinary modular forms of fixed level but arbitrary weights. 
%This roughly translates to the statement that, under Mazur's hypotheses, an ordinary lift of $\rhob_0$to characteristic $0$ unramified outside a specific set of primes arises from an ordinary modular form.

So it is natural to ask whether one can relate $R_{\rhob_0}$ itself to an arithmetic object. In view of Mazur's theorem stated above, the Hecke algebra acting on a space of modular forms of fixed level but arbitrary weights (i.e. a `big' $p$-adic Hecke algebra) seems to be a suitable candidate to consider for this purpose.

\subsection{History}
Following this line of thought, Gouv\^{e}a conjectured, in his thesis (\cite{G}) and in subsequent work (\cite{G2}), that $R_{\rhob_0}$ is isomorphic to a suitable big $p$-adic Hecke algebra. This conjecture was further explored by Gouv\^{e}a and Mazur in \cite{GM}. They proved that $R_{\rhob_0}$ is isomorphic to a big $p$-adic Hecke algebra when $\rhob_0$ is absolutely irreducible and unobstructed (i.e. $R_{\rhob_0}$ is a power series ring of Krull dimension $4$) by constructing an \emph{infinite fern} in the universal deformation space of $\rhob_0$ using work of Coleman. Such theorems have come to known as `big' $R=\TT$ theorems in the field. Roughly speaking, big $R=\TT$ theorems imply that the set of points in the universal deformation space of $\rhob_0$ corresponding to ($p$-adic Galois representations attached to) classical modular eigenforms is Zariski dense.

Their results were extended by B\"{o}ckle (\cite{Bo}) to a wide range of representations. In particular, B\"{o}ckle (\cite{Bo}) proved that, under some mild hypothesis (popularly known as Taylor--Wiles hypotheses), if $\rhob_0$ is absolutely irreducible and it arises from modular forms, then $R_{\rhob_0}$ is isomorphic to a big $p$-adic Hecke algebra. In \cite{A}, Allen extended the results and methods of B\"{o}ckle to prove big $R = \TT$ theorems in the context of polarized Galois representations for CM fields. He also proved big $R=\TT$ theorems for Hilbert modular forms under some hypothesis. 
The results of Allen (\cite{A}) in the case of polarized Galois representations for CM fields have been extended by Hellmann--Margerin--Schraen (\cite{HMS}) by generalizing the infinite fern argument of Gouv\^{e}a--Mazur and Chenevier. %under standard automorphy lifting conjectures.

We refer the reader to the excellent introductions of both \cite{Bo} and \cite{A} to know more about the history of big $R=\TT$ conjectures, their arithmetic consequences and a nice summary of the techniques used by them to prove big $R=\TT$ theorems. 
For a nice summary of various important properties of the big $p$-adic Hecke algebras studied in this article (including the Gouv\^{e}a-Mazur infinite fern), we refer the reader to a nicely written article of Emerton (\cite[Section 2]{Em}).
Note that in \cite{Bo}, \cite{A} and \cite{HMS}, $\rhob_0$ is assumed to be absolutely irreducible.

\subsection{Aim and Setup}\label{aimsec}
The main aim of this article is to explore big $R=\TT$ theorems for reducible $\rhob_0 : \text{Gal}(\overline{\QQ}/\QQ) \to \GL_2(\FF)$. Note that as $\rhob_0$ is reducible, the universal deformation ring of $\rhob_0$ (in the sense of Mazur) may not always exist. So the appropriate ring to consider here is the \emph{universal pseudo-deformation ring} of $\rhob_0$ i.e. the universal deformation ring of the $2$-dimensional \emph{pseudo-representation} (in the sense of Chenevier) $(\tr(\rhob_0),\det(\rhob_0))$ of $G_{\QQ,Np}$ attached to $\rhob_0$. 
The pseudo-deformation ring of $\rhob_0$ depends only on its semi-simplification.
So we assume that $\rhob_0$ is split reducible.

Note that, in \cite{C}, Chenevier uses the term \emph{determinant} instead of pseudo-representation and he formulates this notion using polynomial laws (see \cite[Section 1.5]{C}).
On the other hand, a $2$-dimensional pseudo-representation is a tuple of functions which `behaves' like the trace and determinant of a $2$-dimensional representation (see Definition~\ref{pseudodef} for its precise definition).
Moreover, in the case of dimension $2$, the theory of Chenevier determinants reduces to the theory of pseudo-representations (see \cite[Lemma 1.9]{C}).
In view of this equivalence, we work with pseudo-representations rather than the polynomial laws that Chenevier uses.
%Chenevier establishes a bijection between his determinants and tuples of functions (which we call pseudo-representations) `behaving' like the traces and determinants of $2$-dimensional representations (\cite[Lemma 1.9]{C}).

%We use this equivalence to study the 

%Note that when $\rhob_0$ is absolutely irreducible, this ring is same as the universal deformation ring of $\rhob_0$ (in the sense of Mazur).
Before stating our main result, let us first set up some basic notation and define the objects that we will be studying.

Let $p$ be an odd prime and $N \geq 1$ be an integer not divisible by $p$. Let $G_{\QQ,Np}$ be the Galois group of the maximal extension of $\QQ$ unramified outside primes dividing $Np$ and $\infty$ over $\QQ$.
Let $\FF$ be a finite field of characteristic $p$ and let $W(\FF)$ be the ring of Witt vectors of $\FF$. Denote by $\omega_p$ the mod $p$ cyclotomic character of $G_{\QQ,Np}$.
 Let $\rhob_0 : G_{\QQ,Np} \to \GL_2(\FF)$ be a continuous, odd representation such that $\rhob_0 = \chibar_1 \oplus \chibar_2$ for some continuous characters $\chibar_1, \chibar_2 : G_{\QQ,Np} \to \FF^{\times}$. Let $\chibar=\chibar_1 \chibar_2^{-1}$ and $N_0$ be the \emph{tame} Artin conductor of $\rhob_0$.
Let $\mathcal{C}$ be the category of complete local Noetherian rings with residue field $\FF$.
Let $R^{\ps}_{\rhob_0}$ be the universal deformation ring of the pseudo-representation $(\tr(\rhob_0),\det(\rhob_0)) : G_{\QQ,Np} \to \FF$ in $\mathcal{C}$ and let $(R^{\ps}_{\rhob_0})^{\red}$ be its maximal reduced quotient.

%Let $H(N)$ be the Hecke algebra generated by Hecke operators away from $Np$ acting faithfully on the space of modular forms of level $N$ and all weights. Suppose $\rhob_0$ arises from an eigenform of level $N$. Then it defines a maximal ideal $m_{\rhob_0}$ of $H(N)$ and let $\TT(N)_{\rhob_0}$ be the completion of $H(N)$ at $m_{\rhob_0}$. In this case, we say that $\rhob_0$ is modular of level $N$. 
If $N_0 \mid N$, then it is easy to see that $\rhob_0$ arises from a modular eigenform of level $N$ (see Lemma~\ref{modlem}).
Let $\TT(N)_{\rhob_0}$ be the big $p$-adic Hecke algebra defined in \S\ref{heckesec}. 
Roughly speaking, it is the completion of the Hecke algebra away from $Np$ acting on the space of modular forms of level $N$ and arbitrary weights at its maximal ideal corresponding to $\rhob_0$. The fact that $\rhob_0$ arises from a modular eigenform of level $N$ implies the existence of this maximal ideal.
Note that it is a reduced complete local Noetherian ring with residue field $\FF$ and  there exists a pseudo-representation $(\tau_N, \delta_N) : G_{\QQ,Np} \to \TT(N)_{\rhob_0}$ deforming $(\tr(\rhob_0),\det(\rhob_0))$ (see Lemma~\ref{heckelem}).

\subsection{Main results}
We are now ready to state our main theorem. We are stating a slightly weaker version of the theorem in order to avoid defining more notation.
\begin{thmA}[see Theorem~\ref{reducecor}, Corollary~\ref{isomcorr}, Corollary~\ref{dimcor}]
\label{thma}
Suppose $N_0 \mid N$, $\chibar|_{G_{\QQ_p}} \neq 1, \omega_p^{-1}|_{G_{\QQ_p}}$, $\dim_{\FF}(H^1(G_{\QQ,Np},\chibar)) = 1$ and $p \nmid \phi(N)$. Then:
\begin{enumerate}
\item The morphism $\Phi : R^{\ps}_{\rhob_0} \to \TT(N)_{\rhob_0}$ induced by the pseudo-representation $(\tau_N,\delta_N)$ induces an isomorphism $(R^{\ps}_{\rhob_0})^{\red} \simeq \TT(N)_{\rhob_0}$ of reduced local complete intersection rings of Krull dimension $4$.
\item Moreover, if $1 \leq \dim_{\FF}(H^1(G_{\QQ,Np},\chibar^{-1})) \leq 3$, then the map $\Phi : R^{\ps}_{\rhob_0} \to \TT(N)_{\rhob_0}$ is an isomorphism of reduced local complete intersection rings of Krull dimension $4$.
\end{enumerate}
\end{thmA}

We prove a big $R=\TT$ theorem when $p \mid \phi(N)$ as well (see Theorem~\ref{reducecor} and Corollary~\ref{condlem}). However, we need to replace $\TT(N)_{\rhob_0}$ with a big Hecke algebra of higher level. The reason behind this is the fact that if $p \mid \phi(N)$, then it is possible that $\rhob_0$ arises from a newform of level $N'$ such that $N'>N$ and the prime factors of $N$ and $N'$ are the same (for instance, see discussion after \cite[Theorem $2.7$]{Bo}).
Note that Theorem~\ref{thma} has been proven in a special case in \cite{D2} using different methods (see \cite[Theorem $5.3$, Corollary $5.4$]{D2}).

Among all the hypotheses of Theorem~\ref{thma}, the hypothesis about cyclicity of $H^1(G_{\QQ,Np},\chibar)$ is the strongest. 
We discuss this hypothesis in \S\ref{cyc} and \S\ref{miscres}.

We say that $f$ is an eigenform of \emph{non-optimal} level lifting $\rhob_0$ if $\rhob_0$ arises from $f$, the level of the newform underlying $f$ is strictly bigger than $N_0$ (the tame Artin conductor of $\rhob_0$) and not divisible by $p$.
We have borrowed this convention from a work of Diamond and Taylor (\cite{DT}).
As an application of our methods and results, we prove a result about non-optimal levels of modular eigenforms lifting $\rhob_0$.
%(in the sense of Diamond--Taylor (\cite{DT})) lifting $\rhob_0$.

\begin{thmA}[See Theorem~\ref{levelthm}, Corollary~\ref{levelcor}]
\label{thmb}
Let $\ell_1,\cdots,\ell_r$ be primes such that $\ell_i \nmid N_0$, $\chibar|_{G_{\QQ_{\ell_i}}} = \omega_p^{-1}|_{G_{\QQ_{\ell_i}}}$ and $p \nmid \ell_i-1$ for all $1 \leq i \leq r$. 
Suppose $\chibar|_{G_{\QQ_p}} \neq 1, \omega_p^{-1}|_{G_{\QQ_p}}$, $\dim_{\FF}(H^1(G_{\QQ,N_0p},\chibar)) =1$ and $\chibar_2$ is unramified at $p$. 
Suppose $\chibar_1\chibar_2 = \bar\psi\omega_p^{k_0-1}$, with $\bar\psi$ unramified at $p$ and $2 < k_0 < p$. 
\begin{enumerate}
\item For every integer $k \equiv k_0 \Mod{p-1}$, there exists an eigenform $f$ of weight $k$ and some level $N'$ such that $\prod_{i=1}^{r}\ell_i \mid N'$, $\rhob_0$ arises from $f$ and $f$ is new at $\ell_i$ for every $1 \leq i \leq r$.
\item Moreover, if $p \nmid \phi(N)$, then for every integer $k \equiv k_0 \Mod{p-1}$, there exists a newform $f$ of weight $k$ and level $N_0\prod_{i=1}^{r}\ell_i$ such that $\rhob_0$ arises from $f$.
\end{enumerate}
\end{thmA}

For $r=1$, Theorem~\ref{thmb} gives an analogue of Ribet's level raising theorem for reducible representations and it is known, under less restrictive hypotheses, for $k=k_0$  and $\rhob_0 = 1 \oplus \omega_p^{k_0 - 1}$ by work of Billerey--Menares (see \cite[Theorem $1$]{BM}).
Moreover, in \cite{BM1}, they have also proved a version of Theorem~\ref{thmb} for $r=1$ and general $\rhob_0$ (see \cite[Theorem $2$]{BM1}). 
Their methods are different from our methods.
For $r >1$, such results have been obtained by Yoo (\cite{Y}) and Wake--Wang-Erickson (\cite{WWE1}) in the case of $\rhob_0 = 1 \oplus \omega_p$ using different methods. Note that this case is excluded from our results.
To the best of our knowledge, no other such results are known for $r > 1$.

\subsection{Cyclicity of $H^1(G_{\QQ,Np},\chibar)$}
\label{cyc}
As $\chibar$ is odd, the global Euler characteristic formula implies that $H^1(G_{\QQ,Np},\chibar) \neq 0$.
In \S\ref{miscres}, we analyze the condition $\dim(H^1(G_{\QQ,Np},\chibar))=1$ to give an equivalent reformulation of this condition in terms of restriction of $\chibar$ at primes $\ell \mid Np$ and vanishing of the $\omega_p\chibar^{-1}$ component of a certain class group (see Lemma~\ref{dimonelem}).
This allows us to get a family of examples where the equality $\dim(H^1(G_{\QQ,Np},\chibar))=1$ holds (see the examples appearing after Lemma~\ref{dimonelem}).

For instance, when $N_0=1$, we get:
\begin{enumerate}
\item Suppose $p > 5$ and $2 < k < p-1$ is an even integer.
Then $\dim(H^1(G_{\QQ,p},\omega_p^{k-1}))=1$ if one of the following conditions hold:
\begin{itemize}
\item $k=4,6$.
For $k=4$, this follows from work of Kurihara (\cite[Corollary 3.8]{K}) and for $k=6$, this follows from \cite[Corollary 7.1]{Kup}.
\item $p \equiv 3 \pmod{4}$ and $k = {\dfrac{p+1}{2}}$.
This follows from work of Osburn (\cite[Theorem 1.1]{O}).
\end{itemize}
Suppose $k$ and $p$ satisfy the conditions given above and $\ell_1,\cdots,\ell_r$ are primes such that $p \nmid \prod_{i=1}^{r}\ell_i$ and $ \ell_i^{k-2} \not\equiv 1 \pmod{p}$ for all $1 \leq i \leq r$.
%Let $N=\prod_{i=1}^{r}\ell_i$.
Then Lemma~\ref{dimonelem} implies that $\dim(H^1(G_{\QQ,p\prod_{i=1}^{r}\ell_i},\omega_p^{k-1}))=1$.
 Hence, in all these cases, the tuple $(p, \prod_{i=1}^{r}\ell_i, \omega_p^a \oplus \omega_p^{a+k-1})$, for any integer $a$, satisfies the hypotheses of Theorem~\ref{thma}.
In these cases, if either $k=4$ or $p \equiv 3 \pmod{4}$, then $ \ell_i^{k-2} \not\equiv 1 \pmod{p}$ if and only if $\ell_i \not\equiv \pm 1 \pmod{p}$.
%\item If $p > 5$, $p \equiv 3 \pmod{4}$, $\chibar = \omega_p^{\frac{p-1}{2}}$ and $\ell_1,\cdots,\ell_r$ are primes such that $p \nmid \prod_{i=1}^{r}\ell_i$ and $ \ell_i \not\equiv \pm 1 \pmod{p}$ for all $1 \leq i \leq r$, then $\dim(H^1(G_{\QQ,p\prod_{i=1}^{r}\ell_i},\omega_p^{\frac{p-1}{2}}))=1$ and hence, the tuple $(p, \prod_{i=1}^{r}\ell_i, 1 \oplus \omega_p^{\frac{p-1}{2}})$ satisfies the hypotheses of Theorem~\ref{thma}.
\item If $k>2$ is an even integer and $p > k+1$ is a prime such that $p \nmid B_k$ (where $B_k$ is $k$-th Bernoulli number), then using Herbrand-Ribet theorem and reflection principle, we get that $\dim(H^1(G_{\QQ,p},\omega_p^{1-k}))=1$.
Moreover, if $\ell_1,\cdots,\ell_r$ are primes such that $p \nmid \prod_{i=1}^{r}\ell_i$ and $p \nmid \ell_i^k-1$ for all $1 \leq i \leq r$, then Lemma~\ref{dimonelem} implies that  $\dim(H^1(G_{\QQ,p\prod_{i=1}^{r}\ell_i},\omega_p^{1-k}))=1$.
Therefore, the tuple $(p, \prod_{i=1}^{r}\ell_i, \omega_p^a \oplus \omega_p^{a+k-1})$, for any integer $a$, satisfies the hypotheses of Theorem~\ref{thma}.
Note that we have only excluded finitely many primes $p$ here.
Given a prime $p > k+1$, the set $\{\ell \mid \ell \text{ is a prime, } p \nmid \ell^k-1\}$ is infinite with Dirichlet density $\dfrac{p-m-1}{p-1} \geq \dfrac{\phi(p-1)}{p-1}$, where $m$ is the gcd of $k$ and $p-1$ and $\phi$ is Euler's totient function.
% of primes $\ell$'s satisfying the hypotheses has Dirichlet density $\frac{p-m-1}{p-1}$, where $m$ is the gcd of $k$ and $p-1$.
% \item If $p > 5$ is a \emph{regular} prime such that $p \equiv 3 \pmod{4}$, $\chibar = \omega_p^5$ and $\ell_1,\cdots,\ell_r$ are primes such that $p \nmid \prod_{i=1}^{r}\ell_i$ and $\ell_i \not\equiv \pm 1 \pmod{p}$ for all $1 \leq i \leq r$, then $\dim(H^1(G_{\QQ,p\prod_{i=1}^{r}\ell_i},\omega_p^5))=1$ and hence, the tuple $(p, \prod_{i=1}^{r}\ell_i, 1 \oplus \omega_p^5)$ satisfies the hypotheses of Theorem~\ref{thma}.
%\end{enumerate}
\item In general, if Vandiver's conjecture is true, $p > 5$, $k$ is an even integer such that $k \not\equiv 0,2 \mod{p-1}$ and $\ell_1,\cdots,\ell_r$ are primes such that $p \nmid \prod_{i=1}^{r}\ell_i$ and $p \nmid \ell_i^{k-2}-1$ for all $1 \leq i \leq r$, then $\dim(H^1(G_{\QQ,p\prod_{i=1}^{r}\ell_i},\omega_p^{k-1}))=1$ and hence, the tuple $(p, \prod_{i=1}^{r}\ell_i, \omega_p^a \oplus \omega_p^{a+k-1})$, for any integer $a$, satisfies the hypotheses of Theorem~\ref{thma}.
\end{enumerate}
On the other hand, if $\ell$ is a prime such that $\ell \equiv \pm 1 \pmod{p}$, $N$ is an integer divisible by $\ell$ and $k$ is an even integer, then $\dim(H^1(G_{\QQ,N p},\omega_p^{k-1})) \geq 2$ (see Lemma~\ref{cyclem}).
Thus, when $N_0=1$, $(7, 13, 1 \oplus \omega_p^3)$ is the `smallest' tuple which satisfies all the hypotheses of Theorem~\ref{thma} except the cyclicity of $H^1(G_{\QQ,Np},\chibar)$.
%Note that the hypotheses of \cite[Theorem $1.2$]{SW1} also imply that $\dim(H^1(G_{\QQ,Np},\chibar))=1$.

\subsection{Sketch of proofs of main results}\label{sketchsec}
We now give a brief sketch of the proofs of our main theorems. Our basic strategy is similar to that of B\"{o}ckle (\cite{Bo}) and Allen (\cite{A}) which we will describe now.
B\"{o}ckle and Allen prove big $R=\TT$ theorems by proving that every component of $\spec(R_{\rhob_0})$ contains a modular point (i.e. a prime corresponding to a modular eigenform of level $N$ lifting $\rhob_0$) which is also smooth.
To prove the existence of such points, both B\"{o}ckle and Allen first show (using different methods) that certain locus in the universal deformation space is defined by a suitable number of equations.
This is a key step in both \cite{Bo} and \cite{A}.
Then they combine this result with the corresponding small $R=\TT$ theorems and a result of B\"{o}ckle (\cite[Theorem $2.4$]{Bo2}) about a presentation of the universal deformation rings (or an analogue of this result) to show that the universal deformation rings are local complete intersection rings of appropriate dimension. 
Using this property of the universal deformation rings, along with the number of equations giving the chosen locus and corresponding small $R=\TT$ theorems, they prove the existence of the desired modular points.

We also follow similar steps to prove Theorem~\ref{thma}. However, as we are dealing with residually reducible representations, there are some obstacles.
 Note that the presentation result of B\"{o}ckle (\cite[Theorem $2.4$]{Bo2}) states that the universal deformation ring of a representation $\rhob : G_{\QQ,Np} \to \GL_2(\FF)$ (when it exists) is a quotient of the power series ring $W(\FF) \llbracket X_1,\cdots,X_{d_1} \rrbracket$ by an ideal $J$ generated by at most $d_2$ elements, where $d_i=\dim_{\FF}(H^i(G_{\QQ,Np},\text{Ad}(\rhob)))$ for $i=1,2$ and $\text{Ad}(\rhob)$ is the \emph{adjoint representation} associated to $\rhob$.
We do not have an analogue of this result for the pseudo-deformation rings and very few `small' $R=\TT$ theorems are available. 
Moreover, the pseudo-deformation rings are also not always local complete intersection rings (see \cite[Corollary $4.20$]{D2} for example).
%Note that results of both these types play an important role in the arguments of \cite{Bo} and \cite{A}.
Hence, we need to use slightly different techniques in order to make our strategy work. 

As a first step, we replace the pseudo-deformation ring $R^{\ps}_{\rhob_0}$ with a deformation ring.
To be precise, let $x \in H^1(G_{\QQ,Np},\chibar)$ be a non-zero element, $\rhob_x$ be the non-trivial extension of $\chibar_2$ by $\chibar_1$ corresponding to $x$ and $R_{\rhob_x}$ be the universal deformation ring of $\rhob_x$ (in the sense of Mazur) which exists due to \cite{Ra}. 
%As $x \neq 0$, universal deformation ring $R_{\rhob_x}$ of $\rhob_x$ (in the sense of Mazur) does exist due to \cite{Ra}.
Since we are assuming that $\dim_{\FF}(H^1(G_{\QQ,Np},\chibar))=1$, \cite[Theorem $3.12$]{D2} implies that the map $R^{\ps}_{\rhob_0} \to \TT(N)_{\rhob_0}$ factors through  $R_{\rhob_x}$. 
%We also construct a deformation of $\rhob_x$ to $\TT(N)_{\rhob_0}$ to give a proof of this claim without using \cite[Theorem $3.12$]{D2} (see Lemma~\ref{factorlem} and its proof).

Note that, the existence of a map $R_{\rhob_x} \to \TT(N)_{\rhob_0}$ is equivalent to the existence of a deformation of $\rhob_x$ to $\TT(N)_{\rhob_0}$.
However, in general, a deformation of $\rhob_x$ to $\TT(N)_{\rhob_0}$ may not exist and consequently, there may not be any map from $R_{\rhob_x}$ to $\TT(N)_{\rhob_0}$. 
But \cite[Theorem $3.12$]{D2} implies that such a deformation does exist when $\dim_{\FF}(H^1(G_{\QQ,Np},\chibar))=1$.

To illustrate this point further, we construct, using the assumption $\dim_{\FF}(H^1(G_{\QQ,Np},\chibar))=1$, a deformation $\rho : G_{\QQ,Np} \to \GL_2(\TT(N)_{\rhob_0})$ of $\rhob_x$ such that $\tr(\rho) = \tau_N$ and $\det(\rho)=\delta_N$ (see Lemma~\ref{factorlem} and its proof). 
This, along with the universal properties of $R^{\ps}_{\rhob_0}$ and $R_{\rhob_x}$, also allows us to conclude that the map $R^{\ps}_{\rhob_0} \to \TT(N)_{\rhob_0}$ factors through  $R_{\rhob_x}$. 
So in the rest of the article, we compare $R_{\rhob_x}$ with $\TT(N)_{\rhob_0}$.

%However, the condition $\dim(H^1(G_{\QQ,Np},\chibar))=1$ helps us to construct a deformation of $\rhob_x$ to $\TT(N)_{\rhob_0}$ and hence, a map $R_{\rhob_x} \to \TT(N)_{\rhob_0}$.
%that the map $R^{\ps}_{\rhob_0} \to R_{\rhob_x}$ is surjective and the kernel of this map is nilpotent. As $\TT(N)_{\rhob_0}$ is reduced the map $R^{\ps}_{\rhob_0} \to \TT(N)_{\rhob_0}$ factors through  $R_{\rhob_x}$. 

Let $\rho^{\univ}$ be the universal deformation of $\rhob_x$. As a key step in our proof, we obtain a description of $R_{\rhob_x}[\rho^{\univ}(G_{\QQ_p})]$ in the form of a Generalized Matrix Algebra (GMA) (in the sense of \cite{BC}) and use it to obtain elements $\alpha, \beta, \delta_k \in R_{\rhob_x}$ for every $k \geq 2$ such that if $I_k = (\alpha,\beta,\delta_k)$, then:
\begin{enumerate}
 \item\label{point1} $R_{\rhob_x}/I_k$ is a finite $W(\FF)$-algebra of Krull dimension $1$,
\item\label{point2} $\rho^{\univ} \Mod{I_k}$ is $p$-ordinary with $\det(\rho^{\univ} \Mod{I_k})=\epsilon_k\chi_p^{k-1}$, where $\epsilon_k$ is a character of finite order and $\chi_p$ is the $p$-adic cyclotomic character.
\end{enumerate}
Note that the proof of the Claim~\eqref{point2} above follows directly from the definitions of the elements $\alpha, \beta,$ and $\delta_k $.
We split the proof of Claim~\eqref{point1} above in two cases: one in which the hypotheses of Skinner-Wiles \cite{SW1} hold and one where they do not hold.
In the first case, we use the small $R=\TT$ theorem proved by Skinner--Wiles (\cite[Theorem $6.1$]{SW1}) to conclude the finiteness, while in the second case we use a finiteness result proved by Pan (\cite[Theorem $5.1.2$]{P}). This also implies that $R_{\rhob_x}$ is a local complete intersection ring of Krull dimension $4$.

We then combine the results from previous paragraph with an appropriate modularity lifting theorem to conclude the existence of a modular point in every component of $\spec(R_{\rhob_x})$.
To prove the smoothness of these points, we either use a small $R=\TT$ theorem (when it is available) or a combination of other appropriate results.
To be precise, in the first case, we only use \cite[Theorem $6.1$]{SW1} to conclude the existence of smooth modular points.
However in the second case, there are no small $R=\TT$ theorems available.
So we combine a modularity lifting result of Skinner-Wiles (\cite[Theorem A]{SW2}), results of Kisin (\cite{Ki}) and Weston (\cite{We1}) about the smoothness of the universal deformation ring of the $p$-adic Galois representation attached to a modular eigenform and Proposition~\ref{unobsprop} to conclude the existence of smooth modular
points in every component of $\spec(R_{\rhob_x})$.
Combining the existence of smooth modular points in every component of $\spec(R_{\rhob_x})$ with the Gouv\^{e}a--Mazur infinite fern argument gives us Theorem~\ref{thma}.
% points in every component of $\spec(R_{\rhob_x})$ but it does not imply the smoothness. 
%To conclude that they are smooth, one needs to use results from \cite{Ki} and \cite{We1} when the modular point corresponds to a cusp form. 
%Some work needs to be done to prove that their hypotheses are satisfied in the case where $\chibar=\omega_p^k$.
%When it corresponds to an Eisenstein series, we prove a separate result (Proposition~\ref{unobsprop}) which implies the smoothness of the point.
%such that the $p$-adic Galois representation $\rho_f : G_{\QQ,Np} \to \GL_2(\overline{\QQ_p})$ is unobstructed i.e. $H^2(G_{\QQ,Np}, \text{Ad}(\rho_f))=0$, where $\text{Ad}(\rho_f)$ is the adjoint representation of $\rho_f$.
%Note that this means that the universal deformation ring of $\rho_f$ is a power series ring which is not contained in any other component of $\spec(R^{\ps}_{\rhob_0})$.

Since Pan's finitness result has less restrictive hypotheses than those of \cite[Theorem $6.1$]{SW1}, all the arguments that we use in the second case also work in the first case. 
However, we prefer to keep different proofs for both the cases as their arguments are different and the proof in the first case is shorter.
Note that in the second case, one can prove small $R=\TT$ theorems under some hypotheses using either ordinary $R=\TT$ theorem proved by Wake--Wang-Erickson (\cite{WWE}) or the results proved by Berger--Klosin (\cite{BeKi}) and then use the arguments of the first case of the main theorem to prove a big $R=\TT$ theorem under those hypotheses. 
However this will only cover a small proportion of the cases not satisfying the hypotheses of \cite{SW1}.

To prove Theorem~\ref{thmb}, we combine Proposition~\ref{finprop} along with the main theorem of \cite{SW2} and \cite[Theorem $4.7$]{Bo1} to get a description of the universal deformation ring of $\rhob_x$ for $G_{\QQ,Np}$ (where $N= N_0\prod_{i=1}^{r}\ell_i$) in terms of the universal deformation ring of $\rhob_x$ for $G_{\QQ,N_0p}$ (in the spirit of \cite[Theorem $4.7$]{Bo1}).
We then use this description along with an analogue of Corollary~\ref{fincor} and the modularity lifting theorem of \cite{SW2} to conclude the theorem.

\subsection{Remarks on general case}
Note that we prove our main theorem by relating $\TT(N)_{\rhob_0}$ with the universal deformation ring (in the sense of Mazur) of a non-split reducible representation and proving that they are isomorphic (see Theorem~\ref{mainthm}).
As mentioned in \S\ref{sketchsec}, there may not be a representation over $\TT(N)_{\rhob_0}$ giving rise to the pseudo-representation $(\tau_N,\delta_N)$ when $\dim_{\FF}(H^1(G_{\QQ,Np},\chibar))$ $> 1$ and $\dim_{\FF}(H^1(G_{\QQ,Np},\chibar^{-1})) > 1$.
As a result, we cannot compare $\TT(N)_{\rhob_0}$ with the universal deformation ring of a non-split reducible representation.

But even in this case,  one can use Pan's finiteness result (\cite[Theorem $5.1.2$]{P}) to get an analogue of Proposition~\ref{finprop} (which is a key ingredient in the proof our main theorem) as long as $\chibar|_{G_{\QQ_p}} \neq 1$ and the dimensions of the images of both the restriction maps $H^1(G_{\QQ,Np},\chibar) \to H^1(G_{\QQ_p},\chibar)$ and $H^1(G_{\QQ,Np},\chibar^{-1}) \to H^1(G_{\QQ_p},\chibar^{-1})$ are $1$.
%As long as $\chibar|_{G_{\QQ_p}} \neq 1$ and the dimensions of the images of both the restriction maps $H^1(G_{\QQ,Np},\chibar) \to H^1(G_{\QQ_p},\chibar)$ and $H^1(G_{\QQ,Np},\chibar^{-1}) \to H^1(G_{\QQ_p},\chibar^{-1})$ are $1$, one can use Pan's result (\cite[Theorem $5.1.2$]{P}) to get an analogue of Proposition~\ref{finprop} (which is a key ingredient in the proof our main theorem) even if $\dim_{\FF}(H^1(G_{\QQ,Np},\chibar))>1$ and $\dim_{\FF}(H^1(G_{\QQ,Np},\chibar^{-1}))>1$. 
However, in order to conclude the existence of a modular point in every component of $\spec(R^{\ps}_{\rhob_0})$ using our arguments, we need to know that the dimension of every component of $\spec(R^{\ps}_{\rhob_0})$ is at least $4$. 
But such results are not known when $\dim_{\FF}(H^1(G_{\QQ,Np},\chibar)) >1$ and $\dim_{\FF}(H^1(G_{\QQ,Np},\chibar^{-1})) >1$. 
%In fact the pseudo-deformation rings in these cases are also not necessarily well-behaved (see \cite[Corollary $4.21$]{D2} for example). 
If one can prove that the dimension of every component of $\spec(R^{\ps}_{\rhob_0})$ is at least $4$, then our proof in the second case (after replacing modularity lifting results of Skinner-Wiles (\cite{SW1}, \cite{SW2}) with a modularity lifting result of Pan (\cite[Theorem $1.0.2$]{P})) will give us big $R=\TT$ theorems in almost all such cases.

\subsection{ Structure of the paper}
 In \S\ref{prelimsec}, we list all the notations and describe the setup that we will be working with. Then we define various deformation rings that we will be working with and gather some background results and miscellaneous results which will be used in the proof of the main theorem.
In \S\ref{defosec}, we study properties of $R_{\rhob_x}$ to prove some results about its structure. We also prove the key finiteness result mentioned above and the existence of a smooth modular point in every component of $\spec(R_{\rhob_x})$. This section forms the backbone of this article.
In \S\ref{mainsec}, we review definitions and properties of the big $p$-adic Hecke algebras and prove the main theorem. We also give an application of our main theorem to Hecke algebras modulo $p$.
In \S\ref{levelsec}, we first prove some preliminary results and then combine them with the finiteness result from \S\ref{defosec} to prove Theorem~\ref{thmb}. We have tried to keep this section self-contained and independent from all sections other than \S\ref{defosec}. 
%It only uses the finiteness results from\S\ref{defsec} and some basic results from \S\ref{prelimsec} from the rest of the paper.

\subsection{Acknowledgments}
I would like to thank Patrick Allen, Gebhard B\"{o}ckle, Najmuddin Fakhruddin, Bharathwaj Palvannan and Gabor Wiese for helpful conversations on the topic.
I would also like to thank the anonymous referees for a careful reading of the manuscript and for providing numerous comments and suggestions which greatly helped in improving the exposition.

\section{Preliminaries}
\label{prelimsec}
In this section, we will collect some notation, definitions and preliminary results that will be used later in the article. 
%We start with establishing some notations and describing the setup that we will be working with. We then define various deformation rings that we will be working with and then move on to prove some miscellaneous results which will be used later.

\subsection{Notations and setup}
\label{notsec}
We keep the notation developed in the previous section and also introduce some additional notation here.
%Let $p$ be an odd prime. 
For an integer $M$, let $G_{\QQ,Mp}$ be the Galois group of the maximal extension of $\QQ$ unramified outside primes dividing $Mp$ and $\infty$ over $\QQ$.
 For a prime $\ell$, denote the absolute Galois group of $\QQ_{\ell}$ by $G_{\QQ_{\ell}}$ and its inertia subgroup by $I_{\ell}$. 
Fix an embedding $\iota : \overline{\QQ} \to \mathbb{C}$ and for every prime $\ell$, fix an embedding $\iota_{\ell} : \overline{\QQ} \to \overline{\QQ}_{\ell}$.
For a prime $\ell$ and an integer $M$, denote by $i_{\ell,M} : G_{\QQ_{\ell}} \to G_{\QQ,Mp}$ the map induced by $\iota_{\ell}$.
% Note that such an embedding is well defined up to conjugacy for a fixed $M$.

For a representation $\rho$ of $G_{\QQ,Mp}$, we denote the representation $\rho \circ i_{\ell,M}$ by $\rho|_{G_{\QQ_{\ell}}}$ and denote the restriction of $\rho|_{G_{\QQ_{\ell}}}$ to $I_{\ell}$ by $\rho|_{I_{\ell}}$.
Denote the mod $p$ cyclotomic character of $G_{\QQ,Mp}$ by $\omega_p$ and the $p$-adic cyclotomic character by $\chi_p$.
By abuse of notation, for every prime $\ell$, we also denote $\omega_p|_{G_{\QQ_{\ell}}}$ and $\chi_p|_{G_{\QQ_{\ell}}}$ by $\omega_p$ and $\chi_p$, respectively.

%Let $\FF$ be a finite field of characteristic $p$ and let $W(\FF)$ be the ring of Witt vectors of $\FF$.
If $\eta : G_{\QQ,Mp} \to \FF^{\times}$ is a character, denote its Teichm\"{u}ller lift by $\hat\eta$.
All the representations, pseudo-representations and cohomology groups that we consider in this article are assumed to be continuous unless mentioned otherwise.
Given a representation $\rho$ of $G_{\QQ,Mp}$ defined over $\FF$, denote by $\dim(H^i(G_{\QQ,Mp},\rho))$, the dimension of $H^i(G_{\QQ,Mp},\rho)$ as a vector space over $\FF$.
If $K$ is a number field, then we will denote its class group by $\text{Cl}(K)$.

If $R$ is a ring, then denote by $M_2(R)$ the $R$-module consisting of $2 \times 2$ matrices with entries in $R$. If $P$ is a prime ideal of $R$, then denote by $(R)_P$ the localization of $R$ at $P$. Denote by $(R)^{\red}$ the maximal reduced quotient of $R$ and by $R^{\times}$ the group of units of $R$. If $K$ is a field and $\rho : G_{\QQ,Mp} \to \GL_2(K)$ is a representation, then denote by $\text{Ad}(\rho)$ the representation on $M_2(K)$ in which the action of $g \in G_{\QQ,Mp}$ on $M_2(K)$ is given by conjugation by $\rho(g)$.

We will now briefly recall the setup which we introduced in the previous section and also introduce some additional hypotheses. We will work with this setup and keep the hypotheses of this setup throughout the article unless mentioned otherwise.

\begin{setup}\label{basic} %Let $N \geq 1$ be an integer not divisible by $p$.
 Let $\rhob_0 : G_{\QQ,Np} \to \GL_2(\FF)$ be an odd representation such that $\rhob_0 = \chibar_1 \oplus \chibar_2$ for some characters $\chibar_1, \chibar_2 : G_{\QQ,Np} \to \FF^{\times}$. As $\rhob_0$ is odd and $p$ is odd, $\chibar_1 \neq \chibar_2$. Let $\chibar=\chibar_1 \chibar_2^{-1}$ and denote by $F_{\chibar}$, the extension of $\QQ$ fixed by $\ker(\chibar)$.
Throughout this article, we assume that $\chibar|_{G_{\QQ_p}} \neq 1, \omega_p^{-1}$ and $\dim(H^1(G_{\QQ,Np},\chibar))=1$ unless mentioned otherwise. 
\end{setup}

We finish this subsection by introducing one more piece of notation.
We say that a modular form $f$ has \emph{level} $M$ if $f$ is a modular form of level $\Gamma_1(M)$.
For an eigenform $f$ of level $M$, denote by $\mathcal{O}_f$ be the ring of integers of the extension of $\QQ_p$ obtained by adjoining the Hecke eigenvalues of $f$ to $\QQ_p$, denote by $\varpi_f$ the uniformizer of $\mathcal{O}_f$ and denote by $\rho_f : G_{\QQ,Mp} \to \GL_2(\mathcal{O}_f)$ the semi-simple $p$-adic Galois representations attached to $f$ such that for every prime $\ell \nmid Mp$, $\tr(\rho_f(\frob_{\ell}))$ is the $T_{\ell}$-eigenvalue of $f$.
 We say that $\rho_f$ lifts $\rhob_0$ if $(\tr(\rho_f),\det(\rho_f)) \equiv (\tr(\rhob_0),\det(\rhob_0)) \Mod{\varpi_f}$.

\subsection{Deformation rings}
\label{defsec}
In the Setup~\ref{basic} established in the previous subsection, we will now define various deformation rings.
We will be working with the pseudo-representations introduced by Chenevier in \cite{C} which generalizes the theory of pseudocharacters developed by Procesi, Wiles, Taylor, Nyssen and Rouquier (see the introduction of \cite{C} for more details). 
We begin by recalling the definition of $2$-dimensional pseudo-representations.
\begin{defi}
\label{pseudodef}
Given an integer $M$ and a complete Noetherian local ring $R$, a continuous pseudo-representation of $G_{\QQ,Mp}$ on $R$ of dimension $2$ is a tuple of continuous functions $(t,d) : G_{\QQ,Mp} \to R$ such that
\begin{enumerate}
\item $d : G_{\QQ,Mp} \to R^{\times}$ is a group homomorphism,
\item $t(1)=2$,
\item $t(gh) = t (hg)$ for all $g, h \in G_{\QQ,Mp}$,
\item $d(g)t(g^{-1}h)+t(gh) = t(g)t(h)$ for all $g, h \in G_{\QQ,Mp}$.
\end{enumerate}
\end{defi}
We refer the reader to \cite[Section $1.4$]{BK} for the definition and properties of $2$-dimensional pseudo-representations.
%All the pseudo-representations considered in this articles are assumed to be continuous unless mentioned otherwise.
Since $\rhob_0$ is a $2$-dimensional representation, $(\tr(\rhob_0), \det(\rhob_0)) : G_{\QQ,Np} \to \FF$ is a continuous $2$-dimensional pseudo-representation of $G_{\QQ,Np}$.
Throughout this article, we will only be working with continuous $2$-dimensional pseudo-representations.
So for the ease of notation, we will just refer to them as \emph{pseudo-representations} from now on.
Recall, from \S\ref{aimsec}, that Chenevier actually introduced the notion of determinants in \cite{C} to generalize the notion of pseudocharacters. In the case of dimension $2$, he established a bijection between the set of $R$-valued determinants of $G_{\QQ,Mp}$ and the set of $R$-valued pseudo-representations of $G_{\QQ,Mp}$ recalled in Definition~\ref{pseudodef} above (see \cite[Lemma 1.9]{C}).

%Let $\mathcal{C}$ be the category of complete local Noetherian rings with residue field $\FF$. 
If $R$ is an object of $\mathcal{C}$, denote its maximal ideal by $m_R$.
\begin{defi}
\label{defodef}
%\begin{enumerate}
%\item
 Let $R^{\ps}_{\rhob_0}$ be the universal deformation ring of the pseudo-representation $(\tr(\rhob_0),\det(\rhob_0))$. So $R^{\ps}_{\rhob_0}$ is an object of $\mathcal{C}$ which represents the functor from $\mathcal{C}$ to the category of sets which sends an object $R$ of $\mathcal{C}$ to the set of $2$-dimensional pseudo-representations $(t,d) : G_{\QQ,Np} \to R$ such that $t \Mod{m_R} = \tr(\rhob_0)$ and $d \Mod{m_R} = \det(\rhob_0)$.
\end{defi}

The existence of $R^{\ps}_{\rhob_0}$ is proved by Chenevier (see \cite[Proposition 3.3]{C} and \cite[Proposition 3.7]{C}). 
Given a pseudo-representation $(t,d) : G_{\QQ,Np} \to R$, denote by $(t|_{G_{\QQ_p}}, d|_{G_{\QQ_p}})$ the pseudo-representation $(t \circ i_{p,N}, d \circ i_{p,N}) : G_{\QQ_p} \to R$.

%\item\label{item2} 
\begin{defi}
\label{defodef2}
Since we are assuming $\chibar|_{G_{\QQ_p}} \neq 1$, there exists a $g_0 \in G_{\QQ_p}$ such that $\chibar_1(g_0) \neq \chibar_2(g_0)$. 
Fix such a $g_0 \in G_{\QQ_p}$.
Let $x \in H^1(G_{\QQ,Np},\chibar)$ be a non-zero element.
Define $\rhob_x : G_{\QQ,Np} \to \GL_2(\FF)$ to be the representation such that
\begin{enumerate} 
 \item $\rhob_x(g) = \begin{pmatrix} \chibar_1(g) & *\\ 0 & \chibar_2(g) \end{pmatrix}$ for all $g \in G_{\QQ,Np}$, where $*$ corresponds to $x$,
\item $\rhob_x(g_0) = \begin{pmatrix} \chibar_1(g_0) & 0\\ 0 & \chibar_2(g_0) \end{pmatrix}$.
\end{enumerate}
 Let $R_{\rhob_x}$ be the universal deformation ring (in the sense of Mazur (\cite{M})) of $\rhob_x$ in $\mathcal{C}$ and $\rho^{\univ} : G_{\QQ,Np} \to \GL_2(R_{\rhob_x})$ be a universal deformation of $\rhob_x$. 
\end{defi}

Since we are assuming that $\chibar_1 \neq \chibar_2$ and $x \neq 0$, the existence of $R_{\rhob_x}$ follows from work of Ramakrishna (\cite[Theorm 1.1]{Ra}).

%\item 
\begin{defi}
\label{defodef3}
For $x$ and $\rhob_x$ as given in Definition~\ref{defodef2}, let $\rhob'_x : G_{\QQ,Np} \to \GL_2(\FF)$ be the representation given by $\rhob'_x := \rhob_x \otimes \chibar_2^{-1}$. Let $R^{\text{Sel}}_{\rhob'_x}$ be the object of $\mathcal{C}$ which represents the functor from $\mathcal{C}$ to the category of sets which sends an object $R$ of $\mathcal{C}$ to the set of equivalence classes of representations $\rho : G_{\QQ,Np} \to \GL_2(R)$ such that
%\begin{enumerate}
  $\rho \Mod{m_R} = \rhob'_x$, 
$\rho|_{G_{\QQ_p}} \simeq \begin{pmatrix} \psi_1 & * \\ 0 & \psi_2\end{pmatrix}$, 
 $\psi_2(I_p) =1$ and $\det\rho|_{I_p} = \hat\chibar\hat\omega_p^{-1}\chi_p|_{I_p}$.
%\end{enumerate}
A deformation of $\rho$ of $\rhob'_x$ satisfying the conditions given above is called a \emph{Selmer} deformation of $\rhob'_x$ (see \cite{SW1}).
\end{defi}

The deformation ring $R^{\text{Sel}}_{\rhob'_x}$ is defined in \cite[Section $2$]{SW1} and its existence is proved by combining \cite[Proposition 3]{M} and \cite[Theorem 1.1]{Ra}.
%The existence of $R^{\text{Sel}}_{\rhob'_x}$ follows from \cite{M} and \cite{Ra} (see \cite[Section $2$]{SW1} for more details).

%\item
\begin{defi}
\label{defodef4} 
Let $R^{\ps,\text{ord}}_{\chibar}$ be the object of $\mathcal{C}$ which represents the functor from $\mathcal{C}$ to the category of sets sending $R$ to the set of pseudo-representations $(t,d) : G_{\QQ,Np} \to R$ such that $t \Mod{m_R} = 1+\chibar$, $d=\hat\chibar$ and $t|_{G_{\QQ_p}} = \psi_1 + \psi_2$, where $\psi_1$ and $\psi_2$ are characters of $G_{\QQ_p} \to R^{\times}$ lifting $1$ and $\chibar|_{G_{\QQ_p}}$, respectively. 
%\end{enumerate}
\end{defi}

The deformation ring $R^{\ps,\text{ord}}_{\chibar}$  ia defined by Pan in \cite[Section 5.1.1]{P}.
Since $\chibar|_{G_{\QQ_p}} \neq 1$, it is easy to verify the existence of $R^{\ps,\text{ord}}_{\chibar}$.
\begin{rem}
Note that, the notions of $p$-ordinary pseudo-representations have also been formulated by Wake--Wang-Erickson (\cite[Definition 3.5.1]{WWE} and \cite[Definition 3.8.1]{WWE1}) and Calegari--Specter (\cite[Definition 2.5]{CS}).
However, both sets of authors formulate the $p$-ordinary condition as a \emph{global} condition.
In contrast, Pan formulates the $p$-ordinary condition as a purely \emph{local} condition (as seen in Definition~\ref{defodef4}).
The $p$-ordinary condition of Pan is included in the formulations of both Calegari--Specter and Wake--Wang-Erickson.
Thus their $p$-ordinary conditions are stronger than those of Pan.
As a result, the $p$-ordinary pseudo-deformation rings of Calegari--Specter and Wake--Wang-Erickson for the pseudo-representation $(1+\chibar,\chibar)$ arise as quotients of $R^{\ps,\text{ord}}_{\chibar}$.
\end{rem}

Since $(\tr(\rho^{\univ}), \det(\rho^{\univ}))$ is a deformation of $(\tr(\rhob_0),\det(\rhob_0))$, it induces a map $F_x : R^{\ps}_{\rhob_0} \to R_{\rhob_x}$. 
We first recall a result of Kisin regarding this map $F_x$.
\begin{lem}
\label{kisinjams}
If $\dim(H^1(G_{\QQ,Np},\chibar)) =1$, then the morphism $F_x$ is surjective.
\end{lem}
\begin{proof}
This is part $(2)$ of \cite[Corollary $1.4.4$]{Ki3}.
\end{proof}
%Since we are assuming that $\dim(H^1(G_{\QQ,Np},\chibar)) =1$, it follows, from \cite[Corollary $1.4.4(2)$]{Ki3}, that $F_x$ is surjective. 
In fact, one can say a bit more. We recall here a result from \cite{D2}:

\begin{prop}
\label{nilprop}
If $\dim(H^1(G_{\QQ,Np},\chibar)) =1$, then the kernel of $F_x$ is nilpotent.
\end{prop}
\begin{proof}
This is just \cite[Theorem $3.12$]{D2}.
\end{proof}

Thus we see that the map $F_x$ induces an isomorphism between $(R^{\ps}_{\rhob_0})^{\red}$ and $(R_{\rhob_x})^{\red}$. So, we will mostly focus on the ring $R_{\rhob_x}$ from now on.

\subsection{Background results}
In this subsection, we will collect several results from the literature which will be frequently used in this article.

We begin by defining the tame Artin conductor of a Galois representation. 
If $\ell$ is a prime and $\rho$ is a representation of $G_{\QQ_{\ell}}$, then denote by $(\rho)^{I_{\ell}}$ the subspace of $\rho$ on which the inertia group $I_{\ell}$ acts trivially (i.e. the subspace of $I_{\ell}$-invariants).
\begin{defi}
\begin{enumerate}
\item Let $K$ be a finite extension of $\QQ_p$ and $\rho : G_{\QQ,Np} \to \GL_2(K)$ be a representation.
Let $\mathcal{S}_{\rho} = \{\ell \mid \ell \text{ is a prime }, \ell \neq p \text{ and } \ell \mid N \}$.
For a prime $\ell$, let $$e_{\ell}(\rho) = \dim(\rho) - \dim((\rho|_{G_{\QQ_{\ell}}})^{I_{\ell}}) + \text{sw}(\rho|_{G_{\QQ_{\ell}}}),$$ where $\text{sw}(\rho|_{G_{\QQ_{\ell}}})$ is the \emph{Swan conductor} of $\rho|_{G_{\QQ_{\ell}}}$.
Then the tame Artin conductor of $\rho$ is $\prod_{\ell \in \mathcal{S}_{\rho}} \ell^{e_{\ell}(\rho)}$.

\item Let $\rhob : G_{\QQ,Np} \to \GL_2(\FF)$ be a representation.
Let $\mathcal{S}_{\rhob} = \{\ell \mid \ell \text{ is a prime }, \ell \neq p \text{ and } \ell \mid N \}$.
For a prime $\ell$, let $$e_{\ell}(\rhob) = \dim(\rhob) - \dim((\rhob|_{G_{\QQ_{\ell}}})^{I_{\ell}}) + \text{sw}(\rhob|_{G_{\QQ_{\ell}}}),$$ where $\text{sw}(\rhob|_{G_{\QQ_{\ell}}})$ is the \emph{Swan conductor} of $\rhob|_{G_{\QQ_{\ell}}}$.
Then the tame Artin conductor of $\rhob$ is $\prod_{\ell \in \mathcal{S}_{\rhob}} \ell^{e_{\ell}(\rhob)}$.
\end{enumerate}
\end{defi}
Thus, while defining the tame Artin conductor, we are excluding the contribution coming from $p$.
We refer the reader to \cite[Section 1]{L} for the definition of the Swan conductor and other details.

We now recall a well known result which relates the tame Artin conductor of the $p$-adic Galois representation $\rho_f$ attached to a newform $f$ to the tame level of $f$.
\begin{prop}
\label{condprop}
Let $f$ be a newform and $\rho_f$ be the semi-simple $p$-adic Galois representation attached to $f$. Then the tame Artin conductor of $\rho_f$ is equal to the tame level of $f$.
\end{prop}
\begin{proof}
This is \cite[Lemma 4.1]{L}.
\end{proof}

Given a finite set of primes $S$ which includes the primes dividing $N_0$, we now determine the greatest possible tame level of a newform $f$ such that $\rho_f$ lifts $\rhob_0$ and the prime divisors of the tame level of $f$ are contained in $S$.
For a prime $\ell$, denote the standard $\ell$-adic valuation by $v_{\ell}$.
We first recall a result of Carayol.
\begin{prop}
\label{carayolprop}
Let $f$ be a newform of tame level $N$ such that $\rho_f$ lifts $\rhob_0$ and $\ell$ be a prime dividing $N$. Then:
\begin{enumerate}
\item If $p \nmid \ell^2-1$ and $\ell \nmid N_0$, then $v_{\ell}(N)=1$ and $\chibar|_{G_{\QQ_{\ell}}}$ is either $\omega_p|_{G_{\QQ_{\ell}}}$ or $\omega_p^{-1}|_{G_{\QQ_{\ell}}}$.
\item If $p \nmid \ell-1$ and $\ell \mid N_0$, then $v_{\ell}(N) = v_{\ell}(N_0)$.
\item If $p \mid \ell-1$ and $\ell \nmid N_0$, then $v_{\ell}(N) \leq 2$.
\item  If $p \mid \ell-1$ and $\ell \mid N_0$, then $v_{\ell}(N/N_0) \leq 1$.
\item  If $p \mid \ell-1$, $\ell \mid N_0$ and both $\chibar_1$ and $\chibar_2$ are ramified at $\ell$, then $v_{\ell}(N) = v_{\ell}(N_0)$.
\item If $p \mid \ell+1$ and $\ell \nmid N_0$, then $v_{\ell}(N) \leq 2$ and $\chibar|_{G_{\QQ_{\ell}}} = \omega_p|_{G_{\QQ_{\ell}}}$.
%\item  If $p \mid \ell+1$ and $\chibar|_{G_{\QQ_{\ell}}} \neq \omega_p|_{G_{\QQ_{\ell}}}$, then $v_{\ell}(N)=v_{\ell}(N_0)$.
\end{enumerate}
\end{prop}
\begin{proof}
By Proposition~\ref{condprop}, the tame Artin conductor of $\rho_f$ is $N$.
As $\rho_f$ lifts $\rhob_0$ and $N_0$ is the tame Artin conductor of $\rhob_0$, $N_0 \mid N$.
The proposition now follows directly from  \cite[Proposition $2$]{Ca} and \cite[Remark 1.5]{Ca}.
\end{proof}

As a corollary, we get:
\begin{cor} 
\label{condlem}
Let $N$ be an integer such that $N \neq 1$, $N_0 \mid N$ and $\dim(H^1(G_{\QQ,Np},\chibar)) = 1$. 
For every prime $\ell \mid N$ such that $p \mid \ell-1$, suppose the following holds:
\begin{enumerate}
\item\label{hyp1} If $\ell \nmid N_0$, then $\ell^2 \mid N$. 
\item\label{hyp2} If $\ell \mid N_0$ and either $\chibar_1$ or $\chibar_2$ is unramified at $\ell$, then $\ell \mid \frac{N}{N_0}$.
\end{enumerate}
Suppose there exists a newform $f$ such that $\rho_f$ lifts $\rhob_0$ and let the tame level of $f$ be $N'$.
If every prime divisor of $N'$ is also a prime divisor of $N$, then $N' \mid N$.
\end{cor}
\begin{proof}
%Since $N \neq 1$, it follows that $\chibar \neq \omega_p$  as the dimension of $H^1(G_{\QQ,Np},\omega_p)$ is the number of distinct primes dividing $Np$ (see proof of \cite[Proposition $24$]{D} for more details).
Since $\dim(H^1(G_{\QQ,Np},\chibar)) = 1$, it follows, from Lemma~\ref{dimonelem}, that if $\ell \mid N$, then $\chibar|_{G_{\QQ_\ell}} \neq \omega_p|_{G_{\QQ_\ell}}$.
The corollary now follows directly from Proposition~\ref{carayolprop}.
%In particular, if $\ell \mid N$ and $p \mid \ell^2-1$, then $\chibar|_{G_{\QQ_\ell}} \neq \omega_p|_{G_{\QQ_\ell}}$ and $\chibar^{-1}|_{G_{\QQ_\ell}} \neq \omega_p|_{G_{\QQ_\ell}}$.
%
%Since $f$ is a newform of tame level $N'$, the Artin conductor of $\rho_f$ is $N'$. Suppose $N' \neq N_0$. 
%Let $S_1$ be the set of primes $\ell$ such that $\ell \mid N'$, $\ell \nmid N_0$ and $p \nmid \ell^2-1$. 
%Let $S_2$ be the set of primes $\ell$ such that $\ell \mid N_0$, $p \mid \ell-1$ and either $\chibar_1$ or $\chibar_2$ is unramified at $\ell$. 
%Let $S_3$ be the set of primes $\ell$ such that $\ell \nmid N_0$, $\ell \mid N'$ and $p \mid \ell-1$. 
%Then, combining previous paragraph with \cite[Proposition $2$]{Ca}, we get that $\frac{N'}{N_0} \mid \prod_{\ell' \in S_1 \cup  S_2} \ell' \prod_{\ell'' \in S_3}\ell''^2$.
%
%Since we are assuming that all prime factors of $N'$ divide $N$, $\prod_{\ell' \in S_1 \cup  S_2} \ell' \prod_{\ell'' \in S_3}\ell'' \mid N$.
%Then hypothesis~\eqref{hyp2} implies that $\prod_{\ell' \in S_1 \cup  S_2} \ell' \mid \frac{N}{N_0}$ and hypothesis~\eqref{hyp1} implies that $\prod_{\ell'' \in S_3}\ell''^2 \mid \frac{N}{N_0}$. Since $S_1$, $S_2$ and $S_3$ are disjoint sets, it follows that $\prod_{\ell' \in S_1 \cup  S_2} \ell' \prod_{\ell'' \in S_3}\ell''^2 \mid \frac{N}{N_0}$.
%Therefore, we get that $\frac{N'}{N_0} \mid \frac{N}{N_0}$ and hence, $N' \mid N$.
\end{proof}

We now recall some results about deformation rings.
We begin with a presentation result for deformation rings which was proved by Mazur and B\"{o}ckle.

\begin{prop}
\label{bocprop}
%Let $\rhob_0 : G_{\QQ,Np} \to \GL_2(\FF)$ be an odd, reducible, semi-simple representation as given in Setup~\ref{basic}.
Let $\chibar_1, \chibar_2 : G_{\QQ,Np} \to \FF^{\times}$  be distinct characters such that $\chibar_1 \oplus \chibar_2$ is an odd representation and let $\chibar =\chibar_1\chibar_2^{-1}$.
Let $x \in H^1(G_{\QQ,Np},\chibar)$ be a non-zero element and $\rhob_x$ be a representation given by $\begin{pmatrix} \chibar_1 & * \\ 0 & \chibar_2 \end{pmatrix}$, where $*$ corresponds to $x$.
Let $R_{\rhob_x}$ be the universal deformation ring of $\rhob_x$ in $\mathcal{C}$.
Let $d = \dim(H^1(G_{\QQ,Np},\text{Ad}(\rhob_x)))$ and $d' = \dim(H^2(G_{\QQ,Np},\text{Ad}(\rhob_x)))$.
Then there exists a surjective map $$F : W(\FF)\llbracket X_1,\cdots,X_d\rrbracket \to R_{\rhob_x}$$ such that $\ker(F)$ is generated by at most $d'$ elements.
\end{prop}
\begin{proof}
This is \cite[Theorem 2.4]{Bo2}.
\end{proof}

A presentation result for deformation rings, similar to Proposition~\ref{bocprop}, was first proved by Mazur (see \cite[Proposition 2]{M} and its proof).
However, Mazur proved it only for mod $p$ universal deformation rings of absolutely irreducible representations.
So Proposition~\ref{bocprop} is a generalization of Mazur's result and it was proved by B\"{o}ckle in \cite{Bo2} using ideas from the proof of Mazur's result.
%Using ideas from the proof of \cite[Proposition 2]{M}, B\"{o}ckle proved Proposition~\ref{bocprop} in \cite{Bo2}.

We now recall some results about universal deformation rings of characteristic $0$ Galois representations.
Let $K$ be a finite extension of $\QQ_p$ and $\rho : G_{\QQ,Np} \to \GL_2(K)$ be a representation.
If the only $G_{\QQ,Np}$-linear endomorphisms of $\rho$ are scalars, then the existence of the universal deformation ring of $\rho$ (in the sense of Mazur) has been proved by Kisin (\cite[Lemma 9.3]{Ki2}).
To be precise, let $\mathcal{D}$ be the category of local Artinian algebras with residue field $K$.
Let $\mathcal{H}$ be the functor from $\mathcal{D}$ to the category of sets which sends an object $R$ of $\mathcal{D}$ to the set of deformations (i.e. equivalence classes of lifts) of $\rho$ to $\GL_2(R)$.
Then Kisin has proved the following:

\begin{lem}
\label{deformlem}
If all $G_{\QQ,Np}$-linear endomorphisms of $\rho$ are scalars, then the functor $\mathcal{H}$ is pro-representable by a complete local Noetherian ring $R_{\rho}$ with residue field $K$.
\end{lem}
\begin{proof}
This is \cite[Lemma 9.3]{Ki2}.
\end{proof}

The ring $R_{\rho}$ obtained in Lemma~\ref{deformlem} above is known as the universal deformation ring of $\rho$.
We also have an analogue of Proposition~\ref{bocprop} for $R_{\rho}$.
\begin{lem}
\label{preslem}
Suppose $\rho : G_{\QQ,Np} \to \GL_2(K)$ is a representation such that all its $G_{\QQ,Np}$-linear endomorphisms are scalars.
Let $d = \dim(H^1(G_{\QQ,Np},\text{Ad}(\rho)))$ and $d' = \dim(H^2(G_{\QQ,Np},\text{Ad}(\rho)))$.
Then there exists a surjective map $$F : K\llbracket X_1,\cdots,X_d\rrbracket \to R_{\rho}$$ such that $\ker(F)$ is generated by at most $d'$ elements.
\end{lem}
\begin{proof}
This follows directly from the proof of \cite[Corollary 9.8]{Ki2} which uses ideas from the proof of Mazur's result (\cite[Proposition 2]{M}).
\end{proof}

%Let $\rhob_x : G_{\QQ,Np} \to \GL_2(\FF)$ be a \emph{non-split reducible odd} representation of the form $\begin{pmatrix} \chibar_1 & * \\ 0 & \chibar_2\end{pmatrix}$.
%As $p$ is odd and $\rhob_x$ is odd, we get that $\chibar_1$ and  $\chibar_2$ are distinct characters.
%Let $R_{\rhob_x}$ be the universal deformation ring of $\rhob_x$ in $\mathcal{C}$ and $\rho^{\univ} : G_{\QQ,Np} \to \GL_2(R_{\rhob_x})$ be the universal deformation of $\rhob_x$.
Let $\rhob_x$ be the representation defined in Definition~\ref{defodef2} and let $R_{\rhob_x}$ be the universal deformation ring of $\rhob_x$ in $\mathcal{C}$.

Let $P$ be a prime of $R_{\rhob_x}$ such that $R_{\rhob_x}/P$ is a finite extension of $W(\FF)$. 
As $R_{\rhob_x}/P$ is also an integral domain, there exists a finite extension $K_P$ of $\QQ_p$ such that $R_{\rhob_x}/P$ is isomorphic to a subring of $K_P$.
We fix such an isomorphism and let $\rho_P : G_{\QQ,Np} \to \GL_2(K_P)$ be the representation obtained by composing $\rho^{\univ} \pmod{P}$ with this isomorphism.
Note that $\rho_P$ is odd as $\rhob_x$ is odd.
Suppose $\rho_P$ is either absolutely irreducible or a non-trivial extension of two distinct characters.
Then, by Lemma~\ref{deformlem}, the universal deformation ring $R_{\rho_P}$ of $\rho_P$ exists.

Let $(R_{\rhob_x})_P$ be the localization of $R_{\rhob_x}$ at $P$ and $\widehat{(R_{\rhob_x})_P}$ be the completion of $(R_{\rhob_x})_P$ with respect to its maximal ideal.
We now recall a result of Kisin relating $R_{\rho_P}$ and $R_{\rhob_x}$.

\begin{prop}
\label{kisinprop}
Under the notation established above, the universal deformation ring $R_{\rho_P}$ of $\rho_P$ (when it exists) is isomorphic to $\widehat{(R_{\rhob_x})_P}$ and the Krull dimension of $R_{\rho_P}$ is at least $3$.
\end{prop}
\begin{proof}
The isomorphism between $R_{\rho_P}$ and $\widehat{(R_{\rhob_x})_P}$ is given in \cite[Proposition 9.5]{Ki2}. The lower bound on the Krull dimension of $R_{\rho_P}$ follows directly from oddness of $\rho_P$ and \cite[Corollary 9.8]{Ki2}.
\end{proof}

Now we will recall a result of Kisin (\cite{Ki2}) and Weston (\cite{We1}) about unobstructedness of $\rho_P$ when $\rho_P$ arises from a modular form.
We keep the notation established above. Let $\chibar=\chibar_1\chibar_2^{-1}$.

\begin{thm}
\label{unobsthm}
Suppose $\rho_P$ is the $p$-adic Galois representation attached to a newform $f$. Suppose one of the following hypothesis hold:
\begin{enumerate}
\item\label{ek} $\chibar \neq \omega_p^k$ for any integer $k$.
\item\label{don} $f$ is not a CM modular form, $p$ does not divide the level of $f$ and $\rho_P|_{G_{\QQ_{\ell}}}$ is special for all primes $\ell$ dividing the level of $f$.
\end{enumerate}
Then $R_{\rho_P} \simeq K \llbracket T_1, T_2, T_3 \rrbracket$.
\end{thm}
\begin{proof}
%From Lemma~\ref{preslem}, we see that $R_{\rho_P}$ is isomorphic to a power series ring if $H^2(G_{\QQ,Np},\text{Ad}(\rho_P))=0$.
By Lemma~\ref{preslem}, we know that $R_{\rho_P} \simeq K \llbracket T_1, T_2, T_3 \rrbracket$ if $\dim(H^1(G_{\QQ,Np}, \text{Ad}(\rho_P))) =3$ and $H^2(G_{\QQ,Np},\text{Ad}(\rho_P))=0$.
As $\rho_P$ is odd, the global Euler characteristic formula (\cite[Lemma 9.7]{Ki2}) implies that $H^2(G_{\QQ,Np},\text{Ad}(\rho_P))=0$ if and only if $\dim(H^1(G_{\QQ,Np}, \text{Ad}(\rho_P))) =3$.
%Thus Lemma~\ref{preslem} implies that $R_{\rho_P} \simeq K \llbracket T_1, T_2, T_3 \rrbracket$ if $\dim(H^1(G_{\QQ,Np}, \text{Ad}(\rho_P))) =3$.
Hence, it suffices to prove that $\dim(H^1(G_{\QQ,Np}, \text{Ad}(\rho_P))) =3$.

In the case of Part~\eqref{ek}, we know this from part $(2)$ of \cite[Theorem 8.2]{Ki}.
Moreover, from the proof of \cite[Theorem 8.2]{Ki}, it follows that $\dim(H^1(G_{\QQ,Np}, \text{Ad}(\rho_P))) =3$ if a certain Selmer group (which is also a subgroup of $H^1(G_{\QQ,Np}, \text{Ad}(\rho_P))$) vanishes.
Now if $f$ satisfies the conditions of Part~\eqref{don}, then Weston (\cite{We1}) has proved that this Selmer group vanishes (see \cite[Theorem 1]{We1} and the proof of \cite[Theorem 11.10]{Ki2}).
This completes the proof of the theorem.
\end{proof}

We now record a well-known basic result about $p$-adic Galois representations arising from Eisenstein series.

\begin{lem}
\label{eisenlem}
Let $K$ be a finite extension of $\QQ_p$ and $\mathcal{O}$ be the ring of integers of $K$.
Let $\chi_1,\chi_2 : G_{\QQ,Np} \to \mathcal{O}^{\times}$ be characters of finite order.
Let $N$ be the product of the conductors of $\chi_1$ and $\chi_2$.
If $k \geq 1$ is an integer such that $\rho := \chi_1 \oplus \chi_2\chi_p^{k-1}$ is an odd representation of $G_{\QQ,Np}$, then there exists an Eisenstein series $E$ of level $N$ and weight $k$ such that $\rho = \rho_E$.
\end{lem}
\begin{proof}
The lemma follows immediately from a standard construction of Eisenstein series of level $N$ and weight $k$ using the characters $\chi_1$ and $\chi_2$ and the definition of the $p$-adic Galois representation attached to an Eisenstein series.
\end{proof}

We finish this section by recalling the generalization of Krull's principal ideal theorem.
\begin{thm}
\label{krullthm}
Suppose $R$ is a Noetherian ring of Krull dimension $\kappa$ and $I$ is an ideal of $R$ generated by $a_1,\cdots,a_d \in R$. Then the Krull dimension of $R/I$ is at least $\kappa - d$.
\end{thm}
\begin{proof}
 This is \cite[Theorem $10.2$]{E}.
\end{proof}

 %Let $\rhob'_x = \rhob_x \otimes \chibar_2^{-1}$. 
%Hence, $\rhob'_x(g) = \begin{pmatrix} \chibar(g) & x(g)\\ 0 & 1 \end{pmatrix}$ and 
%Define deformation problems of pseudorepn and representations.
\subsection{Some miscellaneous results}
\label{miscres}
%We begin with a result from Galois cohomology.
We begin with a result giving necessary and sufficient conditions for $\dim(H^1(G_{\QQ,Np},\chibar))$ to be $1$. 
Let $K_{\chibar}$ be the extension of $\QQ$ fixed by $\ker(\chibar^{-1}\omega_p)$.
So $K_{\chibar}$ is a subfield of $F_{\chibar}(\zeta_p)$, where $\zeta_p$ is a primitive $p$-th root of unity in $\overline{\QQ}$.
Recall that we denote the class group of $K_{\chibar}$ by $\text{Cl}(K_{\chibar})$. Note that $\text{Gal}(K_{\chibar}/\QQ)$ acts on the $\FF$-vector space $\text{Cl}(K_{\chibar})/\text{Cl}(K_{\chibar})^p \otimes_{\FF_p} \FF$. 
Given a character $\eta : \text{Gal}(K_{\chibar}/\QQ) \to \FF^{\times}$, denote by $(\text{Cl}(K_{\chibar})/\text{Cl}(K_{\chibar})^p)[\eta]$ the subspace of $\text{Cl}(K_{\chibar})/\text{Cl}(K_{\chibar})^p \otimes_{\FF_p} \FF$ on which $\text{Gal}(K_{\chibar}/\QQ)$ acts via the character $\eta$.

\begin{lem}
\label{dimonelem}
We have $\dim(H^1(G_{\QQ,Np},\chibar)) =1$ if and only if the following conditions hold:
\begin{enumerate}
\item $\chibar|_{G_{\QQ_{\ell}}} \neq \omega_p|_{G_{\QQ_{\ell}}} $ for all primes $\ell \mid Np$,
\item $(\text{Cl}(K_{\chibar})/\text{Cl}(K_{\chibar})^p)[\omega_p\chibar^{-1}] = 0.$
\end{enumerate}
Otherwise, $\dim(H^1(G_{\QQ,Np},\chibar)) \geq 2$
\end{lem}
\begin{proof}
As $\chibar$ is odd, the global Euler characteristic formula implies that $H^1(G_{\QQ,Np},\chibar) \neq 0$.
Therefore, if $\dim(H^1(G_{\QQ,Np},\chibar)) \neq 1$, then $\dim(H^1(G_{\QQ,Np},\chibar)) \geq 2$.
So we will now focus on the case when $\dim(H^1(G_{\QQ,Np},\chibar)) =1$.

By the Greenberg--Wiles formula (\cite[Theorem $2$]{Wa}) and the local Euler characteristic formula, we have 
\begin{multline}
\dim(H^1(G_{\QQ,Np},\chibar)) = \dim(H^1_0(G_{\QQ,Np}, \omega_p\chibar^{-1}))+ 1 + \dim(H^0(G_{\QQ_p}, \omega_p\chibar^{-1}|_{G_{\QQ_p}})) \\ +\sum_{\ell \mid N}\dim(H^0(G_{\QQ_\ell}, \omega_p\chibar^{-1}|_{G_{\QQ_\ell}})),
\end{multline}
where $$H^1_0(G_{\QQ,Np}, \omega_p\chibar^{-1}) = \ker(H^1(G_{\QQ,Np},\omega_p\chibar^{-1}) \to H^1(G_{\QQ_p},\omega_p\chibar^{-1}|_{G_{\QQ_p}}) \times \prod_{\ell \mid N} H^1(G_{\QQ_\ell},\omega_p\chibar^{-1}|_{G_{\QQ_\ell}})).$$
Hence, $\dim(H^1(G_{\QQ,Np},\chibar)) = 1$ if and only if the following conditions hold:
\begin{enumerate}
\item $H^0(G_{\QQ_\ell}, \omega_p\chibar^{-1}|_{G_{\QQ_\ell}})=0$ i.e. $\chibar|_{G_{\QQ_{\ell}}} \neq \omega_p|_{G_{\QQ_{\ell}}} $ for all primes $\ell \mid Np$,
\item $H^1_0(G_{\QQ,Np}, \omega_p\chibar^{-1})=0.$
\end{enumerate}
By the first condition above, we know that if $\ell \mid Np$, then $\omega_p\chibar^{-1}|_{G_{\QQ_\ell}} \neq 1$.
Therefore, for all primes $\ell \mid Np$, $\ker(H^1(G_{\QQ_{\ell}},\omega_p\chibar^{-1}|_{G_{\QQ_\ell}}) \to H^1(I_{\ell},\omega_p\chibar^{-1}|_{I_\ell})) =0.$
Indeed, if we denote the representation $\omega_p\chibar^{-1}|_{G_{\QQ_\ell}}$ of $G_{\QQ_{\ell}}$ by $V_{\ell}$ and the subspace of $V_{\ell}$ fixed by $I_{\ell}$ by $(V_{\ell})^{I_{\ell}}$, then the inflation-restriction sequence implies that $$ H^1(G_{\QQ_{\ell}}/I_{\ell}, (V_{\ell})^{I_{\ell}}) = \ker(H^1(G_{\QQ_{\ell}},\omega_p\chibar^{-1}|_{G_{\QQ_\ell}}) \to H^1(I_{\ell},\omega_p\chibar^{-1}|_{I_\ell})). $$
However, from \cite[Lemma 1]{Wa}, we know that $|H^1(G_{\QQ_{\ell}}/I_{\ell}, (V_{\ell})^{I_{\ell}})| = |H^0(G_{\QQ_{\ell}},V_{\ell})|$.
Since $\omega_p\chibar^{-1}|_{G_{\QQ_\ell}} \neq 1$, $H^0(G_{\QQ_{\ell}},V_{\ell}) = 0$ and hence, $H^1(G_{\QQ_{\ell}}/I_{\ell}, (V_{\ell})^{I_{\ell}}) =0$.
This proves our claim.

Hence, we have $$H^1_0(G_{\QQ,Np}, \omega_p\chibar^{-1}) = \ker(H^1(G_{\QQ,Np},\omega_p\chibar^{-1}) \to H^1(I_p,\omega_p\chibar^{-1}|_{I_p}) \times \prod_{\ell \mid N} H^1(I_{\ell},\omega_p\chibar^{-1}|_{I_{\ell}})).$$
So $H^1_0(G_{\QQ,Np},\omega_p\chibar^{-1})$ classifies isomorphism classes of extensions $$0 \to \omega_p\chibar^{-1} \to V \to 1 \to 0$$ of $G_{\QQ}$-representations which are unramified everywhere.

Hence, a non-zero element of $H^1_0(G_{\QQ,Np},\omega_p\chibar^{-1})$ gives a Galois extension $K$ of $\QQ$ such that $K_{\chibar} \subset K$, $K$ is an everywhere unramified extension of $K_{\chibar}$, $\text{Gal}(K/K_{\chibar})$ is isomorphic to an $\FF_p$-vector subspace of $\FF$ and $\text{Gal}(K_{\chibar}/\QQ)$ acts on $\text{Gal}(K/K_{\chibar})$ via the character $\omega_p\chibar^{-1}$.
On the other hand, a Galois extension $K$ of $\QQ$ satisfying all the properties mentioned above gives a non-zero element of $H^1_0(G_{\QQ,Np},\omega_p\chibar^{-1})$.

Therefore, by Hilbert class field theory, it follows that if the first condition above holds, then $H^1_0(G_{\QQ,Np}, \omega_p\chibar^{-1})  =0$ if and only if $(\text{Cl}(K_{\chibar})/\text{Cl}(K_{\chibar})^p)[\omega_p\chibar^{-1}]=0$. This proves the lemma.
\end{proof}

\textbf{Examples:} In particular, Lemma~\ref{dimonelem} implies that if $p \nmid |\text{Cl}(K_{\chibar})|$ and $\chibar|_{G_{\QQ_{\ell}}} \neq \omega_p|_{G_{\QQ_{\ell}}} $ for all primes $\ell \mid Np$, then $\dim(H^1(G_{\QQ,Np},\chibar)) =1$. 

Now suppose $\chibar=\omega_p^{k-1}$ with $k$ even and $k \not\equiv 0,2 \pmod{p-1}$.
In this case, we know, by reflection principle (\cite[Theorem 10.9]{Wash}), that $(\text{Cl}(\QQ(\zeta_p))/\text{Cl}(\QQ(\zeta_p))^p)[\omega_p^{2-k}]=0$ if $(\text{Cl}(\QQ(\zeta_p))/\text{Cl}(\QQ(\zeta_p))^p)[\omega_p^{k-1}]=0$.
By Herbrand--Ribet theorem, we know that $(\text{Cl}(\QQ(\zeta_p))/\text{Cl}(\QQ(\zeta_p))^p)[\omega_p^{k-1}]=0$ if and only if $p \nmid B_{p+1-k}$.

On the other hand, Kurihara has proved that $(\text{Cl}(\QQ(\zeta_p))/\text{Cl}(\QQ(\zeta_p))^p)[\omega_p^{p-3}]=0$ (see \cite[Corollary 3.8]{K}).
 It follows, from \cite[Corollary 7.1]{Kup}, that $(\text{Cl}(\QQ(\zeta_p))/\text{Cl}(\QQ(\zeta_p))^p)[\omega_p^{p-5}]=0$.
Finally, if $p >3$ and $p \equiv 3 \pmod{4}$, then Osburn has proved that $(\text{Cl}(\QQ(\zeta_p))/\text{Cl}(\QQ(\zeta_p))^p)[\omega_p^{\frac{p+1}{2}}]=0$ (see \cite[Theorem 1.1]{O}). 
Using this we get:
\begin{enumerate}
\item\label{eg1} If $2 <k < p-1$ is an even integer and $\ell_1,\cdots,\ell_r$ are distinct primes such that $p \nmid \prod_{i=1}^{r}\ell_i$ and $p \nmid \ell_i^{k-2}-1$ for all $1 \leq i \leq r$, then $$\dim(H^1(G_{\QQ,p\prod_{i=1}^{r}\ell_i}, \omega_p^{k-1})) =1$$ if one of the following conditions hold:
\begin{itemize}
\item $p>3$, $p \equiv 3 \pmod{4}$ and $k=\dfrac{p+1}{2}$.
\item $p>5$ and $k=4,6$.
\item $p \nmid B_{p+1-k}$.
\item $p$ is a regular prime.
\item Vandiver's conjecture holds for $p$.
\end{itemize}
\item More generally, if $N$ is the tame Artin conductor of $\chibar$ (which means $\chibar$ is ramified at all primes $\ell \mid N$), $p \nmid |\text{Cl}(K_{\chibar})|$ and $\ell_1,\cdots,\ell_r$ are distinct primes such that $\chibar|_{G_{\QQ_{\ell_i}}} \neq \omega_p|_{G_{\QQ_{\ell_i}}}$ for all $1 \leq i \leq r$ then
$$\dim(H^1(G_{\QQ,Np},\chibar)) = \dim(H^1(G_{\QQ,N(\prod_{i=1}^{r}\ell_i)p},\chibar)) =1.$$
Given a prime $p$, we conclude, using Chebotarev's density theorem, that $\{\ell \mid \ell \text{ is a prime, } \chibar|_{G_{\QQ_{\ell}}} \neq \omega_p|_{G_{\QQ_{\ell}}}\}$ is an infinite set with positive Dirichlet density.
\end{enumerate}

On the other hand, we have:
\begin{lem}
\label{cyclem}
Let $k>2$ be an even integer. If $\dim(H^1(G_{\QQ,Np},\omega_p^{k-1})) = 1$ and $\ell \mid N$, then $p \nmid \ell^2-1$. 
\end{lem}
\begin{proof}
Since $\dim(H^1(G_{\QQ,Np},\omega_p^{k-1})) = 1$, by Lemma~\ref{dimonelem}, we see that $\omega_p^{k-2}|_{G_{\QQ_{\ell}}} \neq 1$ for all primes $\ell \mid N$. On the other hand, as $k$ is even, $\omega_p^{k-2}|_{G_{\QQ_\ell}} = 1$ if $p \mid \ell^2-1$.
Combining all these observations, we get that if $\ell \mid N$, then $p \nmid \ell^2-1$.
\end{proof}

%We will now prove some results about levels of modular newforms lifting $\rhob_0$.
Recall that we denoted the tame Artin conductor of $\rhob_0$ by $N_0$. We will now prove that $\rhob_0$ arises from a modular eigenform of level $N_0$.
\begin{lem}
\label{modlem}
There exists a modular eigenform $f$ of level $N_0$ such that $\rho_f$ lifts $\rhob_0$.
\end{lem}
\begin{proof}
Note that, for $i = 1,2$, there exists an integer $0 \leq a_i \leq p-2$ and a character $\epsilon_i : G_{\QQ,Np} \to \FF^{\times}$ unramified at $p$ such that $\chibar_i = \epsilon_i\omega_p^{a_i}$.
So $\rhob_0 = \epsilon_1\omega_p^{a_1} \oplus \epsilon_2\omega_p^{a_2}$.
Without loss of generality assume $a_1 \leq a_2$ and let $k_0 -1 = a_2-a_1$.
Let $\rhob'_0 = \epsilon_1 \oplus \epsilon_2 \omega_p^{k_0-1}$.
Let $k >2$ be an integer such that $k \equiv k_0 \Mod{p-1}$.
By Lemma~\ref{eisenlem}, the representation $\rho : G_{\QQ,N_0p} \to \GL_2(W(\FF))$ given by $\rho := \hat\epsilon_1 \oplus \hat\epsilon_2\chi_p^{k-1}$ is attached to an Eisenstein series $E$ of level $N_0$ and weight $k$. Hence, $\rho_E$ lifts $\rhob'_0$.

Let $\bar{E}$ be the mod $p$ modular form (in the sense of Serre and Swinnerton-Dyer) corresponding to $E$. Since the $q$-expansion of $E$ belongs to $W(\FF)[[q]]$, $\bar{E}$ is the modular form of weight $k$ and level $N_0$ over $\FF$ whose $q$-expansion is the image of the $q$-expansion of $E$ under the natural surjective map $W(\FF)[[q]] \to \FF[[q]]$. Hence, $\bar{E}$ is an eigenform for all Hecke operators away from $N_0p$.

So $\theta^{a_1}(\bar{E})$ is also a modular eigenform over $\FF$ of level $N_0$ and weight $k + (p+1)a_1$, where $\theta$ is the Ramanujan theta operator.
Since $k >2$, we can use the Deligne-Serre lifting lemma to conclude that there exists an eigenform $f$ of weight $k + (p+1)a_1$ and level $N_0$ lifting $\theta^{a_1}(\bar{E})$. 
In other words, the reduction of $T_{\ell}$-eigenvalue of $f$ modulo $\varpi_f$ is the $T_{\ell}$-eigenvalue of $\theta^{a_1}(\bar{E})$ for all primes $\ell \nmid Np$ (see \S\ref{notsec} for the notation $\varpi_f$). 
%In other words, there exists an eigenform $f$ of weight $k + (p+1)a_1$ and level $N_0$ such that if $E_f$ is the finite extension of $\QQ_p$ generated by Hecke eigenvalues of $f$ over $\QQ_p$, $\mathcal{O}_f$ is the ring of integers of $E_f$ and $\pi_f$ is its uniformizer, then reduction of $T_{\ell}$ eigenvalue of $f$ modulo $\pi_f$ is the $T_{\ell}$ eigenvalue of $\theta^{a_1}(\bar{E})$ for all primes $\ell \nmid Np$. 
From the effect of $\theta$ on $q$-expansions, we conclude that $\rho_f$ lifts $\rhob_0$.
\end{proof}

Before proceeding further, we introduce some more notation.
Let $E$ be a finite extension of $\QQ_p$ such that the residue field of $E$ contains $\FF$. Let $\mathcal{O}_E$ be the ring of integers of $E$ and $\varpi_E$ be its uniformizer.
We say that a character $\Psi : G_{\QQ,Np} \to E^{\times}$ lifts $\bar\psi$ if $\Psi$ takes values in $\mathcal{O}_E^{\times}$ and $\Psi \Mod{\varpi_E} = \bar\psi$.
We now obtain a family of reducible, non-split $p$-adic representations of $G_{\QQ,Np}$ which are unobstructed. It will be used in the proof of Proposition~\ref{modprop}.

\begin{prop}
\label{unobsprop}
Let $\bar\psi$ be the character of $G_{\QQ,Np}$ unramified at $p$ such that $\chibar = \bar\psi.\omega_p^{k_0-1}$ for some integer $1 \leq k_0 \leq p-1$.
\begin{enumerate}
\item\label{item1} There exist infinitely many integers $k > 2$ such that $k \equiv k_0 \Mod{p-1}$ and $$\dim_E(H^1(G_{\QQ,Np}, \Psi\chi_p^{k-1})) = \dim_E(H^1(G_{\QQ,Np}, \Psi^{-1}\chi_p^{1-k})) = 1$$ for any lift $\Psi : G_{\QQ,Np} \to E^{\times}$ of $\bar\psi$ unramified at $p$.
\item Let $\Psi : G_{\QQ,Np} \to E^{\times}$ be a lift of $\bar\psi$ which is unramified at $p$. Let $k$ be an integer satisfying the conclusion of part \eqref{item1}. Let $y \in H^1(G_{\QQ,Np}, \Psi\chi_p^{k-1})$ be a non-zero element and $\rho_y : G_{\QQ,Np} \to \GL_2(E)$ be the representation corresponding to $y$. Then $\dim_E(H^1(G_{\QQ,Np},\text{Ad}(\rho_y)))=3$.
\end{enumerate}
\end{prop}
\begin{proof}
 Let $\Psi : G_{\QQ,Np} \to E^{\times}$ be a lift of $\bar\psi$ which is unramified at $p$. Let $F_{\Psi}$ be the extension of $\QQ$ given by the compositum of $F_{\chibar}$ and the fixed field of $\ker(\Psi)$. So, $F_{\Psi}$ is a finite abelian extension of $\QQ$. By abuse of notation, denote $\Psi\chi_p^{k-1}|_{G_{\QQ_q}}$ (resp. $\Psi^{-1}\chi_p^{1-k}|_{G_{\QQ_q}}$)  by $\Psi\chi_p^{k-1}$ (resp. $\Psi^{-1}\chi_p^{1-k}$) for all primes $q$ throughout this proof. Let $k > 2$ be an integer such that $k \equiv k_0 \Mod{p-1}$ and let $$H^1_0(G_{\QQ,Np},\Psi\chi_p^{k-1}) := \ker(H^1(G_{\QQ,Np},\Psi\chi_p^{k-1}) \to H^1(G_{\QQ_p},\Psi\chi_p^{k-1}) \times \prod_{\ell \mid N} H^1(G_{\QQ_\ell},\Psi\chi_p^{k-1})).$$

Let $F_{\Psi,\infty}$ be the cyclotomic $\ZZ_p$-extension of $F_{\Psi}$.
A non-zero element of $H^1_0(G_{\QQ,Np},\Psi\chi_p^{k-1})$ gives an abelian, unramified $\ZZ_p^r$-extension of $F_{\Psi,\infty}$ on which $\text{Gal}(F_{\Psi,\infty}/F_{\Psi})$ acts via $\chi_p^{k-1}$.
Let $K$ be the maximal abelian unramified pro-$p$ extension of $F_{\Psi,\infty}$.
Note that $K$ is Galois over $F_{\Psi}$.
Moreover, as $\text{Gal}(K/F_{\Psi,\infty})$ is abelian, we get an action of $\text{Gal}(F_{\Psi,\infty}/F_{\Psi})$ on $\text{Gal}(K/F_{\Psi,\infty})$ by conjugation.
Therefore, if $H^1_0(G_{\QQ,Np},\Psi\chi_p^{k-1}) \neq 0$, then we get a quotient of the $\text{Gal}(F_{\Psi,\infty}/F_{\Psi})$-representation $\text{Gal}(K/F_{\Psi,\infty}) \otimes_{\ZZ_p} \QQ_p$ on which $\text{Gal}(F_{\Psi,\infty}/F_{\Psi})$ acts via $\chi_p^{k-1}$.
%Observe that if $k \neq k'$, then the quotients of $\text{Gal}(K/F_{\Psi,\infty}) \otimes_{\ZZ_p} \QQ_p$ arising from $H^1_0(G_{\QQ,Np},\Psi\chi_p^{k-1})$ are not isomorphic to those arising from $H^1_0(G_{\QQ,Np},\Psi\chi_p^{k'-1})$.

Note that $F_{\Psi}$ is an abelian CM extension of $\QQ$. We know, by work of Ferrero--Washington (main theorem of \cite{FW}), that the $\mu$-invariant of $F_{\Psi}$ is $0$.
Hence, $\text{Gal}(K/F_{\Psi,\infty})$ has finite $\ZZ_p$-rank i.e. $\text{Gal}(K/F_{\Psi,\infty}) \otimes_{\ZZ_p} \QQ_p$ is a finite dimensional $\QQ_p$-vector space.
Thus there are only finitely many \emph{distinct} characters $\chi$ of $\text{Gal}(F_{\Psi,\infty}/F)$ such that $\QQ_p(\chi)$ is a quotient of $\text{Gal}(K/F_{\Psi,\infty}) \otimes_{\ZZ_p} \QQ_p$ (as a $\text{Gal}(F_{\Psi,\infty}/F)$-representation).
%Thus only finitely many distinct characters of $\text{Gal}(F_{\Psi,\infty}/F)$ can arise as its $\text{Gal}(F_{\Psi,\infty}/F)$-stable quotients.
%Hence, if $K$ is the maximal abelian, unramified pro-$p$ extension of $F_{\Psi,\infty}$, then $\text{Gal}(K/F_{\Psi,\infty})$ has finite $\ZZ_p$-rank.
Therefore, $H^1_0(G_{\QQ,Np},\Psi\chi_p^{k-1})=0$ for all but finitely many $k \equiv k_0 \Mod{p-1}$.

We can also define $H^1_0(G_{\QQ,Np},\Psi^{-1}\chi_p^{1-k})$ similarly and the argument given above also proves that $H^1_0(G_{\QQ,Np},\Psi^{-1}\chi_p^{1-k})=0$ for all but finitely many $k \equiv k_0 \Mod{p-1}$.
Therefore, we get that for all but finitely many $k \equiv k_0 \Mod{p-1}$, $H^1_0(G_{\QQ,Np},\Psi\chi_p^{k-1}) = H^1_0(G_{\QQ,Np},\Psi^{-1}\chi_p^{1-k}) = 0$.

There are only finitely many characters $\Psi : G_{\QQ,Np} \to E^{\times}$ which lift $\bar\psi$ and which are unramified at $p$.
Hence, it follows that there are infinitely many integers $k$ such that $k>2$, $k \equiv k_0 \Mod{p-1}$ and $H^1_0(G_{\QQ,Np},\Psi\chi_p^{k-1}) = H^1_0(G_{\QQ,Np},\Psi^{-1}\chi_p^{1-k}) = 0$ for all lifts $\Psi$ of $\bar\psi$ which are unramified at $p$.

Suppose $k > 2$ and $H^1_0(G_{\QQ,Np},\Psi^{-1}\chi_p^{1-k})= H^1_0(G_{\QQ,Np},\Psi\chi_p^{k-1}) = 0$. As $k>2$, the local Euler characteristic formulas imply that 
$$\dim_E(H^1(G_{\QQ_p},\Psi^{-1}\chi_p^{1-k})) = \dim_E(H^1(G_{\QQ_p},\Psi\chi_p^{k-1})) = 1$$ and $H^1(G_{\QQ_\ell},\Psi^{-1}\chi_p^{1-k})= H^1(G_{\QQ_\ell},\Psi\chi_p^{k-1}) = 0$ for all primes $\ell \mid N$.
Hence, it follows that $$\dim_E(H^1(G_{\QQ,Np},\Psi^{-1}\chi_p^{1-k})) \leq 1 \text{ and } \dim_E(H^1(G_{\QQ,Np},\Psi\chi_p^{k-1})) \leq 1.$$ Since $\Psi\chi_p^{k-1}$ is an odd character, the global Euler characteristic formula implies that $$\dim_E(H^1(G_{\QQ,Np},\Psi^{-1}\chi_p^{1-k})) \geq 1 \text{ and }\dim_E(H^1(G_{\QQ,Np},\Psi\chi_p^{k-1})) \geq 1.$$ This finishes the proof the first part of the proposition.

 We now move to the second part of the proposition.
%Let $\Psi : G_{\QQ,Np} \to \overline{\QQ_p}^{\times}$ be a lift of $\bar\psi$ which is unramified at $p$, $k$ be an integer as found in part REF and $y \in H^1(G_{\QQ,Np}, \Psi\chi_p^{k-1})$ be a non-zero element. 
Note that $\rho_y$ is an odd representation. So, the global Euler characteristic formula implies that $\dim_E(H^1(G_{\QQ,Np},\text{Ad}(\rho_y))) - \dim_E(H^2(G_{\QQ,Np},\text{Ad}(\rho_y)))=3$. Hence, it suffices to prove that $H^2(G_{\QQ,Np},\text{Ad}(\rho_y)) = 0$. Note that $\text{Ad}(\rho_y) = 1 \oplus \text{Ad}^0(\rho_y)$, where $\text{Ad}^0(\rho_y)$ is the subspace of $\text{Ad}(\rho_y)$ given by matrices with trace $0$. As $H^2(G_{\QQ,Np},1) = 0$, it suffices to prove that $H^2(G_{\QQ,Np},\text{Ad}^0(\rho_y)) = 0$.

Let $V$ be the subspace of upper triangular matrices with trace $0$ in $\text{Ad}^0(\rho_y)$. It is clearly a $G_{\QQ,Np}$-sub-representation of $\text{Ad}^0(\rho_y)$. We have the following exact sequences of $G_{\QQ,Np}$ representations:
$$0 \to V \to \text{Ad}^0(\rho_y) \to E(\Psi^{-1}\chi_p^{1-k}) \to 0,$$
$$ 0 \to E(\Psi\chi_p^{k-1}) \to V \to E \to 0.$$
As $H^2(G_{\QQ,Np},1) = 0$, we conclude, from the exact sequences above, that $$\dim_E(H^2(G_{\QQ,Np}, \text{Ad}^0(\rho_y))) \leq \dim_E(H^2(G_{\QQ,Np}, \Psi\chi_p^{k-1})) + \dim_E(H^2(G_{\QQ,Np}, \Psi^{-1}\chi_p^{1-k})).$$
 We have chosen $k$ satisfying the conclusion of part \eqref{item1}. Hence, $\dim_E(H^1(G_{\QQ,Np}, \Psi\chi_p^{k-1})) = \dim_E(H^1(G_{\QQ,Np}, \Psi^{-1}\chi_p^{1-k})) = 1$. Since $\Psi\chi_p^{k-1}$ is an odd character, the global Euler characteristic formula implies that $H^2(G_{\QQ,Np}, \Psi\chi_p^{k-1}) = H^2(G_{\QQ,Np}, \Psi^{-1}\chi_p^{1-k}) = 0$. This implies that $H^2(G_{\QQ,Np}, \text{Ad}^0(\rho_y)) = 0$ which finishes the proof of the second part of the proposition.
\end{proof}

\section{Properties of $R_{\rhob_x}$}
\label{defosec}

In this section, we will study the structural properties of $R_{\rhob_x}$. In \S\ref{gmasec}, we begin with proving a result which gives a description of $R_{\rhob_x}[\rho^{\univ}(G_{\QQ_p})]$ in the form of a Generalized Matrix Algebra (GMA) over $R_{\rhob_x}$. We use this description to define a quotient $R_k$ of $R_{\rhob_x}$ for every integer $k \geq 2$. 
In \S\ref{finsec}, we prove finiteness of $R_k$ over $W(\FF)$ for every $k \geq 2$ and study its consequences for $R_{\rhob_x}$.
In \S\ref{modsec}, we define the notion of modular primes and prove the existence of smooth modular points in every component of $\spec(R_{\rhob_x})$.
These two results play a major role in establishing the main theorem.

We refer the reader to \cite[Chapter 1]{BC} and \cite[Section 2]{Bel} for the definition and basic properties of GMAs.
However, we will mostly be working with sub-$R$-GMAs of $M_2(R)$ for some ring $R$.
To be precise, given ideals $B$ and $C$ of a ring $R$, we define
$$\begin{pmatrix} R & B \\ C & R\end{pmatrix} := \Bigg\{\begin{pmatrix} a & b \\ c & d \end{pmatrix} \in M_2(R) \mid a,d \in R, b \in B, c \in C \Bigg\}.$$
Then we say that an $R$-algebra $A$ is a \emph{sub-$R$-GMA} of $M_2(R)$ if $A$ is an $R$-subalgebra of $M_2(R)$ and there exist ideals $B$ and $C$ of $R$ such that 
$$A = \begin{pmatrix} R & B \\ C & R\end{pmatrix} = \Bigg\{\begin{pmatrix} a & b \\ c & d \end{pmatrix} \in M_2(R) \mid a,d \in R, b \in B, c \in C \Bigg\}.$$
%If $B$ and $C$ are ideals of $R$, then we denote $\Bigg\{\begin{pmatrix} a & b \\ c & d \end{pmatrix} \in M_2(R) \mid a,d \in R, b \in B, c \in C \Bigg\} $ by $\begin{pmatrix} R & B \\ C & R\end{pmatrix}$.
%Hence, the equation above becomes $A = \begin{pmatrix} R & B \\ C & R\end{pmatrix}$.
%If $A$ is an sub-$R$-GMA of $M_2(R)$ and $B$ and $C$ are ideals of $R$ corresponding to it as above, then we say that $A = \begin{pmatrix} R & B \\ C & R\end{pmatrix}$.

\subsection{GMA description of $\rho^{\univ}(G_{\QQ_p})$}
\label{gmasec}
Let $g_0 \in G_{\QQ_p}$ be the element fixed in Definition~\ref{defodef2}. We now begin with a simple observation:
\begin{lem}
\label{diagonal}
There exists a $P \in Id + M_2(m_{R_{\rhob_x}})$ such that $P\rho^{\univ}(g_0)P^{-1} = \begin{pmatrix} a & 0\\ 0 & b \end{pmatrix}$ where $a, b \in R_{\rhob_x}$ are such that $a \equiv \chibar_1(g_0) \Mod{m_{R_{\rhob_x}}}$ and $b \equiv \chibar_2(g_0) \Mod{m_{R_{\rhob_x}}}$.
\end{lem}
\begin{proof}
The characteristic polynomial $f(X)$ of $\rho^{\univ}(g_0)$ reduces to the characteristic polynomial $\bar{f}(X)$ of $\rhob_x(g_0)$ modulo $m_{R_{\rhob_x}}$. By assumption, $\bar{f}(X)$ has roots $\chibar_1(g_0)$ and $\chibar_2(g_0)$ which are distinct. Therefore, by Hensel's lemma, $f(X) = (X-a)(X-b)$ such that $a, b \in R_{\rhob_x}$ with $a \equiv \chibar_1(g_0) \Mod{m_{R_{\rhob_x}}}$ and $b \equiv \chibar_2(g_0) \Mod{m_{R_{\rhob_x}}}$.

Let $M_0$ be the free $R_{\rhob_x}$-module underlying the representation $\rho^{\univ}$ and $V$ and $W$ be the submodules of $M_0$ given by $\ker(\rho^{\univ}(g_0)-a)$ and $\ker(\rho^{\univ}(g_0)-b)$, respectively. As $a-b$ is a unit in $R_{\rhob_x}$, it follows that $$(b-a)^{-1}((\rho^{\univ}(g_0)-a)-(\rho^{\univ}(g_0)-b)) = \begin{pmatrix} 1 & 0\\ 0 & 1\end{pmatrix}.$$ 
Thus, if $m \in M_0$, then $$m = (b-a)^{-1}((\rho^{\univ}(g_0)-a)(m)) - (b-a)^{-1}((\rho^{\univ}(g_0)-b)(m)).$$
Let $w = (b-a)^{-1}((\rho^{\univ}(g_0)-a)(m))$ and $v=(b-a)^{-1}((\rho^{\univ}(g_0)-b)(m))$.
Since $f(X)$ is the characteristic polynomial of $\rho^{\univ}(g_0)$, $$f(\rho^{\univ}(g_0)) = (\rho^{\univ}(g_0)-a)(\rho^{\univ}(g_0)-b)=0.$$
Hence, we get that $v \in V$ and $w \in W$.
So $M_0 = V + W$.
Moreover, if $v' \in V \cap W$, then $(b-a)v'=0$.
As $b-a$ is a unit in $R_{\rhob_x}$, this means $v'=0$.
Therefore, $M_0 = V \oplus W$. 

Now $V/m_{R_{\rhob_x}}V$ and $W/m_{R_{\rhob_x}}W$ are the eigenspaces of $\rhob_x(g_0)$ with eigenvalues $\chibar_1(g_0)$ and $\chibar_2(g_0)$, respectively. 
Hence, by Nakayama's lemma, both $V$ and $W$ are generated by one element each. As $M_0$ is a free $R_{\rhob_x}$ module of rank $2$, it follows that both $V$ and $W$ are free $R_{\rhob_x}$ modules of rank $1$. 

This implies that there exists a $Q \in \GL_2(R_{\rhob_x})$ such that $Q\rho^{\univ}(g_0)Q^{-1} = \begin{pmatrix} a & 0\\ 0 & b \end{pmatrix}$. Since, $Q\rho^{\univ}(g_0)Q^{-1} \Mod{m_{R_{\rhob_x}}} = \rhob_x(g_0) = \rho^{\univ}(g_0) \Mod{m_{R_{\rhob_x}}}$ and $\rhob_x(g_0)$ is a non scalar diagonal, it follows that $\bar Q := Q \Mod{m_{R_{\rhob_x}}}$ is a diagonal matrix. Let $Q'$ be a diagonal matrix in $\GL_2(R_{\rhob_x})$ such that $Q' \Mod{m_{R_{\rhob_x}}} = \bar Q$. Hence, ${ Q'}^{-1}Q$ is the matrix in $Id + M_2(m_x)$ satisfying the claim in the lemma.
\end{proof}

By Lemma~\ref{diagonal}, it follows that we can assume that $\rho^{\univ}(g_0)$ is diagonal by replacing $\rho^{\univ}$ by a suitable deformation in its equivalence class if necessary. We will assume and use this in the rest of the article.

Since $\rho^{\univ}(g_0)$ is diagonal with distinct diagonal entries modulo $m_{R_{\rhob_x}}$, by \cite[Lemma $2.4.5$]{Bel}, it follows that $R_{\rhob_x}[\rho^{\univ}(G_{\QQ_p})]$ is a sub-$R_{\rhob_x}$-GMA of $M_2(R_{\rhob_x})$ i.e. there exist ideals $B_p$ and $C_p$ of $R_{\rhob_x}$ such that $R_{\rhob_x}[\rho^{\univ}(G_{\QQ_p})] = \begin{pmatrix} R_{\rhob_x} & B_p \\ C_p & R_{\rhob_x} \end{pmatrix}$. Since $\chibar|_{G_{\QQ_p}} \neq 1, \omega_p^{-1}$, by the local Euler characteristic formula, it follows that $\dim(H^1(G_{\QQ_p},\chibar^{-1}|_{G_{\QQ_p}})) =1$. So, by \cite[Theorem $1.5.5$]{BC}, $C_p$ is generated by at most $1$ element (see \cite[Lemma $2.5$]{D2} and its proof for more details).

Recall that $G_{\QQ_p}^{\text{ab}} \simeq \ZZ_p^{\times} \times \hat\ZZ$ and under this isomorphism, the image of $I_p$ in $G_{\QQ_p}^{\text{ab}}$ gets mapped to $\ZZ_p^{\times}$.
As a result, we get a surjective homomorphism $\phi_p : I_p \twoheadrightarrow \ZZ_p^{\times}$.
Choose a topological generator $\nu$ of $1+p\ZZ_p \subset \ZZ_p^{\times}$ and fix an element $i_p \in I_p$ such that $\phi_p(i_p) = \nu$.
% Choose a lift $i_p$ of a topological generator of $1+p\ZZ_p$ in $I_p$.
\begin{defi}
\label{eltdef}
Under the notation developed above:
%Let $C_p$ be the ideal of $R_{\rhob_x}$ and $i_p$ be the element of $I_p$ defined above. Then:
\begin{itemize}
\item Let $\alpha \in R_{\rhob_x}$ be a generator of the ideal $C_p$ if $C_p \neq (0)$ and $0$ otherwise. 
%Let $I_p$ be the inertia group at $p$. 
\item If $\rho^{\univ}(i_p) = \begin{pmatrix} a & b\\ c & d \end{pmatrix}$ then $d \equiv \hat\chibar_2(i_p) \pmod{m_{R_{\rhob_x}}}$.
Let $\beta \in m_{R_{\rhob_x}}$ be the element such that $d=\hat\chibar_2(i_p)(1+\beta)$.
%$d=\hat\chibar_2(i_p)(1+\beta)$ for some $\beta \in m_{R_{\rhob_x}}$. 
\item Now $\det(\rhob_x(i_p)) =1$. Let $\gamma \in m_{R_{\rhob_x}}$ be the element such that $\det(\rho^{\univ}(i_p)) = 1+\gamma$. 
\item For an integer $k \geq 2$, let $\delta_k = 1 + \gamma - \chi_p^{k-1}(i_p)$. Note that, $\delta_k \in m_{R_{\rhob_x}}$.
\item For $k \geq 2$, let $R_k=R_{\rhob_x}/(\alpha,\beta,\delta_k)$.
\end{itemize}
\end{defi}

 Specializing to $k=2$, let $\rho :G_{\QQ,Np} \to \GL_2(R_2)$ be the representation obtained by composing $\rho^{\univ}$ with the natural surjective map $\GL_2(R_{\rhob_x}) \to \GL_2(R_2)$. 
\begin{lem} 
\label{selmlem}
The representation $\rho \otimes (\hat\chibar_2)^{-1} : G_{\QQ,Np} \to \GL_2(R_2)$ is a \emph{Selmer} deformation of $\begin{pmatrix} \chibar & * \\ 0 & 1\end{pmatrix}$ (as defined in Definition~\ref{defodef3}).
\end{lem}
\begin{proof}
Since $C_p$ is generated by $\alpha$, it follows that $\rho|_{G_{\QQ_p}}$ is of the form $\begin{pmatrix} \eta_1 & *\\ 0 & \eta_2\end{pmatrix}$. For characters $\eta_1, \eta_2 : G_{\QQ_p} \to R_2^{\times}$ lifting $\chibar_1$ and $\chibar_2$, respectively. Recall that, the image of $i_p$ in $G_{\QQ_p}^{\text{ab}}$ is a topological generator of the pro-$p$ part of the image of $I_p$ in $G_{\QQ_p}^{\text{ab}}$. Hence, from definition of $\beta$, it follows that $\eta_2|_{I_p}=\hat\chibar_2|_{I_p}$ and from definition of $\delta_2$, it follows that $\det(\rho)|_{I_p} = \hat\chibar_1 \hat\chibar_2 \hat\omega_p^{-1}\chi_p|_{I_p}$. The lemma follows immediately from the definition of a Selmer deformation of $\rhob'_x$ given in Definition~\ref{defodef3}.
\end{proof}

\subsection{Finiteness of $R_k$ and its consequences}
\label{finsec}
Let $\alpha, \beta, \gamma \in R_{\rhob_x}$ be the elements defined in \S\ref{gmasec}.
We now prove the key finiteness result.

\begin{prop}
\label{finprop}
$R_{\rhob_x}/(p,\alpha,\beta,\gamma)$ is a finite $\FF$-algebra.
\end{prop}
\begin{proof}
Denote $R_{\rhob_x}/(p,\alpha,\beta,\gamma)$ by $S$. Note that $S$ is a complete Noetherian $\FF$-algebra. Denote the representation obtained by composing $\rho^{\univ}$ with the natural surjective map $R_{\rhob_x} \to S$ by $\rho$. 

We have already noted, just before Proposition~\ref{nilprop}, that the map $R^{\ps}_{\rhob_0} \to R_{\rhob_x}$ induced by $(\tr(\rho^{\univ}), \det(\rho^{\univ}))$ is surjective. 
Therefore, the map $R^{\ps}_{\rhob_0} \to S$ obtained by composing the map $R^{\ps}_{\rhob_0} \to R_{\rhob_x}$ with the natural surjective map $R_{\rhob_x} \to S$ is also surjective. Observe that this map is same as the one induced by the pseudo-representation $(\tr(\rho), \det(\rho))$. 

As $\det(\rho) = \chibar_1\chibar_2$, $S$ is topologically generated by the set $\{\tr(\rho(g)) \mid g \in G_{\QQ,Np} \}$ over $\FF$ as a complete local Noetherian $\FF$-algebra. 
 As $\chibar_2(g) \in \FF$ for all $g \in G_{\QQ,Np}$, $S$ is also topologically generated by the set $\{\tr(\rho'(g)) \mid g \in G_{\QQ,Np} \}$ over $\FF$ as a complete local Noetherian $\FF$-algebra, where $\rho' = \rho \otimes \chibar_2^{-1}$.

%Let $R_{\rhob'_x}$ be the universal deformation ring of $\rhob'_x$ in $\mathcal{C}$ and $\rho'^{\univ} : G_{\QQ,Np} \to R_{\rhob'_x}$ be the universal deformation of $\rhob'_x$. 
%As $\rhob_x \otimes \chibar_2^{-1} = \rhob'_x$, the deformation $\rho^{\univ} \otimes \hat\chibar_2^{-1}$ of $\rhob'_x$ induces a map $f : R_{\rhob'_x} \to R_{\rhob_x}$. On the other hand, the deformation $\rho'^{\univ} \otimes \hat\chibar_2$ of $\rhob_x$ induces a map $g : R_{\rhob_x} \to R_{\rhob'_x}$.
%As $f \circ g \circ \rho^{\univ} = \rho^{\univ}$, it follows, from the universal property of $R_{\rhob_x}$, that $f \circ g$ is the identity map. Hence, $f$ is an isomorphism.
%Let $f' : R_{\rhob'_x} \to S$ be the map obtained by composing $f$ with the natural surjective map $R_{\rhob_x} \to S$.
 %It is easy to verify (using the definitions of $f$ and $S$) that $f \circ \rho'^{\univ} = \rho \otimes \chibar_2^{-1}$. We denote $\rho \otimes \chibar_2^{-1}$ by $\rho'$.
{\bf Case $1$:} Suppose $\chibar$ is ramified at $p$ and the $\chibar$-eigenspace of the $p$-part of the class group of $F_{\chibar}$ is trivial. As $\delta_2 \in (p,\alpha,\beta,\gamma)$, Lemma~\ref{selmlem} implies that $\rho'$ is a Selmer deformation of $\rhob'_x$. 
Let $\phi : R^{\text{Sel}}_{\rhob'_x} \to S$ be the map induced by $\rho'$.
%Hence, the map $f' : R_{\rhob'_x} \to S$ factors through $R^{\text{Sel}}_{\rhob'_x}$ inducing a map $\phi : R^{\text{Sel}}_{\rhob'_x}/(p) \to S$.
%Since $f'$ is surjective, $\phi$ is also surjective.
Now if $\rho^{\text{Sel}}: G_{\QQ,Np} \to \GL_2(R^{\text{Sel}}_{\rhob'_x})$ is the universal Selmer deformation of $\rhob'_x$, then $\phi(\tr(\rho^{\text{Sel}}(g))) = \tr(\rho'(g))$ for all $g \in G_{\QQ,Np}$. As the map $\phi$ is a continuous local morphism of complete local Noetherian $W(\FF)$-algebras, it follows that $\phi$ is surjective.

From \cite[Theorem $6.1$]{SW1}, we know that $R^{\text{Sel}}_{\rhob'_x}$ is a finite $W(\FF)$-algebra. 
%isomorphic to the localization of a Hecke algebra acting faithfully on the space of modular forms of weight $2$ and a suitable fixed level at the maximal ideal generated by $p$ and the prime ideal corresponding to a weight $2$ Eisenstein series (see \cite[Section $4$]{SW1} for the precise definition of this Hecke algebra). In particular, $R^{\text{Sel}}_{\rhob'_x}$ is a finite $W(\FF)$-algebra. 
So it follows that $R^{\text{Sel}}_{\rhob'_x}/(p)$ is a finite $\FF$-algebra and hence, $S$ is a finite $\FF$-algebra. 

{\bf Case $2$:} Suppose either $\chibar$ is unramified at $p$ or the $\chibar$-eigenspace of the $p$-part of the class group of $F_{\chibar}$ is non-trivial.
Let $R^{\ps,\text{ord}}_{\chibar}$ be the universal deformation ring defined in Definition~\ref{defodef4}.
Let $(T_0,D_0) : G_{\QQ,Np} \to R^{\ps,\text{ord}}_{\chibar}$ be the universal pseudo-representation associated to $R^{\ps,\text{ord}}_{\chibar}$ and let $T_0|_{G_{\QQ_p}} = \Psi_1 + \Psi_2$, where $\Psi_1$ and $\Psi_2$ are characters of $G_{\QQ_p}$ deforming $\chibar|_{G_{\QQ_p}}$ and $1$, respectively. Hence, $\Psi_2$ induces a morphism $\ZZ_p[[T]] \to R^{\ps,\text{ord}}_{\chibar}$ which sends $T$ to $\Psi_2(i_p)-1$. By abuse of notation, denote $\Psi_2(i_p) -1$ also by $T$.
Since $\hat\chibar$ satisfies the conditions of \cite[Section $5.1.1$]{P}, \cite[Theorem $5.1.2$]{P} implies that $R^{\ps,\text{ord}}_{\chibar}$ is a finite $\ZZ_p[[T]]$-algebra under the map obtained above.

%We have already noted, just before Proposition~\ref{nilprop}, that the map $R^{\ps}_{\rhob_0} \to R_{\rhob_x}$ induced by $(\tr(\rho^{\univ}), \det(\rho^{\univ}))$ is surjective.
%Therefore, the map $R^{\ps}_{\rhob_0} \to S$ obtained by composing the map $R^{\ps}_{\rhob_0} \to R_{\rhob_x}$ with the natural surjective map $R_{\rhob_x} \to S$ is also surjective. Observe that this map is same as the one induced by the pseudo-representation $(\tr(\rho), \det(\rho))$.
 In this case, $\tr(\rho'|_{G_{\QQ_p}}) = 1 + \chibar|_{G_{\QQ_p}}$. Hence, the pseudo-representation $(\tr(\rho'),\det(\rho'))$ induces a surjective map
$\phi_S : R^{\ps,\text{ord}}_{\chibar} \to S$.
From the previous paragraph, we conclude that 
\begin{enumerate}
\item $ \tr(\rho'|_{G_{\QQ_p}}) = \phi_S \circ \Psi_1 + \phi_S \circ \Psi_2 = 1 + \chibar|_{G_{\QQ_p}}$,
\item $\phi_S \circ \Psi_1$ and $\phi_S \circ \Psi_2$ are $S$-valued characters of $G_{\QQ_p}$ deforming $\chibar|_{G_{\QQ_p}}$ and $1$, respectively.
\end{enumerate}
As $\chibar|_{G_{\QQ_p}} \neq 1$, it follows, from \cite[Proposition 1.5.1]{BC}, that $\phi_S \circ \Psi_2=1$ and $\phi_S \circ \Psi_1=\chibar|_{G_{\QQ_p}}$.
Thus, $\phi_S(T) = \phi_S(\Psi_2(i_p)-1)=0$.
Therefore, the surjective map $\phi_S : R^{\ps,\text{ord}}_{\chibar} \to S$ factors through $R^{\ps,\text{ord}}_{\chibar}/(p,T)$.
Since $R^{\ps,\text{ord}}_{\chibar}$ is a finite $\ZZ_p[[T]]$-algebra, it follows that $R^{\ps,\text{ord}}_{\chibar}/(p,T)$ is a finite $\FF$-algebra.
 Hence, $S$ is also a finite $\FF$-algebra. This finishes the proof of the proposition.
\end{proof}

We will now record several corollaries of Proposition~\ref{finprop} which will be used later.

\begin{cor}
\label{dimcor}
The ring $R_{\rhob_x}$ is a local complete intersection ring of Krull dimension $4$.
\end{cor}
\begin{proof}
By Proposition~\ref{bocprop}, we know that $R_{\rhob_x} \simeq W(\FF)[[X_1,...,X_m]]/I$ where $m = \dim(H^1(G_{\QQ,Np},\ad(\rhob_x)))$ and $I$ is an ideal with minimal number of generators at most $\dim(H^2(G_{\QQ,Np},\ad(\rhob_x)))$. By the global Euler characteristic formula, we know that $$\dim(H^1(G_{\QQ,Np},\ad(\rhob_x))) - \dim(H^2(G_{\QQ,Np},\ad(\rhob_x))) =3.$$ Hence, by Theorem~\ref{krullthm}, it follows that the Krull dimension of $R_{\rhob_x}$ is at least $4$. 

Suppose the Krull dimension of $R_{\rhob_x}$ is $n$.
%The kernel of the natural surjective map $R_{\rhob_x} \to R_{\rhob_x}/(p,\alpha,\beta,\gamma)$ is generated by at most $4$ elements. 
By applying Theorem~\ref{krullthm} again, we get that the Krull dimension of $R_{\rhob_x}/(p,\alpha,\beta,\gamma)$ is at least $n-4$. Combining this with Proposition~\ref{finprop}, we see that $n=4$.

Suppose the minimal number of generators of $I$ is $k$. As seen above, $m-k \geq 3$ and Krull dimension of $R_{\rhob_x}$ is at least $1+m-k$. Since the Krull dimension of $R_{\rhob_x}$ is $4$, it follows that $m-k=3$. This proves the remaining part of the corollary.
\end{proof}

%For an integer $k \geq 2$, let $R_k= R_{\rhob_x}/(\alpha,\beta,\delta_k)$ (see \S\ref{miscres} for definitions of $\alpha$, $\beta$ and $\delta_k$).
%As a corollary, we get:
\begin{cor}
\label{fincor}
For every integer $k \geq 2$, the ring $R_k$ is a finite $W(\FF)$-algebra of Krull dimension $1$.
\end{cor}
\begin{proof}
Let $k \geq 2$ be an integer. Note that we have $(p,\alpha,\beta,\gamma)=(p,\alpha,\beta,\delta_k)$.
So, by Proposition~\ref{finprop}, $R_k/(p)$ is a finite $\FF$-algebra. Hence, it follows that $R_k$ is a finite $W(\FF)$-algebra.
As $R_k/(p)$ has Krull dimension $0$, Theorem~\ref{krullthm} implies that the Krull dimension of $R_k$ is at most $1$.
As $R_{\rhob_x}/(\alpha,\beta,\delta_k) = R_k$, Theorem~\ref{krullthm}, along with Corollary~\ref{dimcor}, implies that the Krull dimension of $R_k$ is at least $1$. This proves the corollary.
\end{proof}

%For an integer $k \geq 2$, let $\psi_k : W(\FF)[[T_1,T_2,T_3]] \to R_{\rhob_x}$ be the morphism of $W(\FF)$ algebras sending $T_1$ to $\alpha$, $T_2$ to $\beta$ and $T_3$ to $\delta_k$.

\begin{cor}
\label{mincor}
Let $k \geq 2$ be an integer. If $P$ is a minimal prime ideal of $R_{\rhob_x}$, then the Krull dimension of $R_{\rhob_x}/(P,\alpha,\beta,\delta_k)$ is $1$.
\end{cor}
\begin{proof}
Since $R_{\rhob_x}/(P,\alpha,\beta,\delta_k)$ is a quotient of $R_k$, it follows from Corollary~\ref{fincor} that its Krull dimension at most $1$. 
From Corollary~\ref{dimcor}, we know that $R_{\rhob_x}$ is a local complete intersection ring of Krull dimension $4$. Hence, from \cite[Corollary $18.14$]{E}, we conclude that the Krull dimension of $R_{\rhob_x}/P$ is $4$.
So, Theorem~\ref{krullthm} implies that the Krull dimension of $R_{\rhob_x}/(P,\alpha,\beta,\delta_k)$ is at least $1$ which proves the corollary.
%Note that, $R_{\rhob_x}/(p,\psi(T_1),\psi(T_2),\psi(T_3)) = R_{\rhob_x}/(p,\alpha,\beta,\delta_k)=R_k/(p)$. From Lemma REF, we know that $R/(p)$ is a finite $\FF$-algebra. Hence, by Nakayama's lemma, it follows that the map $\psi_k$ makes $R_{\rhob_x}$ a finite $W(\FF)[[T_1,T_2,T_3]]$-algebra. The lemma now follows from REF.
\end{proof}

\subsection{Modular primes}
\label{modsec}
In this section, we will prove that every irreducible component of $\spec(R_{\rhob_x})$ has a smooth modular prime.
We begin with introducing the notion of modular primes.
\begin{defi}
We call a prime ideal $Q$ of $R_{\rhob_x}$ a \emph{modular prime} if the following conditions are satisfied:
\begin{enumerate}
\item\label{defitem1} $R_{\rhob_x}/Q$ is isomorphic to a subring of $\overline{\QQ_p}$, 
\item If the map $R_{\rhob_x}/Q \to \overline{\QQ_p}$ from \eqref{defitem1} is denoted by $\phi$, then the semi-simplification of the Galois representation $\rho : G_{\QQ,Np} \to \GL_2(\overline{\QQ_p})$ obtained by composing $\rho^{\univ} \Mod{Q}$ with $\phi$ is isomorphic to the $p$-adic Galois representation attached to a classical modular eigenform.
\end{enumerate}
\end{defi}

We are now ready to prove the main result of this section.
%We will now prove the existence of modular primes in $R_{\rhob_x}$ satisfying some nice properties.
This is the second key result from this section which will be crucially used in the proof of the main theorem.
\begin{prop}
\label{modprop}
Every minimal prime ideal $P$ of $R_{\rhob_x}$ is contained in a prime ideal $Q(P)$ of $R_{\rhob_x}$ such that
\begin{enumerate} 
\item $Q(P)$ is a modular prime,
 \item $P$ is the unique minimal prime of $R_{\rhob_x}$ contained in $Q(P)$,
\item $(R_{\rhob_x})_{Q(P)}$ is a regular local ring.
\end{enumerate}
\end{prop}
\begin{proof}
Let $k \geq 2$ be an integer, $P$ be a minimal prime of $R_{\rhob_x}$ and $Q(P)$ be a prime of $R_{\rhob_x}$ which is minimal over the ideal $(P,\alpha,\beta,\delta_k)$. 
In other words, if $\pi_P : R_{\rhob_x} \to R_{\rhob_x}/(P,\alpha,\beta,\delta_k)$ is the natural surjective map, then there exists a minimal prime ideal $\bar{Q}$ of $R_{\rhob_x}/(P,\alpha,\beta,\delta_k)$ such that $\pi_P^{-1}(\bar{Q}) = Q(P)$.
We will first prove that $Q(P)$ is a modular prime.
%By Lemma REF $R_{\rhob_x}/Q(P)$ is a domain of Krull dimension $1$. 
By Corollary~\ref{mincor} and Corollary~\ref{fincor}, $R_{\rhob_x}/Q(P)$ is a domain of Krull dimension $1$ which is finite over $W(\FF)$. Hence, it is isomorphic to a subring of $\overline{\QQ_p}$. Fix such an isomorphism.
%Let $E$ be the fraction field of $R_{\rhob_x}/Q(P)$. Note that it is a finite extension of $\QQ_p$.
By abuse of notation, denote the representation $G_{\QQ,Np} \to \GL_2(\overline{\QQ_p})$ obtained from $\rho^{\univ}\Mod{Q(P)} : G_{\QQ,Np} \to \GL_2(R_{\rhob_x}/Q(P))$ by $\rho^{\univ}\Mod{Q(P)}$ as well.

Consider the representation $\rho^{\univ} \Mod{Q(P)} : G_{\QQ,Np} \to \GL_2(\overline{\QQ_p})$. If it is reducible, then its semi-simplification is $\chi_1 \oplus \chi_2$, where $\chi_1$ and $\chi_2$ are deformations of $\chibar_1$ and $\chibar_2$, respectively. 
By definition of $\beta$, we see that $\chi_2|_{I_p}$ has finite order and hence, $\chi_2$ has finite order.
By definition of $\delta_k$, $\chi_1=\chi'_1\chi_p^{k-1}$ for a character $\chi'_1$ of finite order.
Hence, by Lemma~\ref{eisenlem}, this semi-simplification is isomorphic to the $p$-adic Galois representation attached to a classical Eisenstein series of weight $k$ and level equal to the product of conductors of $\chi'_1$ and $\chi_2$. 

If it is irreducible, then, from the results of Skinner--Wiles (\cite{SW1} and \cite{SW2}), we conclude that it is isomorphic to the $p$-adic Galois representation attached to a classical cuspidal modular eigenform.
Indeed, from the definition of $\alpha$, $\beta$ and $\delta_k$, it follows that $\rho^{\univ} \Mod{Q(P)}|_{G_{\QQ_p}} \simeq \begin{pmatrix} \eta_1 & * \\ 0 & \eta_2\end{pmatrix}$, $\eta_2\hat\chibar_2^{-1}$ is unramified at $p$ (and hence, of finite order) and $\det(\rho^{\univ} \Mod{Q(P)}) = \eta\chi_p^{k-1}$ for some character $\eta$ of finite order.
Since $\dim(H^1(G_{\QQ,Np},\chibar)) =1$, it follows, from Lemma~\ref{dimonelem}, that if $\ell \mid N$, then either $\chibar$ is ramified at $\ell$ or $\chibar(\text{Frob}_{\ell}) \neq \ell$.
So if $\chibar_2$ is unramified at $p$, $\chibar$ is ramified at $p$ and the $\chibar$-eigenspace of the $p$-part of the class group of $F_{\chibar}$ is trivial, then \cite[Theorem $1.2$]{SW1} implies that $\rho^{\univ} \Mod{Q(P)}$ is modular.
Otherwise, we can use \cite[Theorem A]{SW2} to conclude the modularity of $\rho^{\univ} \Mod{Q(P)}$.
This completes the proof of our claim.

So for every minimal prime $P$ of $R_{\rhob_x}$, we have found infinitely many modular primes containing $P$.
%From these primes, we will now find primes which contain a unique minimal prime of $R_{\rhob_x}$.
Note that, for every minimal prime $P$, it suffices to find a modular prime $Q(P)$ such that $P \subset Q(P)$ and $(R_{\rhob_x})_{Q(P)}$ is a regular local ring.
%So in each case, we have found modular primes $Q(P)$ such that $(R_{\rhob_x})_{Q(P)}$ is a regular local ring. 
Indeed, if such a prime exists, then, by \cite[Theorem $19.19$]{E}, $(R_{\rhob_x})_{Q(P)}$ is a domain. Thus, $(R_{\rhob_x})_{Q(P)}$ has a unique minimal prime. So, $Q(P)$ contains a unique minimal prime of $R_{\rhob_x}$. As we already know that $P \subset Q(P)$, this completes the proof of the proposition.

We will now do a case by case analysis to prove the existence of such modular primes for every minimal prime of $R_{\rhob_x}$.

%Since $\hat\chibar_2$ is a Dirichlet character of finite order, by twisting $\rho \Mod{Q(P)} \otimes \hat\chibar_2^{-1}$ by $\hat\chibar_2$, we see that semi-simplification of $\rho \Mod{Q(P)}$ is also isomorphic to a Galois representation attached to classical modular eigenform of tame level $N$. 
{\bf Case $1$:} Suppose $\chibar$ is ramified at $p$ and the $\chibar$-eigenspace of the $p$-part of the class group of $F_{\chibar}$ is trivial.
Let $P$ be a minimal prime of $R_{\rhob_x}$.
Let $Q(P)$ be a prime of $R_{\rhob_x}$ which is minimal over the ideal $(P,\alpha, \beta, \delta_2) \subset R_{\rhob_x}$.
We know, from the analysis above, that $Q(P)$ is modular.
From Corollary~\ref{mincor}, Corollary~\ref{fincor} and definition of $R_2$, we see that the image of $Q(P)$ under the natural surjective map $R_{\rhob_x} \to R_2$ is a minimal prime of $R_2$. Let us denote this minimal prime of $R_2$ by $Q'(P)$.

From Lemma~\ref{selmlem} and the proof of Proposition~\ref{finprop}, we conclude that there exists a surjective map $\phi' : R^{\text{Sel}}_{\rhob'_x} \to R_2$.
Let $Q''(P) = \phi^{-1}(Q'(P))$. 
By \cite[Theorem $6.1$]{SW1}, $R^{\text{Sel}}_{\rhob'_x}$ is reduced and has Krull dimension $1$. So, it follows that $Q''(P)$ is a minimal prime ideal of $R^{\text{Sel}}_{\rhob'_x}$ and $(R^{\text{Sel}}_{\rhob'_x})_{Q''(P)}$ is a field. The map $\phi'$ induces a surjective morphism $(R^{\text{Sel}}_{\rhob'_x})_{Q''(P)} \to (R_2)_{Q'(P)}$. Hence, $(R_2)_{Q'(P)}$ is also a field. 

Now the natural surjective map $R_{\rhob_x} \to R_{\rhob_x}/(\alpha,\beta,\delta_2)=R_2$ induces a surjective map $(R_{\rhob_x})_{Q(P)} \to (R_2)_{Q'(P)}$. Moreover, the kernel of this surjective map is the ideal of $(R_{\rhob_x})_{Q(P)}$ generated by the images of $\alpha$, $\beta$ and $\delta_2$ in $(R_{\rhob_x})_{Q(P)}$. Since $(R_2)_{Q'(P)}$ is a field, it follows that the kernel of the surjective map $(R_{\rhob_x})_{Q(P)} \to (R_2)_{Q'(P)}$ is the maximal ideal of $(R_{\rhob_x})_{Q(P)}$. Thus the maximal ideal of $(R_{\rhob_x})_{Q(P)}$ is generated by at most $3$ elements.

Since $R_{\rhob_x}$ is a local complete intersection ring, it is a Cohen-Macaulay ring. By Corollary~\ref{dimcor}, $R_{\rhob_x}$ has Krull dimension $4$ and $R_{\rhob_x}/Q(P)$ has Krull dimension $1$. Hence, by \cite[Corollary $18.10$]{E}, $(R_{\rhob_x})_{Q(P)}$ has Krull dimension $3$. Therefore, $(R_{\rhob_x})_{Q(P)}$ is a regular local ring. 

{\bf Case $2$:} Suppose either $\chibar$ is unramified at $p$ or the $\chibar$-eigenspace of the $p$-part of the class group of $F_{\chibar}$ is non-trivial.
Choose an integer $k >2$ satisfying the properties of Proposition~\ref{unobsprop}.
Let $Q(P)$ be a prime of $R_{\rhob_x}$ which is minimal over the ideal $(P,\alpha, \beta, \delta_k) \subset R_{\rhob_x}$.
We know, from the analysis above, that $Q(P)$ is modular.
 
Let $E$ be the fraction field of $R_{\rhob_x}/Q(P)$. Note that it is a finite extension of $\QQ_p$.
By abuse of notation, denote the representation $G_{\QQ,Np} \to \GL_2(E)$ obtained from $\rho^{\univ}\Mod{Q(P)} : G_{\QQ,Np} \to \GL_2(R_{\rhob_x}/Q(P))$ by $\rho^{\univ}\Mod{Q(P)}$ as well.
We will now prove that $(R_{\rhob_x})_{Q(P)}$ is a regular local ring by using Proposition~\ref{kisinprop} and Theorem~\ref{unobsthm}.
%the relation between $R_{\rhob_x}$ and the universal deformation ring of $\rho^{\univ} \Mod{Q(P)}$ proved by Kisin (\cite{Ki2}).

Let $\widehat{(R_{\rhob_x})_{Q(P)}}$ be the completion of $(R_{\rhob_x})_{Q(P)}$ with respect to its maximal ideal.
Now we are choosing $\rho^{\univ}$ such that $\rho^{\univ}(g_0)$ is diagonal with distinct entries on the diagonal which are also distinct modulo $m_{R_{\rhob_x}}$.
Therefore, \cite[Lemma $2.4.5$]{Bel} implies that $R_{\rhob_x}/(Q(P))[\rho^{\univ}\Mod{Q(P)}(G_{\QQ,Np})]$ is a sub-$R_{\rhob_x}/(Q(P))$-GMA of $M_2(R_{\rhob_x}/(Q(P)))$ i.e. there exist ideals $B$ and $C$ of $R_{\rhob_x}/Q(P)$ such that $$R_{\rhob_x}/(Q(P))[\rho^{\univ}\Mod{Q(P)}(G_{\QQ,Np})] = \begin{pmatrix} R_{\rhob_x}/(Q(P)) & B\\ C & R_{\rhob_x}/(Q(P)) \end{pmatrix}.$$ As $\rho^{\univ}\Mod{Q(P)}$ is a lift of $\rhob_x$, it follows that $B$ is not zero.
%at lease one of $B$ and $C$ is not zero.
One can easily verify from this description that $\rho^{\univ}\Mod{Q(P)}$ is either reducible or absolutely irreducible.
Note that if $\rho^{\univ}\Mod{Q(P)}$ is reducible, then it is a non-trivial extension of two distinct characters.
Hence, in both cases, Proposition~\ref{kisinprop} implies that $\widehat{(R_{\rhob_x})_{Q(P)}}$ is the universal deformation ring of $\rho^{\univ}\Mod{Q(P)}$.
% by \cite[Proposition $9.5$]{Ki2}.

If $\rho^{\univ}\Mod{Q(P)}$ is reducible, then Proposition~\ref{unobsprop} implies that $$\dim(H^1(G_{\QQ,Np},\text{Ad}(\rho^{\univ}\Mod{Q(P)}))) =3.$$
As $\rho^{\univ}\Mod{Q(P)}$ is odd, the global Euler characteristic formula (\cite[Lemma 9.7]{Ki2}) implies that $H^2(G_{\QQ,Np},\text{Ad}(\rho^{\univ}\Mod{Q(P)}))=0$.
Therefore, using Lemma~\ref{preslem}, we conclude that $\widehat{(R_{\rhob_x})_{Q(P)}}$ is a complete Noetherian regular local ring of Krull dimension $3$.
% (see proof of \cite[Corollary $9.8$]{Ki2}).
Recall that $Q(P)$ is a modular prime.
So, if $\rho^{\univ}\Mod{Q(P)}$ is irreducible and $\chibar \neq \omega_p^k$ for any integer $k$, then, by Theorem~\ref{unobsthm},  $\widehat{(R_{\rhob_x})_{Q(P)}} \simeq E\llbracket T_1,T_2,T_3\rrbracket$ i.e. it is a complete Noetherian regular local ring of Krull dimension $3$.

Suppose $\chibar=\omega_p^{k-1}$ for some even $k \in \ZZ$ with $k > 3$.
From the main theorem of \cite{SW2}, we know that $\rho^{\univ}\Mod{Q(P)} \otimes \hat\chibar_2^{-1}$ is modular. 
From Lemma~\ref{cyclem}, we see that if $\ell \mid N$, then $p \nmid \ell^2-1$. Hence, definition of $\delta_k$ implies that $\det(\rho^{\univ}\Mod{Q(P)} \otimes \hat\chibar_2^{-1}) = \chi_p^{k-1}$.
Let $f$ be the newform giving rise to $\rho^{\univ}\Mod{Q(P)} \otimes \hat\chibar_2^{-1}$.
Then the level of $f$ is $\Gamma_0(N')$ for some integer $N'$.

%Comparing the Artin conductors of $\rho^{\univ}\Mod{Q(P)} \otimes \hat\chibar_2^{-1}$ and $\rhob'_x$, we conclude, using \cite[Proposition $2$]{Ca} and Lemma~\ref{cyclem}, that $N'$ is either $1$ or a squarefree integer bigger than $1$.
If $p \mid N'$, then \cite[Theorem $5$]{AL} implies that the $U_p$-eigenvalue of $f$ is either $\pm p^{(k-2)/2}$ or $0$.
Note that $\rho^{\univ}\Mod{Q(P)} \otimes \hat\chibar_2^{-1} = \rho_f$ is $p$-ordinary i.e. $\rho_f|_{G_{\QQ_p}} \simeq \begin{pmatrix} \eta_1 & *\\ 0 & \eta_2 \end{pmatrix}$, where $\eta_2$ is an unramified character of $G_{\QQ_p}$. So, it follows, from the work of Saito (\cite{S}), that $f$ is $p$-ordinary i.e. $U_p$-eigenvalue of $f$ is a $p$-adic unit (see \cite[Proposition $3.6$]{Mo}).
Since $f$ is $p$-ordinary and $k > 2$, we get that $p \nmid N'$.
Now $\rho_f$ lifts $1 \oplus \omega_p^{k-1}$ and by Proposition~\ref{condprop}, the tame Artin conductor of $\rho_f$ is $N'$.
Hence, by Proposition~\ref{carayolprop}, we get that $N'$ is either $1$ or a squarefree integer bigger than $1$.

If $N'=1$, then $f$ is not a CM modular form. Suppose $N'>1$. Since $N'$ is square-free and the level of $f$ is $\Gamma_0(N')$, it follows that $\rho_f|_{G_{\QQ_\ell}}$ is special for every $\ell \mid N$ (see \cite[Section $3.3$]{We2}). Hence, $f$ is not a CM modular form in this case as well. 
Hence, we conclude, using Theorem~\ref{unobsthm}, that
$$\dim(H^1(G_{\QQ,Np}, \text{Ad}(\rho^{\univ}\Mod{Q(P)} \otimes \hat\chibar_2^{-1}))) =3,$$
$$ H^2(G_{\QQ,Np}, \text{Ad}(\rho^{\univ}\Mod{Q(P)} \otimes \hat\chibar_2^{-1})) =0.$$

Note that $\text{Ad}(\rho^{\univ}\Mod{Q(P)} \otimes \hat\chibar_2^{-1}) = \text{Ad}(\rho^{\univ}\Mod{Q(P)})$.
Therefore, by Lemma~\ref{preslem}, we get that $\widehat{(R_{\rhob_x})_{Q(P)}} \simeq E \llbracket T_1,T_2,T_3 \rrbracket$ i.e. it is a complete Noetherian regular local ring of Krull dimension $3$. This implies that $(R_{\rhob_x})_{Q(P)}$ is a regular local ring. 
%Therefore, combining \cite[Theorem $1$]{We1}, the proof of part $(3)$ of \cite[Theorem $11.10$]{Ki2} and the proof of \cite[Theorem $8.2$]{Ki}, we get that $$\dim(H^1(G_{\QQ,Np}, \text{Ad}(\rho^{\univ}\Mod{Q(P)} \otimes \hat\chibar_2^{-1}))) = \dim(H^1(G_{\QQ,Np},\text{Ad}(\rho^{\univ}\Mod{Q(P)}))) = 3.$$
%So, by the global Euler characteristic formula, we get that $H^2(G_{\QQ,Np},\text{Ad}(\rho^{\univ}\Mod{Q(P)})) =0$.
%
%Hence, by arguments of the proof of \cite[Proposition $2$]{M}, we see that $\widehat{(R_{\rhob_x})_{Q(P)}}$ is a complete Noetherian regular local ring of Krull dimension $3$. This implies that $(R_{\rhob_x})_{Q(P)}$ is a regular local ring. 
\end{proof}

As a corollary, we get:
\begin{cor}
\label{redcor}
$R_{\rhob_x}$ is reduced.
\end{cor}
\begin{proof}
Suppose $R_{\rhob_x}$ is not a reduced ring. Let $r$ be a non-zero nilpotent element of $R_{\rhob_x}$. Let $P$ be a minimal prime of $R_{\rhob_x}$ and $Q(P)$ be a modular prime of $R_{\rhob_x}$ containing $P$ which we found in Proposition~\ref{modprop}. So $P$ is the only minimal prime of $R_{\rhob_x}$ contained in $Q(P)$ and moreover, $(R_{\rhob_x})_{Q(P)}$ is a domain. 

Hence, the image of $r$ in $(R_{\rhob_x})_{Q(P)}$ is $0$. This means that there exists an element $a \in R_{\rhob_x} \setminus Q(P)$ such that $ar=0$. Thus, the annihilator ideal of $r$, which we will denote by $J$, is not contained in $Q(P)$. Therefore, $J \not\subset P$ for any minimal prime ideal $P$ of $R_{\rhob_x}$. Hence, by prime avoidance lemma (\cite[Lemma $3.3$]{E}), $J \not\subset \cup_{P \in S} P$ where $S$ is the set of minimal prime ideals of $R_{\rhob_x}$.

Since $R_{\rhob_x}$ is a Cohen-Macaulay ring, by \cite[Corollary $18.14$]{E}, all its associated primes are minimal primes. Hence, the set of zero divisors of $R_{\rhob_x}$ is just $\cup_{P \in S} P$ (see \cite[Theorem $3.1$]{E}). Since $J$ is the annihilator ideal of $r$, it consists of zero divisors of $R_{\rhob_x}$. This implies that  $J \subset \cup_{P \in S} P$ giving us a contradiction.
\end{proof}

Combining Corollary~\ref{redcor} with the results of \cite{D2}, we get the following corollary:
\begin{cor}
\label{isomcor}
If $p \nmid \phi(N)$ and $1 \leq \dim(H^1(G_{\QQ,Np},\chibar^{-1})) \leq 3$, then the map $R^{\ps}_{\rhob_0} \to R_{\rhob_x}$ induced by $(\tr(\rho^{\univ}), \det(\rho^{\univ}))$ is an isomorphism.
\end{cor}
\begin{proof}
Since $p \nmid \phi(N)$ and $1 \leq \dim_{\FF}H^1(G_{\QQ,Np},\chibar^{-1}) \leq 3$, by \cite[Proposition $3.14$]{D2}, it suffices to check that $p$ is not a zero divisor in $R_{\rhob_x}$. By Corollary~\ref{dimcor}, $R_{\rhob_x}$ is a local complete intersection ring and hence, a Cohen-Macaulay ring. So by \cite[Corollary $18.14$]{E} all its associated primes are minimal primes. Let $P$ be a minimal prime of $R_{\rhob_x}$. By Corollary~\ref{fincor} and Corollary~\ref{mincor}, we see that $p \not\in P$. Therefore, $p$ is not contained in any associated prime of $R_{\rhob_x}$ which means $p$ is not a zero divisor in $R_{\rhob_x}$ (by \cite[Theorem $3.1$]{E}). 
\end{proof}

\section{Comparison between deformation rings and Hecke algebras}
\label{mainsec}
In this section, we will prove our main theorem by comparing the deformation ring $R_{\rhob_x}$ with a `big' $p$-adic Hecke algebra.
In \S\ref{heckesec}, we will first recall the definition of these Hecke algebras from \cite{BK} and \cite{D} and collect their relevant properties.
In \S\ref{proofsec}, we will prove the main theorem using results from \S\ref{defosec} and the infinite fern argument of Gouv\^{e}a--Mazur (\cite{GM}).
In \S\ref{modpsec}, we will give an application of our main theorem to the structure of Hecke algebras modulo $p$ (as defined in \cite{BK} and \cite{D}).

\subsection{Hecke algebras}
\label{heckesec}
In this section, we will define the $p$-adic Hecke algebras acting on the spaces of modular forms of fixed level but varying weights and collect its properties. 
We will mostly follow the exposition of \cite{BK} and \cite{D} in this subsection. We will be rather brief and we refer the reader to \cite[Section $1.2$]{BK} and \cite[Section $1$]{D} for more details. Assume throughout this section that $N_0 \mid N$ unless mentioned otherwise, where $N_0$ is the tame Artin conductor of $\rhob_0$.

For an integer $i \geq 0$, denote by $M_i(N, W(\FF))$ the space of classical modular forms of weight $i$ and level $N$ (i.e. level $\Gamma_1(N)$) with  Fourier coefficients in $W(\FF)$. Using the $q$-expansion principle, we will view it as a subspace of $W(\FF)[[q]]$.
For an integer $k \geq 0$, let $M_{\leq k}(N, W(\FF)) = \sum_{i=0}^{k} M_i(N, W(\FF)) \subset W(\FF)[[q]]$.
Note that this is a direct sum.

For a prime $\ell \nmid Np$, denote by $\langle \ell \rangle$ the diamond operator corresponding to $\ell$.
Let $\TT_{\leq k}(N)$ be the $W(\FF)$-subalgebra of $\text{End}_{W(\FF)}(M_{\leq k}(N, W(\FF)))$ generated by the Hecke operators $T_{\ell}$ and $S_{\ell}$ for primes $\ell \nmid Np$. Here $S_{\ell}$ is the operator whose action on $M_i(N, W(\FF))$ coincides with the action of the operator $\langle \ell \rangle \ell^{i-2}$.
We now define $\TT(N) := \varprojlim_k \TT_{\leq k}(N)$.

From Lemma~\ref{modlem}, we know that there exists an eigenform $f$ of level $N$ such that $\rho_f$ lifts $\rhob_0$.
Hence, $f$ gives a morphism $\phi_f : \TT(N) \to \mathcal{O}_f$ of local complete Noetherian $W(\FF)$-algebras.
Let $m_{\rhob_0}$ be the kernel of the map $\TT(N) \to \mathcal{O}_f/\varpi_f$ obtained by composing $\phi_f$ with the natural surjective map $\mathcal{O}_f \to \mathcal{O}_f/\varpi_f$.
 So $m_{\rhob_0}$ is a maximal ideal of $\TT(N)$. Let $\TT(N)_{\rhob_0}$ be the localization of $\TT(N)$ at $m_{\rhob_0}$.
Note that $\TT(N)_{\rhob_0}$ is a complete Noetherian local ring with residue field $\FF$. Denote its maximal ideal by $\mathfrak{m}$.

By gluing pseudo-representations corresponding to $p$-adic Galois representations of classical modular eigenforms of level $N$ lifting $\rhob_0$, we get the following lemma:

%lifting the system of Hecke eigenvalues corresponding to $\rhob_0$, we get the following lemma:
\begin{lem}
\label{heckelem}
There exists a pseudo-representation $(\tau_N,\delta_N) : G_{\QQ,Np} \to \TT(N)_{\rhob_0}$ such that $\tau_N \Mod{\mathfrak{m}} = \tr(\rhob_0)$, $\delta_N \Mod{\mathfrak{m}} = \det(\rhob_0)$ and for all primes $\ell \nmid Np$, $\tau_N(\frob_{\ell})=T_{\ell}$ and $\delta_N(\frob_{\ell})=\ell S_{\ell}$.
\end{lem}
\begin{proof}
See \cite[Lemma $4$]{D} and \cite[Proposition $2$]{BK}.
\end{proof}

Let $(T^{\univ}, D^{\univ}) : G_{\QQ,Np} \to R^{\ps}_{\rhob_0}$ be the universal pseudo-deformation.
As an immediate consequence of Lemma~\ref{heckelem}, we get:
\begin{lem}
\label{surjlem}
There exists a surjective map $\Phi : R^{\ps}_{\rhob_0} \to \TT(N)_{\rhob_0}$ such that $\Phi(T^{\univ}(\frob_{\ell})) = T_{\ell}$ and $\Phi(D^{\univ}(\frob_{\ell})) = \ell S_{\ell}$  for all primes $\ell \nmid Np$.
\end{lem}

\begin{lem}
\label{factorlem}
The surjective map $\Phi : R^{\ps}_{\rhob_0} \to \TT(N)_{\rhob_0}$ factors through $R_{\rhob_x}$ yielding a surjective map $\Psi : R_{\rhob_x} \to \TT(N)_{\rhob_0}$.
\end{lem}
\begin{proof}
Since $\TT(N)_{\rhob_0}$ is reduced, the lemma follows from Proposition~\ref{nilprop}.
But for the convenience of the reader, we will also give a slightly different proof of the lemma by constructing a deformation of $\rhob_x$ taking values in $\TT(N)_{\rhob_0}$.

Let $K(N)_{\rhob_0}$ be the total fraction field of $\TT(N)_{\rhob_0}$.
By combining \cite[Proposition $2.4.2$]{Bel} and \cite[Proposition $1.3.12$]{BC}, it follows that there exist fractional ideals $B$ and $C$ of $K(N)_{\rhob_0}$ and a representation $\rho : G_{\QQ,Np} \to \GL_2(K(N)_{\rhob_0})$ such that 
\begin{enumerate}
\item $\tr(\rho) = \tau_N$ and $\det(\rho) = \delta_N$,
\item $\TT(N)_{\rhob_0}[\rho(G_{\QQ,Np})] = \begin{pmatrix} \TT(N)_{\rhob_0} & B \\ C & \TT(N)_{\rhob_0} \end{pmatrix} = \Bigg\{\begin{pmatrix} a & b \\ c & d \end{pmatrix} \in M_2(K(N)_{\rhob_0}) \mid a,d \in \TT(N)_{\rhob_0}, b \in B, c \in C\Bigg\}$,
\item $BC \subset \mathfrak{m}$,
\item If $\rho(g) = \begin{pmatrix} a_g & b_g \\ c_g & d_g \end{pmatrix}$, then $a_g \pmod{\mathfrak{m}} = \chibar_1(g)$ and $d_g \pmod{\mathfrak{m}} = \chibar_2(g)$.
\end{enumerate}
Since $\dim(H^1(G_{\QQ,Np},\chibar))=1$, it follows, from the proof of \cite[Theorem $1.5.5$]{BC}, that $B$ is generated by at most $1$ element as a $\TT(N)_{\rhob_0}$-module (see also \cite[Lemma $2.5$]{D2}). Let $\alpha$ be a generator of $B$ as a $\TT(N)_{\rhob_0}$-module (we take $\alpha=0$ if $B=0$).

As $BC \subset \TT(N)_{\rhob_0}$, it follows that if $\alpha \not\in K(N)_{\rhob_0}^{\times}$, then there exists a minimal prime $P$ of $\TT(N)_{\rhob_0}$ such that $\alpha.C = BC \subset P$.
Then, from the description of $\TT(N)_{\rhob_0}[\rho(G_{\QQ,Np})]$, it follows that $a \pmod{P} : G_{\QQ,Np} \to (\TT(N)_{\rhob_0}/P)^{\times}$ sending $g \mapsto a_g \pmod{P}$ is a character deforming $\chibar_1$ and $d \pmod{P} : G_{\QQ,Np} \to (\TT(N)_{\rhob_0}/P)^{\times}$ sending $g \mapsto d_g \pmod{P}$ is a character deforming $\chibar_2$.

Therefore, if $R_{\chibar_i}$ is the universal deformation ring in $\mathcal{C}$ of the character $\chibar_i : G_{\QQ,Np} \to \FF^{\times}$ for $i=1,2$, then the characters $a \pmod{P}$ and $d \pmod{P}$ induce maps $f_1 : R_{\chibar_1} \to \TT(N)_{\rhob_0}/P$ and $f_2 : R_{\chibar_2} \to \TT(N)_{\rhob_0}/P$, respectively.
Hence, combining these maps gives us a map $f : R_{\chibar_1} \hat\otimes_{W(\FF)} R_{\chibar_2} \to \TT(N)_{\rhob_0}/P$. 

Note that $a_g \pmod{P} \in \text{Im}(f_1)$ and $d_g \pmod{P} \in \text{Im}(f_2)$ for all $g \in G_{\QQ,Np}$.
Hence, $a_g \pmod{P} +d_g \pmod{P} =\tau_N(g) \pmod{P} \in \text{Im}(f)$ and $a_g \pmod{P}.d_g\pmod{P} =\delta_N(g) \pmod{P} \in \text{Im}(f)$ for all $g \in G_{\QQ,Np}$.
Therefore, by Lemma~\ref{heckelem}, the map $f$ is surjective.

Suppose the pro-$p$ part of the abelianization of $G_{\QQ,Np}$ is $\mathbb{Z}_p \times \prod_{i=1}^{r}\mathbb{Z}/p^{e_i}\mathbb{Z}$.
Then for $i=1,2$, $R_{\chibar_i} \simeq W(\FF)[[T,X_1,\cdots,X_{r}]]/I$, where $I$ is the ideal generated by the set $\{ (1+X_i)^{p^{e_i}}-1 \mid 1 \leq i \leq r \}$.
So, we see that the Krull dimension of $R_{\chibar_1} \hat\otimes_{W(\FF)} R_{\chibar_2}$ is $3$ which implies that the Krull dimension of $\TT(N)_{\rhob_0}/P$ is at most $3$.
However, by Gouv\^{e}a--Mazur infinite fern argument, we know that every component of $\text{Spec}(\TT(N)_{\rhob_0})$ has Krull dimension at least $4$ (see \cite[Corollary $2.28$]{Em} and \cite[Theorem $1$]{GM}).
Therefore, we get a contradiction which implies that $\alpha \in K(N)_{\rhob_0}^{\times}$.

As $B=\alpha\TT(N)_{\rhob_0}$ and $\alpha.C \subset \TT(N)_{\rhob_0}$, it follows, from the properties of $\rho$ listed above, that if $M= \begin{pmatrix} \alpha^{-1} & 0 \\ 0 & \alpha \end{pmatrix}$, then conjugating $\rho$ by $M$ gives us a representation $\rho'=M\rho M^{-1} : G_{\QQ,Np} \to \GL_2(\TT(N)_{\rhob_0})$ such that 
\begin{enumerate}
\item $\tr(\rho') = \tau_N$ and $\det(\rho') = \delta_N$,
\item $\rho' \pmod{\mathfrak{m}} = \begin{pmatrix} \chibar_1 & * \\ 0 & \chibar_2\end{pmatrix}$, where $*$ is non-zero.
\end{enumerate}
Since $\dim(H^1(G_{\QQ,Np},\chibar))=1$, we can conjugate $\rho'$ further by a diagonal matrix in $\GL_2(\TT(N)_{\rhob_0})$ if necessary to get a deformation $\rho'' : G_{\QQ,Np} \to \GL_2(\TT(N)_{\rhob_0})$ of $\rhob_x$ such that $\tr(\rho'') = \tau_N$ and $\det(\rho'') = \delta_N$.

Hence, $\rho''$ induces a map $\Psi : R_{\rhob_x} \to \TT(N)_{\rhob_0}$. 
Now $$\Psi \circ F_x \circ (T^{\univ}, D^{\univ}) = \Psi \circ (\tr(\rho^{\univ}),\det(\rho^{\univ})) = (\tr(\rho''),\det(\rho'')) = (\tau_N, \delta_N).$$
Therefore, the universal property of $R^{\ps}_{\rhob_0}$ implies that $\Psi \circ F_x = \Phi$. This proves the lemma.
\end{proof}

\subsection{Main theorem and its proof}
\label{proofsec}
We are now ready to state and prove our main theorem.
\begin{thm}
\label{reducecor}
Suppose $N$ satisfies the conditions of Corollary~\ref{condlem}. Then the map $\Phi : R^{\ps}_{\rhob_0} \to \TT(N)_{\rhob_0}$ induces an isomorphism $(R^{\ps}_{\rhob_0})^{\red} \simeq \TT(N)_{\rhob_0}$.
\end{thm}

To prove Theorem~\ref{reducecor}, we first prove the following theorem.
\begin{thm}
\label{mainthm}
Suppose $N$ satisfies the conditions of Corollary~\ref{condlem}. Then the surjective map $\Psi : R_{\rhob_x} \to \TT(N)_{\rhob_0}$ is an isomorphism.
\end{thm}
\begin{proof}
Let $P$ be a minimal prime of $R_{\rhob_x}$ and let $Q(P)$ be a modular prime corresponding to $P$ as found in Proposition~\ref{modprop}. So $P$ is the unique minimal prime of $R_{\rhob_x}$ contained in $Q(P)$. Let $f_P$ be the classical modular eigenform corresponding to $Q(P)$. From Corollary~\ref{condlem}, it follows that the tame level of $f_P$ is $N$. Note that, from our construction of $Q(P)$, we immediately see that $p^2$ does not divide the conductor of the nebentypus $\epsilon_{f_P}$ of $f_P$. Hence, \cite[Theorem $1$]{H} implies that $f_P$ is a $p$-adic modular form of tame level $N$. Let $k_p$ be the weight of $f_P$.

By \cite[Proposition I.$3.9$]{G}, $\TT(N)$ is the Hecke algebra acting faithfully on the space of $p$-adic modular forms of level $N$. Combining this with the fact that $\rho_{f_P}$ lifts $\rhob_0$, we get a map $\phi_P : \TT(N)_{\rhob_0} \to \overline{\QQ_p}$ sending $T_{\ell}$ to $a_{\ell}(f_P)$, the $T_{\ell}$ eigenvalue of $f_P$, and $S_{\ell}$ to $\epsilon_{f_P}(\ell)\ell^{k_P-2}$ for every prime $\ell \nmid Np$. Let $Q'(P)$ be the kernel of this map. 

Note that there exists an injective map $\iota : R_{\rhob_x}/Q(P) \to \overline{\QQ_p}$ such that composing $\iota$ with $\rho^{\univ} \Mod{Q(P)}$ gives us $\rho_{f_P} : G_{\QQ,Np} \to \GL_2(\overline{\QQ_p})$.
Hence, if $F_P : R_{\rhob_x} \to \overline{\QQ_p}$ is the map obtained by composing the natural surjective map $R_{\rhob_x} \to R_{\rhob_x}/Q(P)$ with $\iota$, then Lemma~\ref{surjlem} and Lemma~\ref{factorlem}, along with the previous paragraph, imply that $$(\tr(F_P \circ \rho^{\univ}(\frob_{\ell})), \det(F_P \circ \rho^{\univ}(\frob_{\ell}))) = (\tr(\phi_P \circ \Psi \circ \rho^{\univ}(\frob_{\ell})), \det(\phi_P \circ \Psi \circ \rho^{\univ}(\frob_{\ell})))$$ for all primes $\ell \nmid Np$. We have already noted that in our setting, the map $R^{\ps}_{\rhob_0} \to R_{\rhob_x}$ induced by $(\tr(\rho^{\univ}),\det(\rho^{\univ}))$ is surjective.
Combining this with the Chebotarev density theorem, we conclude that $F_P = \phi_P \circ \Psi$.
As $\ker(F_P) = Q(P)$, $\Psi$ is surjective and $\ker(\phi_P) = Q'(P)$, we get that $Q(P) = \Psi^{-1}(Q'(P))$.
%Note that, $Q(P) = \Psi^{-1}(Q'(P))$. Indeed, since the map $R^{\ps} \to R_{\rhob_x}$ induced by $\tr(\rho^{\univ})$ is surjective, by definition of $Q(P)$ and $f_P$, it is the ideal of $R_{\rhob_x}$ generated by the set $\{\tr(\rho^{\univ}(\frob_q)) - a_q(f_P) | q \text{ a prime }, q \nmid Np\}$. Moreover, $\Psi(\tr(\rho^{\univ}(\frob_q))) = T_q$ for all primes $q \nmid Np$ and hence, $Q(P) \subset \ker(\phi_P \circ \Psi)$. Now $R_{\rhob_x}/Q(P)$ is a domain of Krull dimension $1$ and $Im(\phi_P \circ \Psi)$ is a subring of $\overline{\QQ_p}$. This implies that $Q(P) = \ker(\phi_P \circ \Psi)$ and hence, $Q(P) = \Psi^{-1}(Q'(P))$. 

Let $P'$ be a minimal prime of $\TT(N)_{\rhob_0}$ contained in $Q'(P)$. Let $J$ be the kernel of the map $R_{\rhob_x} \to \TT(N)_{\rhob_0}/P'$ obtained by composing $\Psi$ with the natural surjective map $\TT(N)_{\rhob_0} \to \TT(N)_{\rhob_0}/P'$. So $J$ is a prime ideal of $R_{\rhob_x}$ which is contained in $Q(P)$. 

Now, by Gouv\^{e}a--Mazur infinite fern argument (\cite[Corollary $2.28$]{Em} and \cite[Theorem $1$]{GM}), we know that the Krull dimension of $\TT(N)_{\rhob_0}/P'$ is at least $4$. From Corollary~\ref{dimcor}, we know that $R_{\rhob_x}$ has Krull dimension $4$. Therefore, $J$ is a minimal prime of $R_{\rhob_x}$ contained in $Q(P)$ and hence, by Proposition~\ref{modprop}, $J=P$. As $R_{\rhob_x}/P \simeq \TT(N)_{\rhob_0}/P'$, we see that $\ker(\Psi) \subset P$. Thus, we see that $\ker(\Psi)$ is contained in every minimal prime of $R_{\rhob_x}$. By Corollary~\ref{redcor}, $R_{\rhob_x}$ is reduced and hence, $\Psi$ is injective. Since we already know that $\Psi$ is surjective from Lemma~\ref{factorlem}, this gives the theorem.
\end{proof}

We are now ready to prove the main theorem.
%\section{Comparison between $R^{\ps}$ and $R_{\rhob_x}$}
\begin{proof}[Proof of Theorem~\ref{reducecor}]
Recall that composition of the morphism $R^{\ps}_{\rhob_0} \to R_{\rhob_x}$ with $\Psi$ is $\Phi$.
The theorem follows immediately from Theorem~\ref{mainthm} and Proposition~\ref{nilprop}.
\end{proof}

We now get the following results as immediate corollaries:
\begin{cor}
If $p \nmid \phi(N)$, then the map $\Phi : R^{\ps}_{\rhob_0} \to \TT(N)_{\rhob_0}$ induces an isomorphism $(R^{\ps}_{\rhob_0})^{\red} \simeq \TT(N)_{\rhob_0}$.
\end{cor}
\begin{proof}
If $p \nmid \phi(N)$, then $N$ satisfies the conditions of Corollary~\ref{condlem}. Hence, the corollary follows from Theorem~\ref{reducecor}.
\end{proof}

\begin{cor}
\label{isomcorr}
If $p \nmid \phi(N)$ and $1 \leq \dim_{\FF}H^1(G_{\QQ,Np},\chibar^{-1}) \leq 3$, then the map $\Phi : R^{\ps}_{\rhob_0} \to \TT(N)_{\rhob_0}$ is an isomorphism.
\end{cor}
\begin{proof}
 As $N$ satisfies the conditions of Corollary~\ref{condlem}, the corollary follows from Theorem~\ref{mainthm} and Corollary~\ref{isomcor}.
\end{proof}

\subsection{A consequence for Hecke algebras modulo $p$}
\label{modpsec}
We now give an application of Theorem~\ref{mainthm} to the structure of mod $p$ Hecke algebras. 
For an integer $i$, let $S_i(N,W(\FF))$ be the space of cusp forms of level $N$ and weight $i$ with Fourier coefficients in $W(\FF)$. 
If we replace $M_i(N,W(\FF))$ with $S_i(N,W(\FF))$ everywhere in the definition of $\TT(N)$, we get the cuspidal $p$-adic Hecke algebra which we denote by $\TT'(N)$ (see \cite[Section $1.2$]{BK} and \cite[Section $1$]{D} for more details).
The restriction of the action of Hecke operators from the space of modular forms to the space of cusp forms gives a surjective map $j : \TT(N) \to \TT'(N)$.

As we are assuming $p$ to be odd, the pseudo-representation $(\tau_N, \delta_N)$ is determined by $\tau_N$ and $\tau_N : G_{\QQ,Np} \to \TT(N)_{\rhob_0}$ is a pseudo-character of dimension $2$ in the sense of Rouquier (\cite{R}) (see \cite[Section $1.4$]{BK} for more details).

Let $I_0$ be the reducibility ideal of $\tau_N$ in the sense of \cite{BC}. So it is the smallest ideal of $\TT(N)_{\rhob_0}$ modulo which $\tau_N$ is sum of two characters lifting $\chibar_1$ and $\chibar_2$. As $\chibar_1 \neq \chibar_2$, such an $I_0$ does exist (see \cite[Section $1.5$]{BC} for the existence and properties of the reducibility ideal). 

\begin{lem}
\label{reddimlem}
The Krull dimension of $\TT(N)_{\rhob_0}/I_0$ is at most $3$.
\end{lem}
\begin{proof}
This is a standard result. We will give a brief sketch of its proof for the convenience of the reader.
 Since $\tau_N \Mod{I_0}$ is a sum of two characters lifting $\chibar_1$ and $\chibar_2$, the map $R^{\ps}_{\rhob_0} \to \TT(N)_{\rhob_0}/I_0$ factors through $R_{\chibar_1} \hat\otimes_{W(\FF)} R_{\chibar_2}$ where $R_{\chibar_i}$ is the universal deformation ring in $\mathcal{C}$ of the character $\chibar_i : G_{\QQ,Np} \to \FF^{\times}$ for $i=1,2$ (see proof of Lemma~\ref{factorlem}). 
By Lemma~\ref{surjlem} this map is surjective.
We already saw in the proof of Lemma~\ref{factorlem} that the Krull dimension of $R_{\chibar_1} \hat\otimes_{W(\FF)} R_{\chibar_2}$ is $3$ which implies the lemma.
\end{proof}

%\begin{prop}
%There exists a representation $\rho_{\TT} : G_{\QQ,Np} \to \GL_2(\TT(N)_{\rhob_0})$ such that $\tr(\rho_{\TT})=\tau_N$ and $\rho_{\TT} \Mod{\mathfrak{m}} = \rhob_x$.
%\end{prop}
%\begin{proof}
%Since $\TT(N)_{\rhob_0}$ is reduced, the proposition follows from REF.
%\end{proof}

\begin{lem}
\label{cuspheckelem}
\begin{enumerate}
\item\label{cusp1} The image of the maximal ideal $m_{\rhob_0}$ under the surjective map $j : \TT(N) \to \TT'(N)$ is a maximal ideal $m'_{\rhob_0}$ of $\TT'(N)$.
\item\label{cusp2} The natural surjective map $j : \TT(N) \to \TT'(N)$ induces an isomorphism between $\TT(N)_{\rhob_0}$ and $\TT'(N)_{\rhob_0}$ where $\TT'(N)_{\rhob_0}$ is the localization of $\TT'(N)$ at $m'_{\rhob_0}$.
\end{enumerate}
\end{lem}
\begin{proof}
We begin with the first part of the lemma. Suppose the image of $m_{\rhob_0}$ under $j$ is not a maximal ideal of $\TT'(N)$. This implies that $\ker(j) \not\subset m_{\rhob_0}$. This means that there does not exist a cuspidal eigenform $f$ of level $N$ such that $\rho_f$ lifts $\rhob_0$.

Let $E$ be an Eisenstein series of level $N$ such that $\rho_E$ lifts $\rhob_0$. Denote by $P_E$ the corresponding prime ideal of $\TT(N)_{\rhob_0}$. So $P_E$ is the kernel of the map $\TT(N)_{\rhob_0} \to \overline{\QQ_p}$ which sends $T_{\ell}$ to $T_{\ell}$-eigenvalue of $E$ for $\ell \nmid Np$. Let $I = \cap P_E$ where the intersection is taken over all Eisenstein series of level $N$ such that $\rho_E$ lifts $\rhob_0$. From the definition of $\TT(N)_{\rhob_0}$ and the non-existence of a cuspidal eigenform $f$ of level $N$ such that $\rho_f$ lifts $\rhob_0$, it follows that $I=0$.

By \cite[Proposition $1.5.1$]{BC}, we know that $I_0 \subset P_E$ for all Eisenstein series $E$ of level $N$ such that $\rho_E$ lifts $\rhob_0$. Hence, it follows that $I_0 =0$. Therefore, Lemma~\ref{reddimlem} implies that Krull dimension of $\TT(N)_{\rhob_0}$ is at most $3$. But by Gouv\^{e}a--Mazur infinite fern argument, we know that the Krull dimension of $\TT(N)_{\rhob_0}$ is at least $4$ (see \cite[Corollary $2.28$]{Em} and \cite[Theorem $1$]{GM}). Hence, we get a contradiction. This proves the first part of the lemma.

 Let $J$ be the kernel of the map $ \TT(N)_{\rhob_0} \to \TT'(N)_{\rhob_0}$ induced by $j$. From definition of $\TT(N)_{\rhob_0}$ and $\TT'(N)_{\rhob_0}$ it follows that $IJ=0$.
By Gouv\^{e}a--Mazur infinite fern argument, we know that every component of $\text{Spec}(\TT(N)_{\rhob_0})$ has Krull dimension at least $4$ (see \cite[Corollary $2.28$]{Em} and \cite[Theorem $1$]{GM}).
So, from the proof of the first part of the lemma, we conclude that $I$ is not contained in any of the minimal primes of $\TT(N)_{\rhob_0}$. So $J$ is contained in every minimal prime ideal of $\TT(N)_{\rhob_0}$. Since $\TT(N)_{\rhob_0}$ is reduced, we get that $J =0$. This proves the second part of the lemma.
\end{proof} 

\begin{rem}
Note that Part~\eqref{cusp1} of Lemma~\ref{cuspheckelem} implies that there exists a cuspidal eigenform $f$ of tame level $N$ lifting $\rhob_0$.
On the other hand, Part~\eqref{cusp2} of Lemma~\ref{cuspheckelem} implies that the subspace generated by such cusp forms is `big' inside the space generated by all modular eigenforms of tame level $N$ lifting $\rhob_0$.
\end{rem}

We will now recall the definition of mod $p$ Hecke algebras from \cite{BK} and \cite{D}.
Let $S_{\leq k}(N, W(\FF)) := \sum_{i=0}^{k}S_i(N,W(\FF))$.
Using $q$-expansion principle, we view $S_{\leq k}(N,W(\FF))$ as a submodule of $W(\FF)[[q]]$. Let $S_{\leq k}(N,\FF)$ be the image of $S_{\leq k}(N, W(\FF))$ under the natural surjective map $W(\FF)[[q]] \to \FF[[q]]$.
The action of Hecke operators $T_{\ell}$ and $S_{\ell}$ on $S_{\leq k}(N,W(\FF))$ induces an action of the Hecke operators on $S_{\leq k}(N,\FF)$.
Let $A_{\leq k}(N)$ be the $\FF$-subalgebra of $\text{End}_{\FF}(S_{\leq k}(N,\FF))$ generated by the Hecke operators $T_{\ell}$ and $S_{\ell}$ for all primes $\ell \nmid Np$.

Let $A(N) := \varprojlim_k A_{\leq k}(N)$. We have a natural surjective map $\TT'(N) \to A(N)$ and the image of $m'_{\rhob_0}$ under this surjective map gives a maximal ideal $m''_{\rhob_0}$ of $A(N)$ (see \cite[Section $1$]{D} for more details).
Let $A(N)_{\rhob_0}$ be the localization of $A(N)$ at $m''_{\rhob_0}$.
Let $\tilde R^{\ps,0}_{\rhob_0}$ be the universal pseudo-deformation ring of $(\tr(\rhob_0),\det(\rhob_0))$ with constant determinant in characteristic $p$ (see \cite[Section $1$]{D} for its definition).

\begin{cor}
Suppose $N$ satisfies the conditions of Corollary~\ref{condlem} and $p > 3$.
The ring $A(N)_{\rhob_0}$ has Krull dimension $2$ and $(\tilde R^{\ps,0}_{\rhob_0})^{\red} \simeq (A(N)_{\rhob_0})^{\red}$.
\end{cor}
\begin{proof}
By Theorem~\ref{mainthm} and Lemma~\ref{cuspheckelem}, we conclude that the map $R^{\ps}_{\rhob_0} \to \TT'(N)_{\rhob_0}$ obtained by composing $\Phi$ with the surjective map $\TT(N)_{\rhob_0} \to \TT'(N)_{\rhob_0}$ induces an isomorphism $(R^{\ps}_{\rhob_0})^{\red} \simeq \TT'(N)_{\rhob_0}$ and both have Krull dimension $4$. 
Note that while proving \cite[Theorem $3$]{D}, we prove that if $R^{\ps}_{\rhob_0} \simeq \TT'(N)_{\rhob_0}$ and the Krull dimension of $R^{\ps}_{\rhob_0}$ is $4$, then $(\tilde R^{\ps,0}_{\rhob_0})^{\red} \simeq (A(N)_{\rhob_0})^{\red}$ and the Krull dimension of $A(N)_{\rhob_0}$ is $2$.
However, our proofs of these two claims also work if we replace $R^{\ps}_{\rhob_0}$ by $(R^{\ps}_{\rhob_0})^{\red}$ in the assumptions (see \cite[Section 8]{D} for more details).
Hence, the corollary follows from Theorem~\ref{reducecor}.
% Hence, following the proof of \cite[Theorem $3$]{D} gives us the corollary.
\end{proof}

\section{Application to level raising of modular forms}
\label{levelsec}
In this section, we will give an application of our results to finding \emph{newforms} $f$ of non-optimal levels (i.e. of level bigger than $N_0$) such that $\rho_f$ lifts $\rhob_0$.
These results are in the same spirit as the ones obtained by Diamond--Taylor (\cite{DT}) for irreducible $\rhob_0$.
%Note that such results have already been obtained by Yoo (\cite{Y}) in the case of $\rhob_0 = 1 \oplus \omega_p$ using different methods.
In the case of $\rhob_0 = 1 \oplus \omega_p^k$ for odd $k > 1$, our results answer the conjecture of Billerey--Menares (\cite[Conjecture $3.2$]{BM}) partially in the setup described below.

We will first describe the setup with which we will be working. We also slightly change the notation for the deformation rings that we have been using so far.

\begin{setup}
\label{levelsetup}
Let $\rhob_0 : G_{\QQ,Cp} \to \GL_2(\FF)$ be an odd, reducible and semi-simple representation for some integer $C$ (as introduced in Setup~\ref{basic}).
We keep the other notations of Setup~\ref{basic}.
Let $N_0$ be the tame Artin conductor of $\rhob_0$.
Suppose $\dim(H^1(G_{\QQ,N_0p},\chibar)) =1$ and $\chibar|_{G_{\QQ_p}} \neq 1, \omega_p^{-1}|_{G_{\QQ_p}}$.
\begin{enumerate}
%\item Let $N_0$ be the Artin conductor of $\rhob_0$. 
\item\label{bhaag1} Let $S_0$ be the set of primes $\ell$ such that $\ell \mid N_0$, $p \mid \ell -1$ and either $\chibar_1$ or $\chibar_2$ is unramified at $\ell$. 
\item\label{bhaag2} Let $N'_0 = N_0 \prod_{\ell \in S_0} \ell$.
\item\label{bhaag3} Let $x \in H^1(G_{\QQ,N_0p},\chibar)$ be a non-zero element and $\rhob_x : G_{\QQ,N_0p} \to \GL_2(\FF)$ be the representation corresponding to $x$ as defined in Definition~\ref{defodef2}.
\item\label{bhaag4} Let $r \geq 1$ be an integer and let $\ell_1,\cdots,\ell_r$ be primes such that $\ell_i \nmid N_0p$, $\chibar|_{G_{\QQ_{\ell_i}}} = \omega_p^{-1}|_{G_{\QQ_{\ell_i}}}$ and $p \nmid \ell_i-1$ for all $1 \leq i \leq r$. 
\item\label{bhaag5} Let $s$ be the cardinality of the set $\{j \mid 1 \leq j \leq r \text{ and } p \nmid (\ell_j+1)\}$. If $s \neq r$, then assume, after reordering if necessary, that $p \nmid (\ell_i+1)$ for all $1 \leq i \leq s$ and $p \mid (\ell_i+1)$ for all $s+1 \leq i \leq r$.
\item\label{bhaag6} If $s=r$, then let $N=N'_0\prod_{i=1}^{r}\ell_i$ and if $s \neq r$, then let $N=N'_0\prod_{i=1}^{s}\ell_i\prod_{j=s+1}^{r}\ell_j^2$.
%Observe that both $N$ and $N'_0$ satisfy the conditions of Lemma~\ref{condlem}.
\item\label{bhaag7} Note that $\rhob_x$ is also a representation of $G_{\QQ,Np}$. Let $R_{\rhob_x}(N_0)$ and $R_{\rhob_x}(N)$ be the universal deformation rings of $\rhob_x$ in $\mathcal{C}$ for the groups $G_{\QQ,N_0p}$ and $G_{\QQ,Np}$, respectively.
\item\label{bhaag8} Let $\rho^{\univ} : G_{\QQ,Np} \to \GL_2(R_{\rhob_x}(N))$ be the universal deformation of $\rhob_x$.
\end{enumerate}
\end{setup}

Throughout this section, we will use the notation given in the setup above and also assume that the hypotheses of the setup hold unless mentioned otherwise.

\subsection{Preliminary results}
We will now prove a some preliminary results which will be used later.
We begin with a result which is well known but we include a proof for the sake of completion.
\begin{lem}
\label{unramlem}
Suppose $f$ is a weight $k$ newform of level $M$ and $\ell$ is a prime not dividing $Mp$. Then $\rho_f|_{G_{\QQ_\ell}} \simeq \eta_1 \oplus \eta_2$, where $\eta_1, \eta_2 : G_{\QQ_{\ell}} \to \overline{\QQ_p}^{\times}$ are unramified characters of $G_{\QQ_\ell}$ such that $\eta_1\eta_2^{-1} \neq \chi_p|_{G_{\QQ_{\ell}}}, \chi_p^{-1}|_{G_{\QQ_{\ell}}}$.
\end{lem}
\begin{proof}
Since $\ell \nmid Mp$, we know that $\rho_f|_{G_{\QQ_{\ell}}}$ is unramified at $\ell$. Hence, we get that $\rho_f|_{G_{\QQ_\ell}} \simeq \eta_1 \oplus \eta_2$, for some unramified characters $\eta_1, \eta_2 : G_{\QQ_{\ell}} \to \overline{\QQ_p}^{\times}$ and $T_{\ell}$-eigenvalue of $f$ is $\tr(\rho_f(\text{Frob}_{\ell}))$.
Suppose $\eta_1\eta_2^{-1}$ is either $\chi_p|_{G_{\QQ_{\ell}}}$ or $\chi_p^{-1}|_{G_{\QQ_{\ell}}}$.
So there exists an $a \in \overline{\QQ_p}$ such that $\tr(\rho_f(\text{Frob}_{\ell})) = a(1+\ell)$ and $\det(\rho_f(\text{Frob}_{\ell})) = a^2\ell$.
Now $\det(\rho_f(\text{Frob}_{\ell})) = \epsilon(\text{Frob}_\ell)\ell^{k-1}$, where $\epsilon$ is the nebentypus of $f$.
Hence, it follows that the $T_{\ell}$-eigenvalue of $f$ is $(1+\ell)\ell^{(k-2)/2}\zeta_0$ for some squareroot $\zeta_0$ of $\epsilon(\text{Frob}_\ell)$.
But Weil bound on the $T_{\ell}$ eigenvalue of $f$ (given by the Ramanujan--Petersson conjecture) implies that $(\ell+1)\ell^{(k-2)/2} \leq 2 \ell^{(k-1)/2}$ (see \cite[Lemma $6.2$]{DM} and its proof for more details). 
This means $\ell+1 \leq 2\sqrt{\ell}$ which is not possible as $\ell >1$.
This gives us a contradiction. 
\end{proof}

\begin{lem}
\label{ordlem}
Suppose $\chibar_2$ is unramified at $p$ and $\chibar_1\chibar_2 = \bar\psi\omega_p^{k_0-1}$, with $\bar\psi$ unramified at $p$ and $1 < k_0 < p$. 
Let $k>2$ be an integer such that $k \equiv k_0 \Mod{p-1}$. 
Suppose $g$ is a $p$-ordinary eigenform of weight $k$ such that $\rho_g$ lifts $\rhob_0$ and suppose $M$ is the level of the newform underlying $g$.
Then $p \nmid M$.
\end{lem}
\begin{proof}
Let $f$ be the newform underlying $g$.
Thus, the representation $\rho_g$ factors through $G_{\QQ,Mp}$ and is isomorphic to $\rho_f$ as a $G_{\QQ,Mp}$-representation.
Since $g$ is $p$-ordinary, it follows, from \cite[Theorem $2$]{Wi}, that $\rho_f$ is an ordinary representation such that $\det(\rho_f) = \epsilon\chi_p^{k-1}$, where $\epsilon$ is the nebentypus of $f$. In other words, $\rho_f|_{G_{\QQ_p}} = \begin{pmatrix} \eta_1\chi_p^{k-1} & *\\ 0 & \eta_2 \end{pmatrix}$, where $\eta_1$ and $\eta_2$ are unramified characters of $G_{\QQ_p}$.
As $k \equiv k_0 \Mod{p-1}$, we see that the character $\epsilon$ is unramified at $p$.

%Moreover, from the previous paragraph, we conclude that the nebentypus of $f$ is unramified at $p$. 
Suppose $p$ divides the level of $f$. Then the corresponding $\GL_2(\QQ_p)$-representation $\pi_{f,p}$ (i.e. $p$-component of the automorphic representation attached to $f$) is either principal series, special or supercuspidal. 
Suppose $\pi_{f,p}$ is the principal series representation $\pi(\epsilon_1,\epsilon_2)$ for some characters $\epsilon_1$ and $\epsilon_2$. As $p$ divides the level of $f$, at least one of $\epsilon_1$ and $\epsilon_2$ should be ramified. Since the nebentypus of $f$ is unramified at $p$, it follows that both $\epsilon_1$ and $\epsilon_2$ are ramified. But then \cite[Section $2.2$]{LW} implies that the $U_p$-eigenvalue of $f$ is $0$ which contradicts the fact that $f$ is $p$-ordinary. Hence, $\pi_{f,p}$ is not principal series.

If $\pi_{f,p}$ is supercuspidal, then \cite[Section $2.2$]{LW} implies that the $U_p$-eigenvalue of $f$ is $0$ which again gives a contradiction. Hence, $\pi_{f,p}$ is not supercuspidal.
If $\pi_{f,p}$ is special, then \cite[Section $2.2$]{LW} implies that the $U_p$-eigenvalue of $f$ is either $0$ or $p^{(k-2)/2}.\zeta$ for some root of unity $\zeta$. Note that $k > 2$ which means the $U_p$-eigenvalue of $f$ is a not a $p$-adic unit. This again contradicts the fact that $f$ is $p$-ordinary.
%This means that if $\pi_{f,p}$ is special, then the $U_p$-eigenvalue of $f$ is a not a $p$-adic unit which again contradicts the fact that $f$ is $p$-ordinary.
 Hence, it follows that $p$ does not divide the level of $f$.
\end{proof}

Let $N$ be the integer defined in Point~\eqref{bhaag6} of Setup~\ref{levelsetup}.
\begin{lem}
\label{ord}
Suppose $\chibar_2$ is unramified at $p$, $\chibar|_{G_{\QQ_p}} \neq 1, \omega_p^{-1}$ and $\chibar_1\chibar_2 = \bar\psi\omega_p^{k_0-1}$, with $\bar\psi$ unramified at $p$ and $1 < k_0 < p$. 
Let $k$ be an integer such that $k \equiv k_0 \Mod{p-1}$. 
Then there exist $\alpha, \beta, \delta_k \in R_{\rhob_x}(N)$ such that if $\rho_k := \rho^{\univ} \pmod{(\alpha,\beta,\delta_k)}$ then
%is ordinary of weight $k$ i.e.
\begin{enumerate} 
\item $\rho_k|_{G_{\QQ_p}} \simeq \begin{pmatrix} \eta_1 & * \\ 0 & \eta_2 \end{pmatrix}$, where $\eta_1$ and $\eta_2$ are characters lifting $\chibar_1$ and $\chibar_2$, repsectively,
\item $\eta_2$ is unramified at $p$,
\item $\det(\rho_k) = \epsilon_k\chi_p^{k-1}$ and the character $\epsilon_k$ is unramified at $p$.
\end{enumerate}
\end{lem}
\begin{proof}
From the discussion after Lemma~\ref{diagonal}, it immediately follows that the elements $\alpha$, $\beta$ and $\delta_k$ defined in Definition~\ref{eltdef} satisfy the conclusion of the lemma.
 %In fact the elements $\alpha$, $\beta$ and $\delta_k$ defined there will satisfy the conclusion of the lemma. 
Note that their definition only depended on the hypothesis $\chibar|_{G_{\QQ_p}} \neq 1, \omega_p^{-1}$ and Lemma~\ref{diagonal} (which in turn only depended on the hypothesis $\chibar|_{G_{\QQ_p}} \neq 1$) and not on whether $\dim(H^1(G_{\QQ,Np},\chibar))=1$ or not. 
So we can define them in the current setup in the exact same way.
\end{proof}

\begin{rem}
Note that the definitions of $\alpha$ and $\beta$ are independent of $k$.
\end{rem}
%Define $R_k(N_0)$ and $R_k(N)$ as quotients of $R_{\rhob_x}(N_0)$ and $R_{\rhob_x}(N)$, respectively in the same way as $R_k$ was defined as a quotient of $R_{\rhob_x}$ in \S\ref{miscres}. Since $\dim(H^1(G_{\QQ,N_0p},\chibar)) =\dim(H^1(G_{\QQ,Np},\chibar)) =1$, we can indeed define both $R_k(N_0)$ and $R_k(N)$.
\subsection{Increasing the ramification}
\label{ramsec}
Let $r$ be the integer and $\ell_1,\cdots,\ell_r$ be the primes from Point~\eqref{bhaag4} of Setup~\ref{levelsetup}.
For any $1 \leq i \leq r$, our assumptions imply that the characteristic polynomial of $\rhob_0(\text{Frob}_{\ell_i})$ has distinct roots. 
Let $g_{\ell_i}$ be a lift of $\text{Frob}_{\ell_i}$ in $G_{\QQ,Np}$. By Hensel's lemma, the characteristic polynomial of $\rho^{\univ}(g_{\ell_i})$ has two distinct roots $\alpha_i$ and $\beta_i$ in $R_{\rhob_x}(N)$ such that $\alpha_i \pmod{m_{R_{\rhob_x}(N)}} = \chibar_2(\text{Frob}_{\ell_i})$ and $\beta_i \pmod{m_{R_{\rhob_x}(N)}} = \chibar_1(\text{Frob}_{\ell_i})$ (which means $\alpha_i\beta_i^{-1} \equiv \ell_i \Mod{m_{R_{\rhob_x}(N)}}$).

Recall, from Point~\eqref{bhaag5} of Setup~\ref{levelsetup}, that $s$ is $|\{j \mid 1 \leq j \leq r \text{ and } p \nmid \ell_j+1 \}|$, $p \nmid \ell_i+1$ for all $1 \leq i \leq s$ and $p \mid \ell_i+1$ for all $s +1 \leq i \leq r$.
We will now recall some results of B\"{o}ckle (\cite{Bo1}) which will be crucially used in the proof of the main result of this section.

\begin{lem}
\label{boclem}
 For all $1 \leq i \leq r$, $\rho^{\univ}|_{I_{\ell_i}}$ factors through the $\ZZ_p$-quotient (i.e. the maximal pro-$p$ quotient) of the tame inertia group at $\ell_i$.
Moreover, if $i_{\ell_i}$ is a lift of a topological generator of this $\ZZ_p$-quotient in $I_{\ell_i}$, then, 
\begin{enumerate}
 \item For every $1 \leq i \leq s$, there exists a $P_i \in Id + M_2(m_{R_{\rhob_x}(N)})$ and $w_i \in m_{R_{\rhob_x}(N)}$ such that $$P_i(\rho^{\univ}(i_{\ell_i}))P_i^{-1} = \begin{pmatrix} 1 & 0 \\ w_i & 1 \end{pmatrix} \text{ and } P_i(\rho^{\univ}(g_{\ell_i}))P_i^{-1} = \begin{pmatrix} \beta_i & 0 \\ 0 & \alpha_i \end{pmatrix},$$
\item For every $s+1 \leq j \leq r$, there exists a $Q_j \in Id + M_2(m_{R_{\rhob_x}(N)})$ and $u_j, v_j \in m_{R_{\rhob_x}(N)}$ such that $$Q_j(\rho^{\univ}(i_{\ell_j}))Q_j^{-1} = \begin{pmatrix} \sqrt{1+u_jv_j} & u_j \\ v_j & \sqrt{1+u_jv_j} \end{pmatrix} \text{ and } Q_j(\rho^{\univ}(g_{\ell_j}))Q_j^{-1} = \begin{pmatrix} \beta_j & 0 \\ 0 & \alpha_j \end{pmatrix}.$$
Moreover, there exist $f_{\ell_j}, h_{\ell_j} \in R_{\rhob_x}(N)$ be such that $$(Q_j(\rho^{univ}(i_{\ell_j}))Q_j^{-1})^{\ell_j} = \begin{pmatrix} f_{\ell_j}  & u_jh_{\ell_j} \\ v_jh_{\ell_j} &  f_{\ell_j} \end{pmatrix}.$$
%For every $s+1 \leq j \leq r$, let $\phi_j=\beta_j\alpha_j^{-1}$.
\end{enumerate}
\end{lem}
\begin{proof}
This follows directly from \cite[Lemma $4.9$]{Bo1} and its proof.
\end{proof}

Before recalling the next result, we establish some more notation.
If $s \neq 0$, then let $M := \dfrac{N}{\ell_1}$ and if $s = 0$, then let $M := \dfrac{N}{\ell_1^2}$.
Thus, $\ell_1 \nmid M$.
Moreover, if $s \neq 0$, then $p \nmid \ell_1 + 1$ and if $s =0$, then $p \mid \ell_1+1$.
Let $\phi_1 := \beta_1\alpha_1^{-1}$.
Suppose $R(M)_{\rhob_x} \simeq \dfrac{W(\FF) \llbracket X_1,\cdots,X_m \rrbracket}{(r_1,\cdots,r_{m'})}$. In the notation established above, we have:

\begin{thm}
\label{bocthm}
\begin{enumerate}
 \item If $s \neq 0$, then there exists a surjective map $$\mathcal{G} : W(\FF)\llbracket X_1,\cdots,X_m, T \rrbracket \to R(N)_{\rhob_x}$$
and elements $f_1,\cdots,f_{m'},g \in R_0$ such that $f_i \pmod{T} = r_i$ for all $1 \leq i \leq m'$, $\mathcal{G}(T) = w_1$, $\mathcal{G}(g) = \alpha_1-\ell_1\beta_1$ and $\ker(\mathcal{G})$ is generated by the set $\{f_1,\cdots,f_{m'}, gT\}$.
\item If $s = 0$, then there exists a surjective map $$\mathcal{G} : W(\FF)\llbracket X_1,\cdots,X_m, T,T' \rrbracket \to R(N)_{\rhob_x}$$
and elements $f_1,\cdots,f_{m'},g,g' \in R_0$ such that $f_i \pmod{(T,T')} = r_i$ for all $1 \leq i \leq m'$, $\mathcal{G}(T) = u_1$, $\mathcal{G}(T') = v_1$, $\mathcal{G}(g) = h_{\ell_1}-\phi_1$, $\mathcal{G}(g') = h_{\ell_1}-\phi_1^{-1}$ and $\ker(\mathcal{G})$ is generated by the set $\{f_1,\cdots,f_{m'}, gT,g'T'\}$.
%For every $s+1 \leq j \leq r$, let $\phi_j=\beta_j\alpha_j^{-1}$.
\end{enumerate}
\end{thm}
\begin{proof}
This is \cite[Theorem $4.7$]{Bo1}.
\end{proof}

Let $R^{\ps}_{\rhob_0}(N)$ and $R^{\ps}_{\rhob_0}(N_0)$ be the universal deformation rings of the pseudo-representation $(\tr(\rhob_0),\det(\rhob_0))$ in $\mathcal{C}$ for the groups $G_{\QQ,Np}$ and $G_{\QQ,N_0p}$, respectively.
%under a suitable basis, $\rho^{\univ}(i_{\ell_i})=\begin{pmatrix} 1 & t_i \\ 0 & 1 \end{pmatrix}$ for some $t_i \in R_{\rhob_x}(N)$.

\begin{lem}
\label{surjective}
The morphism $\phi : R^{\ps}_{\rhob_0}(N) \to R_{\rhob_x}(N)/(u_{s+1},\cdots,u_r)$ induced by $(\tr(\rho^{\univ}),\det(\rho^{\univ}))\pmod{(u_{s+1},\cdots,u_r)}$ is surjective.
\end{lem}
\begin{proof}
Let $J = (u_{s+1},\cdots,u_r)$ and  $\rho' := \rho^{\univ} \pmod{J}$.
Let $f : R_{\rhob_x}(N)/J \to \FF[\epsilon]/(\epsilon^2)$ be a map and let $\rho := f \circ \rho'$.
To prove the lemma, it suffices to prove that if $\tr(\rho)=\tr(\rhob_0)$ and $\det(\rho) = \det(\rhob_0)$, then $\rho \simeq \rhob_x$.

For every $1 \leq i \leq s$, let $\bar{w_i} := w_i \pmod{J}$ and $\overline{P_i} := P_i \pmod{J}$.
For every $s+1 \leq j \leq r$, let $\bar{v_j} := v_j \pmod{J}$ and $\overline{Q_j} := Q_j \pmod{J}$.

Since $\tr(\rho)=\tr(\rhob_0)$ and $\rhob_0$ is unramified at $\ell_i$ for all $1 \leq i \leq r$, it follows that $\tr(\rho(gi_{\ell_i})) = \tr(\rho(g))$ for all $g \in G_{\QQ,Np}$ and all $1 \leq i \leq r$.
As $\rho := f \circ \rho'$, it follows that for all $g \in G_{\QQ,Np}$ and $1 \leq i \leq r$, $\tr(\rho'(gi_{\ell_i})) - \tr(\rho'(g)) \in \ker(f)$.

Let $h \in G_{\QQ,Np}$ be such that $\rhob_x(h) = \begin{pmatrix} 1 & \xi \\ 0 & 1 \end{pmatrix}$ for some $\xi \in \FF^{\times}$.
So for every $1 \leq i \leq s$, $\overline{P_i}\rho'(h)\overline{P_i}^{-1} = \begin{pmatrix} 1+a_i & b_i \\  c_i & 1+d_i \end{pmatrix}$ with $b_i \in (R_{\rhob_x}(N)/J)^{\times}$ and for every $s+1 \leq j \leq r$, $\overline{Q_j}\rho'(h)\overline{Q_j}^{-1} = \begin{pmatrix} 1+a_j & b_j \\  c_j & 1+d_j \end{pmatrix}$ with $b_j \in (R_{\rhob_x}(N)/J)^{\times}$ and $a_j, c_j, d_j \not\in (R_{\rhob_x}(N)/J)^{\times}$.

So for every $1 \leq i \leq s$, $$\tr(\rho'(h i_{\ell_i}) ) - \tr(\rho'(h)) = \tr(\overline{P_i}\rho'(h i_{\ell_i})\overline{P_i}^{-1} ) - \tr(\overline{P_i}\rho'(h)\overline{P_i}^{-1}) = b_i \bar{w_i},$$
and for every $s+1 \leq j \leq r$, $$\tr(\rho'(h i_{\ell_j}) ) - \tr(\rho'(h)) = \tr(\overline{Q_j}\rho'(h i_{\ell_j})\overline{Q_j}^{-1} ) - \tr(\overline{Q_j}\rho'(h)\overline{Q_j}^{-1}) = \nu_j \bar{v_j}$$ for some $\nu_j \in (R_{\rhob_x}(N)/J)^{\times}$.

So for all $1 \leq i \leq s$, $\bar{w_i} \in \ker(f)$ and for all $s+1 \leq j \leq r$, $\bar{v_j} \in ker(f)$. 
So the kernel of the surjective map $R_{\rhob_x}(N) \to \FF[\epsilon]/(\epsilon^2)$ obtained by composing the natural map $R_{\rhob_x}(N) \to R_{\rhob_x}(N)/J$ with $f$ contains $w_i$ for all $1 \leq i \leq s$ and $u_j, v_j$ for all $s+1 \leq j \leq r$.
 Hence, $\rho$ is unramified at $\ell_i$ for all $1 \leq i \leq r$. Thus, $\rho$ factors through $G_{\QQ,N_0p}$.

 Since $\dim(H^1(G_{\QQ,N_0p},\chibar))=1$, we know, by Lemma~\ref{kisinjams}, that the natural map $R^{\ps}_{\rhob_0}(N_0) \to R_{\rhob_x}(N_0)$ is surjective.
Therefore, we conclude that $\rho \simeq \rhob_x$ as $\rho$ is a representation of $G_{\QQ,N_0p}$ with $\tr(\rho)=\tr(\rhob_0)$ and $\det(\rho) = \det(\rhob_0)$.
This proves the lemma.
\end{proof}

\begin{lem}
\label{finite}
Suppose $\chibar_2$ is unramified at $p$ and $\chibar_1\chibar_2 = \bar\psi\omega_p^{k_0-1}$, with $\bar\psi$ unramified at $p$ and $1 < k_0 < p$. 
Suppose $\chibar|_{G_{\QQ_p}} \neq 1,\omega_p^{-1}|_{G_{\QQ_p}}$.
Let $k>2$ be an integer such that $k \equiv k_0 \pmod{p-1}$ and let $\alpha, \beta, \delta_k \in R_{\rhob_x}(N)$ be the elements defined in Lemma~\ref{ord}.
Then $R_{\rhob_x}(N)/(\alpha,\beta,\delta_k,u_{s+1},\cdots,u_r)$ is a finite $W(\FF)$-algebra of Krull dimension at most $1$.
\end{lem}
\begin{proof}
Let $\tau := \rho^{\univ} \pmod{(\alpha,\beta,\delta_k,u_{s+1},\cdots,u_r)}$.
By Lemma~\ref{surjective}, the map $$F_k : R^{\ps}_{\rhob_0}(N) \to R_{\rhob_x}(N_0)/(\alpha,\beta,\delta_k,u_{s+1},\cdots,u_r)$$ induced by $(\tr(\tau),\det(\tau))$ is surjective.
By Lemma~\ref{ord}, we know that $\tau|_{G_{\QQ_p}} = \begin{pmatrix} \eta_1 & * \\ 0 & \eta_2 \end{pmatrix}$, where $\eta_2$ is an unramified character, and $\det(\tau)|_{I_p} = \chi_p^{k-1}$.
Therefore, using Pan's finiteness result (\cite[Theorem $5.1.2$]{P}) and the proof of Proposition~\ref{finprop} for Case $2$, we get that $F_k$ factors through a quotient of $R^{\ps}_{\rhob_0}(N)$ which is finite over $W(\FF)$. This proves the lemma.
\end{proof}

\begin{lem}
\label{dimension}
Suppose $\dim(H^1(G_{\QQ,N_0p},\chibar)) =1$, $\chibar_2$ is unramified at $p$, $\chibar|_{G_{\QQ_p}} \neq 1,\omega_p^{-1}|_{G_{\QQ_p}}$, $\chibar_1\chibar_2 = \bar\psi\omega_p^{k_0-1}$, with $\bar\psi$ unramified at $p$ and $1 < k_0 < p$. Let $\ell$ be a prime such that $\chibar|_{G_{\QQ_{\ell}}} = \omega_p^{-1}|_{G_{\QQ_{\ell}}}$ and $p \nmid \ell - 1$.
Then the ring $R_{\rhob_x}(N_0\ell)$ has Krull dimension $4$.
\end{lem}
\begin{proof}
Suppose $p \nmid \ell^2-1$. Since $\dim(H^1(G_{\QQ,N_0p},\chibar))=1$, Lemma~\ref{dimonelem} implies that $(\text{Cl}(K_{\chibar})/\text{Cl}(K_{\chibar})^p)[\omega_p\chibar^{-1}]=0$ and $\chibar|_{G_{\QQ_{q}}} \neq \omega_p|_{G_{\QQ_{q}}}$ for all primes $q \mid N_0p$.
From our assumptions on $\ell$, we know that $\chibar|_{G_{\QQ_{\ell}}} \neq \omega_p|_{G_{\QQ_{\ell}}}$.
Hence, Lemma~\ref{dimonelem} implies that $\dim(H^1(G_{\QQ,N_0\ell p},\chibar)) =1$. So Corollary~\ref{dimcor} implies that the Krull dimension of $R_{\rhob_x}(N_0\ell)$ is $4$.

Now suppose $p \mid \ell + 1$. 
From Lemma~\ref{boclem}, we know that $\rho|_{I_{\ell}}$ factors through the $\ZZ_p$-quotient of the tame inertia group at $\ell$ and we  can assume $\rho^{\univ}(g_\ell) = \begin{pmatrix} \phi_1 & 0 \\ 0  & \phi_2 \end{pmatrix}$ and $\rho^{\univ}(i_\ell) = \begin{pmatrix}  \sqrt{1+uv} & u \\ v & \sqrt{1+uv} \end{pmatrix}$.

Moreover, Lemma~\ref{boclem} also implies that there exist $f_{\ell}, h_{\ell} \in R_{\rhob_x}(N_0\ell)$ such that $\rho^{univ}(i_{\ell})^{\ell} = \begin{pmatrix} f_{\ell}  & uh_{\ell} \\ vh_{\ell} &  f_{\ell} \end{pmatrix}$.
Let $\phi = \phi_1\phi_2^{-1}$.
Therefore, the relation $\rho^{\univ}(g_{\ell}i_{\ell}g_{\ell}^{-1}) = \rho^{\univ}(i_{\ell})^{\ell}$ gives us $u(h_{\ell}-\phi) =0$ and $v(h_{\ell}-\phi^{-1})=0$.
% (see the proof of \cite[Lemma $4.9$]{Bo1} for more details).

Let $P$ be a minimal prime of $R_{\rhob_x}(N_0\ell)$. So $P$ contains either $u$ or $h_{\ell}-\phi$.
If $P$ contains $u$, then by Lemma~\ref{finite}, we see that $R_{\rhob_x}(N_0\ell)/(\alpha,\beta,\delta_k,P)$ is finite over $W(\FF)$ of Krull dimension at most $1$. So Theorem~\ref{krullthm} implies that the Krull dimension of $R_{\rhob_x}(N_0\ell)/P$ is at most $4$.

Suppose $P$ does not contain $u$. So $h_{\ell}-\phi \in P$. Now $P$ contains either $v$ or $h_{\ell} - \phi^{-1}$.
Suppose $v \in P$. 
From previous paragraph, we know that $R:=R_{\rhob_x}(N_0\ell)/(\alpha,\beta,\delta_k,u,P)$ is finite over $W(\FF)$ of Krull dimension at most $1$.
Suppose its Krull dimension is $1$.
Let $Q$ be a minimal prime of $R$. 
So $R/Q$ is isomorphic to a subring of $\overline{\QQ_p}$. 
Fix such an isomorphism to view $R/Q$ as a subring of $\overline{\QQ_p}$.
Let $\rho : G_{\QQ,N_0\ell p} \to \GL_2(R)$ be the representation obtained by composing $\rho^{\univ}$ with the natural surjective map $R_{\rhob_x}(N_0\ell) \to R$.
Let $\rho\Mod{Q} : G_{\QQ,Np} \to \GL_2(\overline{\QQ_p})$ be the representation obtained by composing $\rho$ with the surjective map $R \to R/Q$.

Since the images of $\alpha,\beta,\delta_k$ in $R$ are $0$, we see, from Lemma~\ref{ord}, that $\rho \Mod{Q}|_{G_{\QQ_p}} = \begin{pmatrix} \mu'_1 & *\\ 0 & \mu'_2 \end{pmatrix}$, where $\mu'_2$ is an unramified character.
Hence, by the main theorem of \cite{SW2} and Corollary~\ref{condlem}, $\rho\Mod{Q}$ is the $p$-adic Galois representation attached to a newform $f$ of weight $k$ whose tame level divides $N_0 \ell$.
Moreover, as the images of $u$ and $v$ in $R$ are $0$, we see that $\rho\Mod{Q}$ is unramified at $\ell$ and hence, $\ell$ does not divide the level of $f$.
Observe, from the definition of $h_{\ell}$, that $h_{\ell} \equiv \ell \pmod{uv}$. 
So the image of $\phi -\ell$ in $R/Q$ is $0$ and hence, $\rho_f|_{G_{\QQ_{\ell}}} = \chi \oplus \chi\chi_p$ for some unramified character $\chi : G_{\QQ_{\ell}} \to \overline{\QQ_p}^{\times}$. But Lemma~\ref{unramlem} gives a contradiction.
Therefore, the Krull dimension of $R$ is $0$.
So Theorem~\ref{krullthm} implies that the Krull dimension of $R_{\rhob_x}(N_0\ell)/P$ is at most $4$ in this case.
%But Lemma~\ref{unramlem} gives a contradiction. Hence, the Krull dimension of $R(N\ell)_{\rhob_b}/(\alpha,\beta,\delta_k,u,P)$ is $0$ which means, the Krull dimension of $R(N\ell)_{\rhob_b}/P$ is at most $4$.

Now suppose $h_{\ell} - \phi^{-1} \in P$. Since $h_{\ell} - \phi \in P$, this means that $\phi^2-1 \in P$ which implies that $\phi + 1 \in P$ (since the image of $\phi$ in $\FF$ is $-1$). 
Since $h_{\ell} \equiv \ell \pmod{uv}$, we have $\ell + 1 \in (u,P)$.
Since $R_{\rhob_x}(N_0\ell)/(\alpha,\beta,\delta_k,u,P)$ is finite over $W(\FF)$, it has Krull dimension $0$ (as $\ell+1 \in (u,P)$).
If the Krull dimension $R_{\rhob_x}(N_0\ell)/P$ is $\kappa$, then, by Theorem~\ref{krullthm}, the Krull dimension of $R_{\rhob_x}(N_0\ell)/(\alpha,\beta,\delta_k,u,P)$ is at least $\kappa-4$.
Therefore, we conclude that the Krull dimension of $R_{\rhob_x}(N_0\ell)/P$ is at most $4$ in this case as well.
This proves the lemma.
\end{proof}

%For every $s+1 \leq j \leq r$, let $f_{\ell_j}, h_{\ell_j} \in R_{\rhob_x}(N)$ be such that $(Q_j(\rho^{univ}(i_{\ell_j}))Q_j^{-1})^{\ell_j} = \begin{pmatrix} f_{\ell_j}  & u_jh_{\ell_j} \\ v_jh_{\ell_j} &  f_{\ell_j} \end{pmatrix}$.
%For every $s+1 \leq j \leq r$, let $\phi_j=\beta_j\alpha_j^{-1}$.

\subsection{Main result}
We will now prove a key proposition which describes the structure of $R_{\rhob_x}(N)$ in terms of the structure of $R_{\rhob_x}(N_0)$ (similar to Theorem~\ref{bocthm}). It will be used crucially in the proof of the main theorem of this section and we believe that it is also of independent interest. Note that we use the notation established in Lemma~\ref{boclem} and the first paragraph of \S\ref{ramsec}.
For every $s+1 \leq j \leq r$, let $\phi_j=\beta_j\alpha_j^{-1}$.

 %the previous paragraph and the notations developed in the discussion after Lemma~\ref{ord} in the statement of the proposition below.
\begin{prop}
\label{strprop}
Suppose we are in Setup~\ref{levelsetup} as above and suppose $\chibar_2$ is unramified at $p$.
Let $n = \dim(H^1(G_{\QQ,N_0p},\text{Ad}(\rhob_x)))$ and $r$ and $s$ be the integers defined in Points~\eqref{bhaag4} and \eqref{bhaag5} of  Setup~\ref{levelsetup}.
Then there exist a surjective morphism $$\mathcal{F} : R_0:= W(\FF)\llbracket X_1,\cdots,X_n,W_1,\cdots,W_s,U_{s+1},\cdots,U_r,V_{s+1},\cdots,V_r \rrbracket \to R_{\rhob_x}(N)$$ and a subset $\mathcal{S}_0:= \{f_1,\cdots, f_{n-3},g_1,\cdots,g_s,h_{s+1},\cdots,h_r,h'_{s+1},\cdots,h'_{r}\}$ of $R_0,$ such that 
\begin{enumerate}
\item $\mathcal{F}(W_i) = w_i$ for all $1 \leq i \leq s$,
\item $\mathcal{F}(g_i) = \alpha_{i}-\ell_i\beta_{i}$ for all $1 \leq i \leq s$,
\item $\mathcal{F}(U_j) = u_j$ and $\mathcal{F}(V_j) = v_j$ for all $ s+1 \leq j \leq r$,
\item $\mathcal{F}(h_j) = h_{\ell_j} - \phi_j$ and $\mathcal{F}(h'_j) = h_{\ell_j} - \phi_j^{-1}$ for all $s+1 \leq j \leq r$,
\item $I_0 := \ker(\mathcal{F})$ is the ideal generated by the set $$\{f_1,\cdots,f_{n-3}\} \cup \{W_ig_i\}_{1 \leq i \leq s} \cup \{U_jh_j\}_{s+1 \leq j \leq r} \cup \{V_jh'_j\}_{s+1 \leq j \leq r},$$
%$\{f_1,\cdots, f_{n-3},W_1g_1,\cdots,W_sg_s,U_{s+1}h_{s+1},\cdots,U_rh_r,V_{s+1}h'_{s+1},\cdots,V_rh'_{r}\}$,
\item %The kernel of the natural surjective map $R_{\rhob_x}(N) \to R_{\rhob_x}(N_0)$ is generated by the set $$\{w_i\}_{1 \leq i \leq s} \cup \{u_j\}_{s+1 \leq j \leq r} \cup \{v_j\}_{s+1 \leq j \leq r}$$ and as a result 
$R_{\rhob_x}(N_0) \simeq W(\FF)[[X_1,\cdots,X_n]]/(\overline{f_1},\cdots,\overline{f_{n-3}})$ where $$f_i \pmod{(W_1,\cdots,W_s,U_{s+1},\cdots,U_r,V_{s+1},\cdots,V_r)} = \overline{f_i}.$$
\end{enumerate}
\end{prop}
\begin{proof}
We will prove the proposition by using induction on $r$. For $r=1$, the proposition is just a combination of Theorem~\ref{bocthm} and Corollary~\ref{dimcor}..
Assume the proposition is true for $r=m$.
Now suppose $r=m+1$.

From Theorem~\ref{bocthm} and the induction hypothesis, it follows that there exists a surjective map 
$$\mathcal{F} : R_0 = W(\FF) \llbracket X_1,\cdots,X_n,W_1,\cdots,W_s,U_{s+1},\cdots,U_{m+1},V_{s+1},\cdots,V_{m+1} \rrbracket \to R_{\rhob_x}(N)$$
satisfying points $(1)-(4)$ and $(6)$ of the proposition.

%Note that for all $1 \leq i \leq s+1$, $R_{\rhob_x}(N_0\ell_i) \simeq R_{\rhob_x}(N)/J_i$, where $J_i$ is the ideal generated by the set $ \{w_{j'}\}_{1 \leq j' \leq s, j' \neq i} \cup \{u_j\}_{s+1 \leq j \leq r} \cup \{v_j\}_{s+1 \leq j \leq r}$.
%Similarly, for all $s+1 \leq j \leq r$, $R_{\rhob_x}(N_0\ell_j) \simeq R_{\rhob_x}(N)/J_j$, where $J_j$ is the ideal generated by the set $ \{w_i\}_{1 \leq i \leq s} \cup \{u_{j'}\}_{s+1 \leq j' \leq r, j' \neq j} \cup \{v_{j'}\}_{s+1 \leq j' \leq r, j' \neq j}$.

\textbf{Case $1$:} Suppose $s \neq 0$.
Then the induction hypothesis and Theorem~\ref{bocthm} imply that there exist $$f_1,\cdots,f_{n-3} \in R_0 \text{ and } F_2,\cdots,F_{s},G_{s+1},\cdots,G_{m+1},H_{s+1},\cdots,H_{m+1}\in (W_1)$$ such that $\ker(\mathcal{F})$ is generated by the set 
\begin{multline}
S := \{f_1,\cdots,f_{n-3}, W_1g_1\} \cup \{W_2g_2+F_2,\cdots, W_{s}g_{s}+F_{s}\}  \cup \\ \{U_{s+1}h_{s+1}+G_{s+1}, \cdots, U_{m+1}h_{m+1}+G_{m+1}\}  \cup \{V_{s+1}h'_{s+1}+H_{s+1},\cdots, V_{m+1}h'_{m+1}+H_{m+1}\}.
\end{multline}
Let $M_0$ be the maximal ideal of $R_0$ and $V$ be the $\FF$-vector subspace of $\ker(\mathcal{F})/M_0\ker(\mathcal{F})$ generated by the images of $f_1,\cdots, f_{n-3}, W_1g_1$.
Let $T$ be a subset of $\{f_1,\cdots,f_{n-3}, W_1g_1\}$ such that the images of the elements of $T$ in $\ker(\mathcal{F})/M_0\ker(\mathcal{F})$ form a basis of $V$.
So, by Nakayama's lemma, $\ker(\mathcal{F})$ is generated by the set $T \cup (S \setminus \{f_1,\cdots,f_{n-3}, W_1g_1\})$.

Note that $$\{W_2g_2,\cdots,W_{s}g_{s},U_{s+1}h_{s+1},\cdots,U_{m+1}h_{m+1},V_{s+1}h'_{s+1},\cdots,V_{m+1}h'_{m+1}\} \subset \ker(\mathcal{F})$$ which means $$\{F_2,\cdots,F_{s},G_{s+1},\cdots,G_{m+1},H_{s+1},\cdots,H_{m+1}\}\subset \ker(\mathcal{F}).$$
So $\ker(\mathcal{F})$ is generated by the set 
\begin{multline}
T \cup \{W_2g_2,\cdots,W_{s}g_{s}\}$ $\cup \{U_{s+1}h_{s+1},\cdots,U_{m+1}h_{m+1}\} \\ \cup \{V_{s+1}h'_{s+1},\cdots,V_{m+1}h'_{m+1}\} \cup \{ F_2,\cdots,F_{s},G_{s+1},\cdots,G_{m+1},H_{s+1},\cdots,H_{m+1}\}.
\end{multline}

Therefore, by Nakamyama's lemma, we can find a subset $T'$  of the set 
\begin{multline}
 \{W_2g_2,\cdots,W_{s}g_{s}\} \cup \{U_{s+1}h_{s+1},\cdots,U_{m+1}h_{m+1}\} \\ \cup \{V_{s+1}h'_{s+1},\cdots,V_{m+1}h'_{m+1}\}  \cup \{ F_2,\cdots,F_{s},G_{s+1},\cdots,G_{m+1},H_{s+1},\cdots,H_{m+1}\}
\end{multline} 
such that $T \cup T'$ is a set of minimal generators of $\ker(\mathcal{F})$.

As $T \cup (S \setminus \{f_1,\cdots,f_{n-3},W_1g_1\})$ is a set of generators of $\ker(\mathcal{F})$, it follows that $|T \cup T'|\leq |T \cup (S \setminus \{f_1,\cdots,f_{n-3},W_1g_1\})|$.
This means that $|T'| \leq 2m-(s-1)$.
If $W_ig_i \not\in T'$ for some $2 \leq i \leq s$, then $$T' \subset (W_1,\cdots,W_{i-1},W_{i+1},\cdots,W_{s},U_{s+1},\cdots,U_{m+1},V_{s+1},\cdots,V_{m+1})$$ as $$\{F_2,\cdots,F_{s},G_{s+1},\cdots,G_{m+1},H_{s+1},\cdots,H_{m+1}\} \subset (W_1).$$

This means $$\ker(\mathcal{F}) \subset (f_1,\cdots,f_{n-3},W_1,\cdots,W_{i-1},W_{i+1},\cdots,W_{s},U_{s+1},\cdots,U_{m+1},V_{s+1},\cdots,V_{m+1}).$$
Since the kernel of the map $R_{\rhob_x}(N) \to R_{\rhob_x}(N_0\ell_i)$ is the ideal generated by the set $$ \{w_{j'}\}_{1 \leq j' \leq s, j' \neq i} \cup \{u_j\}_{s+1 \leq j \leq m+1} \cup \{v_j\}_{s+1 \leq j \leq m+1},$$ we conclude, using Theorem~\ref{krullthm}, that the Krull dimension of $R_{\rhob_x}(N\ell_i)$ is at least $5$.
But Lemma~\ref{dimension} gives a contradiction to this. 
Therefore, $\{W_2g_2,\cdots,W_{s}g_{s}\} \subset T'$.

If $U_jh_j \not\in T'$ for some $s+1 \leq j \leq m+1$, then by the same logic as above we have $$\ker(\mathcal{F}) \subset (f_1,\cdots,f_{n-3},W_1,\cdots,W_{s},U_{s+1},\cdots,U_{j-1},U_{j+1},\cdots,U_{m+1},V_{s+1},\cdots,V_{m+1}).$$
Since the kernel of the map $R_{\rhob_x}(N) \to R_{\rhob_x}(N_0\ell_j)$ is the ideal generated by the set $$ \{w_i\}_{1 \leq i \leq s} \cup \{u_{j'}\}_{s+1 \leq j' \leq m+1, j' \neq j} \cup \{v_{j'}\}_{s+1 \leq j' \leq m+1, j' \neq j},$$ it follows that the kernel of the surjective map $R_0 \to R_{\rhob_x}(N_0\ell_j)$ obtained by composing the map $\mathcal{F}$ with $R_{\rhob_x}(N) \to R_{\rhob_x}(N_0\ell_j)$ is the ideal generated by the set 
\begin{multline}
\{f_1,\cdots,f_{n-3},W_1,\cdots,W_{s}\} \cup \{U_{s+1},\cdots,U_{j-1},U_{j+1},\cdots,U_{m+1}\} \\ \cup \{V_{s+1},\cdots,V_{j-1},V_{j+1},\cdots,V_{m+1},V_{j}h'_j\}.
\end{multline}
So we conclude, using Theorem~\ref{krullthm}, that the Krull dimension of $R_{\rhob_x}(N\ell_j)$ is at least $5$.
But Lemma~\ref{dimension} gives a contradiction to this. 

Therefore, $\{U_{s+1}h_{s+1},\cdots,U_{m+1}h_{m+1}\} \subset T'$. By repeating the same logic after replacing $U_j$'s with $V_j$'s everywhere, we get that $\{V_{s+1}h'_{s+1},\cdots,V_{m+1}h'_{m+1}\} \subset T'$.

 As $|T'| \leq 2m-(s-1)$, it follows that $$T'=\{W_2g_2,\cdots,W_{s}g_{s},U_{s+1}h_{s+1},\cdots,U_{m+1}h_{m+1},V_{s+1}h'_{s+1},\cdots,V_{m+1}h'_{m+1}\}.$$
Note that $\ker(\mathcal{F})$ is generated by $T \cup T'$ and $T \subset \{f_1,\cdots,f_{n-3},W_1g_1\}$.
Hence, we get that $\ker(\mathcal{F})$ is generated by the set $$\{f_1,\cdots,f_{n-3},W_1g_1,\cdots,W_{s}g_{s},U_{s+1}h_{s+1},\cdots,U_{m+1}h_{m+1},V_{s+1}h'_{s+1},\cdots,V_{m+1}h'_{m+1}\}$$ which proves the proposition in this case.

\textbf{Case $2$:} Suppose $s = 0$. The proof in this case is very similar to the proof in the previous case. So we only give a brief sketch.
From Theorem~\ref{bocthm} and induction hypothesis, it follows that there exist $G_{2},\cdots,G_{m+1},H_{2},\cdots H_{m+1} \in (U_{1},V_{1})$ such that $\ker(\mathcal{F})$ is generated by 
\begin{multline}
S:= \{f_1,\cdots,f_{n-3},U_{1}h_{1}, V_{1}h'_{1}\} \cup \{U_{2}h_{2}+G_{2},\cdots, U_{m+1}h_{m+1}+G_{m+1}\} \\ \cup \{V_{2}h'_{2}+H_{2},\cdots,V_{m+1}h'_{m+1}+H_{m+1}\}.
\end{multline}
%Let us call this set $S$.
Let $V$ be the $\FF$-vector subspace of $\ker(\mathcal{F})/M_0\ker(\mathcal{F})$ generated by the images of $f_1,\cdots, f_{n-3}, U_1h_1, V_1h'_1$.
Let $T$ be a subset of $\{f_1,\cdots,f_{n-3}, U_1h_1, V_1h'_1\}$ such that the images of the elements of $T$ in $\ker(\mathcal{F})/M_0\ker(\mathcal{F})$ form a basis of $V$.
So, by Nakayama's lemma, $\ker(\mathcal{F})$ is generated by the set $T \cup (S \setminus \{f_1,\cdots,f_{n-3}, U_1h_1, V_1h'_1\})$.

%Using Theorem~\ref{bocthm}, Lemma~\ref{dimension} and the argument used in the previous case, we conclude that the images of $f_1,\cdots,f_{n-3},U_{1}h_{1},V_{1}h'_{1}$ in $\ker(\mathcal{F})/M_0\ker(\mathcal{F})$ are linearly independent over $\FF$.
Note that $$\{U_{2}h_{2},\cdots, U_{m+1}h_{m+1},V_{2}h'_{2},\cdots,V_{m+1}h'_{m+1}\} \subset \ker(\mathcal{F}).$$
So $\ker(\mathcal{F})$ is generated by the set 
\begin{multline}
\{f_1,\cdots,f_{n-3},U_{1}h_{1}, V_{1}h'_{1}\} \cup \{U_{2}h_{2},\cdots, U_{m+1}h_{m+1},V_{2}h'_{2},\cdots,V_{m+1}h'_{m+1}\} \\ \cup \{G_{2},\cdots,G_{m+1},H_{2},\cdots,H_{m+1}\}.
\end{multline}

So by Nakayama's lemma, we can find a subset $T'$ of $$\{U_{2}h_{2},\cdots, U_{m+1}h_{m+1},V_{2}h'_{2},\cdots,V_{m+1}h'_{m+1}\} \cup \{G_{2},\cdots,G_{m+1},H_{2},\cdots,H_{m+1}\}$$ which will extend the set $T$ to a set of minimal generators of $\ker(\mathcal{F})$.

Since $T \cup (S \setminus \{f_1,\cdots,f_{n-3}, U_1h_1, V_1h'_1\})$ is a set of generators of $\ker(\mathcal{F})$, we conclude that $|T'| \leq 2m$.
Using the same argument that we used in the previous case, we conclude that $U_jh_j, V_jh'_j \in T'$ for all $2 \leq j \leq m+1$.
Therefore, $$T' = \{U_{2}h_{2},\cdots, U_{m+1}h_{m+1},V_{2}h'_{2},\cdots,V_{m+1}h'_{m+1}\}.$$
As $\ker(\mathcal{F})$ is generated by the set $T \cup T'$ and $T \subset \{f_1,\cdots,f_{n-3},U_{1}h_{1},V_1h'_1\}$, we get the proposition in this case.
\end{proof}

We are now ready to prove the main theorem of this section. Recall that we have defined $N'_0$ in Point~\eqref{bhaag2} of Setup~\ref{levelsetup}.
Thus, if $f$ is a newform of level $M$ such that $\rho_f$ lifts $\rhob_0$ and $\ell$ is a prime dividing $N_0$, then $v_{\ell}(M) = v_{\ell}(N'_0)$ (see also Corollary~\ref{condlem}).
\begin{thm}
\label{levelthm}
%Let $\rhob_0$ be a representation as in Setup~\ref{levelsetup} with tame Artin conductor $N_0$. 
Let $\ell_1,\cdots,\ell_r$ be primes satisfying the hypothesis of Point~\eqref{bhaag4} of Setup~\ref{levelsetup}. 
%Let $\ell_1,\cdots,\ell_r$ be primes such that $\ell_i \nmid N_0 p$, $\chibar|_{G_{\QQ_{\ell_i}}} = \omega_p^{-1}|_{G_{\QQ_{\ell_i}}}$ and $p \nmid \ell_i-1$ for all $1 \leq i \leq r$. 
Suppose $\chibar|_{G_{\QQ_p}} \neq 1, \omega_p^{-1}|_{G_{\QQ_p}}$, $\dim(H^1(G_{\QQ,N_0p},\chibar)) =1$, $\chibar_2$ is unramified at $p$ and $\chibar_1\chibar_2$ is ramified at $p$. 
Suppose $\chibar_1\chibar_2 = \bar\psi\omega_p^{k_0-1}$, with $\bar\psi$ unramified at $p$ and $1 < k_0 < p$. 
Then for every integer $k > 2$ such that $k \equiv k_0 \Mod{p-1}$, there exists an eigenform $f$ of level $N'_0\prod_{i=1}^{r}\ell_i$ and weight $k$ such that $\rho_f$ lifts $\rhob_0$ and $f$ is new at $\ell_i$ for every $1 \leq i \leq r$.
\end{thm}
\begin{proof}
Let $N=N'_0\prod_{i=1}^{s}\ell_i\prod_{j=s+1}^{r}\ell_j^2$ and $k > 2$ be an integer such that $k \equiv k_0 \Mod{p-1}$.
Let $I_k$ be the ideal of $R_{\rhob_x}(N)$ generated by the set $$S:=\{\alpha,\beta,\delta_k,\alpha_1-\ell_1\beta_1,\cdots, \alpha_s-\ell_s\beta_s,u_{s+1},\cdots,u_r,h_{\ell_{s+1}}-\phi_{s+1}^{-1},\cdots, h_{\ell_r}-\phi_r^{-1}\}.$$
Here $\alpha,\beta,\delta_k \in R_{\rhob_x}(N)$ are the elements defined in Lemma~\ref{ord} and all the other elements of $S$ are from Proposition~\ref{strprop}.

Let $R_k(N):= R_{\rhob_x}(N)/I_k$ and $\rho : G_{\QQ,Np} \to \GL_2(R_k(N))$ be the representation obtained by composing $\rho^{\univ}$ with the natural surjective map $R_{\rhob_x}(N) \to R_k(N)$.
From Lemma~\ref{finite}, we get that $R_k(N)$ is finite over $W(\FF)$ of Krull dimension at most $1$.
%By combining Lemma~\ref{ord}, Lemma~\ref{surjective}, Pan's finiteness result (\cite[Theorem $5.1.2$]{P}) and the proof of Proposition~\ref{finprop} for Case $2$, we get that $R_k(N)$ is finite over $W(\FF)$ of Krull dimension at most $1$ (see the proof of Lemma~\ref{dimension}).
On the other hand, combining Proposition~\ref{strprop} and Theorem~\ref{krullthm}, we get that the Krull dimension of $R_k(N)$ is at least $1$.
Therefore, we get that $R_k(N)$ is a finite $W(\FF)$-algebra of Krull dimension $1$.

Let $Q$ be a minimal prime of $R_k(N)$. From previous paragraph, we conclude that $R_k(N)/Q$ is isomorphic to a subring of $\overline{\QQ_p}$. Using this isomorphism, we view $\rho\Mod{Q}$ as a representation over $\overline{\QQ_p}$.
By definition of $R_k(N)$, $\rho \Mod{Q} : G_{\QQ,Np} \to \GL_2(\overline{\QQ_p})$ is a $p$-ordinary representation such that $\det(\rho\Mod{Q}) = \epsilon\chi_p^{k-1}$, where $\epsilon$ is a character unramified at $p$. In other words, $\rho \Mod{Q}|_{G_{\QQ_p}} = \begin{pmatrix} \eta'_1\chi_p^{k-1} & *\\ 0 & \eta'_2 \end{pmatrix}$, where $\eta'_1$ and $\eta'_2$ are unramified characters of $G_{\QQ_p}$. 
Hence, using the modularity lifting theorem of Skinner-Wiles (main theorem of \cite{SW2}), we get that $\rho\Mod{Q}$ is the $p$-adic Galois representation attached to a newform $f$ of weight $k$.
Moreover, Proposition~\ref{carayolprop} implies that the tame level of $f$ divides $N$.
%Hence, using main theorem of \cite{SW2}, \cite[Proposition $2$]{Ca} and the proof of Lemma~\ref{condlem}, we get that $\rho\Mod{Q}$ is the $p$-adic Galois representation attached to a newform $f$ of weight $k$ whose tame level divides $N$.

Since $\rho_f = \rho \Mod{Q}$ is $p$-ordinary, it follows, from \cite[Proposition $3.6$]{Mo}, that $f$ is $p$-ordinary. From previous paragraph, we know that the nebentypus of $f$ is unramified at $p$. Hence, it follows, from Lemma~\ref{ordlem}, that $p$ does not divide the level of $f$. Thus, $f$ is an eigenform of level $N$ and weight $k$ such that $\rho_f$ lifts $\rhob_0$.

%Suppose $\rho_f$ is unramified at $\ell_i$ i.e. $\ell_i$ does not divide the level of $f$. Note that we have $a_i - \ell_ib_i \in Q$. Let $\bar{a_i}$ and $\bar{b_i}$ be the images of $a_i$ and $b_i$ in $R_k(N)/Q)$, respectively. So, $\bar{a_i}=\ell_i\bar{b_i}$ and $\bar{a_i}\bar{b_i} = \ell_i\bar{b_i}^2 = \Psi(\text{Frob}_{\ell_i})\ell_i^{k-1}$. Hence, $T_{\ell_i}$-eigenvalue of $f$ is $\bar{a_i}+\bar{b_i}= \zeta_0(\ell_i+1)\ell_i^{k-2/2}$ where $\zeta_0$ is a squareroot of $\Psi(\text{Frob}_{\ell_i})$. So $\zeta_0$ is an $n$-th root of unity for some $n$. 
%But Weil bound on the $T_{\ell_i}$ eigenvalue of $f$ (given by the Ramanujan-Petersson conjecture) implies that $(\ell_i+1)\ell_i^{k-2/2} \leq 2 \ell_i^{k-1/2}$ which is not possible as $\ell_i >1$ (see REF). This gives us a contradiction. 
Note that $\alpha_i - \ell_i\beta_i \in Q$ for all $1 \leq i \leq r$ (since $h_{\ell_j} \equiv \ell_j \pmod{u_jv_j}$). So it follows that for all $1 \leq i \leq r$, the semi-simplification of $\rho \Mod{Q}|_{G_{\QQ_{\ell_i}}}$ is $\chi_i \oplus \chi_i\chi_p$ for some character $\chi_i$. Note that the order of $\chi_i$ is not finite as $k >2$. This means that $f$ is cuspidal.
Hence, by Lemma~\ref{unramlem}, we see that $\ell_i$ divides the level of $f$ and $\rho \Mod{Q}$ is ramified at $\ell_i$ for all $1 \leq i \leq r$.
Therefore, we have $\rho \Mod{Q}|_{G_{\QQ_{\ell_i}}} \simeq \begin{pmatrix} \chi_i\chi_p & * \\ 0 & \chi_i\end{pmatrix}$ with $* \neq 0$ for all $1 \leq i \leq r$.
As $\chi_i$ is an unramified character of $G_{\QQ_{\ell_i}}$, this implies that $\ell_i$ divides the level of $f$ exactly once and $f$ is new at $\ell_i$ for all $1 \leq i \leq r$. This finishes the proof of the theorem.
\end{proof}

We now record some immediate corollaries of Theorem~\ref{levelthm}:
\begin{cor}
\label{levelcor}
Suppose the hypotheses of Theorem~\ref{levelthm} hold. Let $\ell_1,\cdots,\ell_r$ be the primes as in Theorem~\ref{levelthm} and $k_0$ be as in Theorem~\ref{levelthm}. Moreover, assume that $p \nmid \phi(N_0)$. Then for every integer $k >2$ such that $k \equiv k_0 \Mod{p-1}$, there exists a newform $f$ of level $N_0\prod_{i=1}^{r}\ell_i$ and weight $k$ such that $\rho_f$ lifts $\rhob_0$.
\end{cor}
\begin{proof}
Since we are assuming $p \nmid \phi(N_0)$, we get that $N_0 = N'_0$. The corollary now follows by combining Theorem~\ref{levelthm} with the facts that $N_0$ is the tame Artin conductor of $\rhob_0$ and none of the $\ell_i$'s divide $N_0$.
\end{proof}

We now recall a conjecture of Billerey--Menares (\cite[Conjecture $3.2$]{BM}).

\begin{conjecture}[Billerey-Menares]
\label{bmconj}
Let $p$ be a prime and $k_0 \geq 4$ be an even integer such that $ k_0 < p-1$. Let $\ell_1,\cdots,\ell_r$ be primes such that $p \neq \ell_i$ for all $1 \leq i \leq r$.
Then there exists a newform $f$ of level $\Gamma_0(\prod_{i=1}^{r}\ell_i)$ and weight $k_0$ such that $\rho_f$ lifts $1 \oplus \omega_p^{k_0-1}$ if and only if at least one of the following conditions hold:
\begin{enumerate} 
\item $p \mid (\ell_i^{k_0}-1)(\ell_i^{k_0-2}-1)$ for all $1 \leq i \leq r$ and there exists a $1 \leq j \leq r$ such that $p \mid \ell_j^{k_0}-1$.
\item $p \mid (\ell_i^{k_0-2}-1)$ for all $1 \leq i \leq r$ and $p$ divides the numerator of $\frac{B_{k_0}}{k_0}$, where $B_{k_0}$ is the $k_0$-th Bernoulli number.
\end{enumerate}
\end{conjecture}

The following corollary partially answers Conjecture~\ref{bmconj} in some cases.
\begin{cor}
Let $k_0$ be an even integer such that $2 < k_0 < p-1$. Let $\ell_1,\cdots,\ell_r$ be primes such that $p \nmid \ell_i-1$ and $p \mid \ell_i^{k_0}-1$ for all $1 \leq i \leq r$. Let $k$ be an integer such that $k \equiv k_0 \Mod{p-1}$. Suppose one of the following conditions hold:
\begin{enumerate}
\item $p \nmid B_{p+1-k_0}$.
\item $p$ is a regular prime.
\item Vandiver's conjecture holds for $p$.
\item $p > 5$ and $k_0=4,6$.
\item $p>3$, $p \equiv 3 \pmod{4}$ and $k_0=\dfrac{p+1}{2}$.
\end{enumerate}
% $(\text{Cl}(\QQ(\zeta_p))/\text{Cl}(\QQ(\zeta_p))^p) [\omega_p^{2-k_0}] =0$.
Then there exists a newform $f$ of level $\Gamma_0(\ell_1\cdots\ell_r)$ and weight $k$ such that $\rho_f$ lifts $1 \oplus \omega_p^{k_0-1}$.
\end{cor}
\begin{proof}
From the discussion just before Example~\eqref{eg1}, it follows that $$(\text{Cl}(\QQ(\zeta_p))/\text{Cl}(\QQ(\zeta_p))^p) [\omega_p^{2-k_0}] =0$$ if one of the conditions given above hold.
Therefore, by Lemma~\ref{dimonelem}, we get that $\dim(H^1(G_{\QQ,p},\omega_p^{k_0-1})) = 1.$
%Since $p$ is a regular prime, \cite[Lemma $21$]{BK} implies that $\dim(H^1(G_{\QQ,p},\omega_p^{k_0-1})) = 1.$ 
Note that $\omega_p^{k_0-1}|_{G_{\QQ_{\ell_i}}} = \omega_p^{-1}|_{G_{\QQ_{\ell_i}}}$ if and only if $\ell_i^{k_0} \equiv 1 \Mod{p}$. The corollary now follows from Corollary~\ref{levelcor} and the fact that $p \nmid \ell_i-1$ for all $1 \leq i \leq r$.
\end{proof}

%\textbf{Competing interests:} The author declares none.

\end{document}